\documentclass[10pt]{article}
\usepackage{graphicx} 
\usepackage{algorithm}
\usepackage{algpseudocode}

\usepackage[margin=1in]{geometry}
\usepackage{booktabs}
\usepackage{siunitx}
\usepackage{caption}

\algrenewcommand\algorithmicthen{} 
\usepackage[hidelinks]{hyperref}
\usepackage{amsmath,amsfonts,amssymb, amsthm, mathtools}
\usepackage{cite}
\usepackage{enumitem}
\newtheorem{theorem}{Theorem}[section]
\newtheorem{corollary}{Corollary}[theorem]
\newtheorem{lemma}[theorem]{Lemma}

\newtheorem{assumption}{Assumption}
\DeclareMathOperator*{\argmin}{argmin}
\DeclarePairedDelimiter{\ceil}{\lceil}{\rceil}
\DeclarePairedDelimiter{\floor}{\lfloor}{\rfloor}
\usepackage{booktabs,tabularx,makecell,array,xcolor}
\newcolumntype{L}[1]{>{\raggedright\arraybackslash}p{#1}}
\newcolumntype{C}[1]{>{\centering\arraybackslash}p{#1}}
\newcolumntype{R}[1]{>{\raggedleft\arraybackslash}p{#1}}
\usepackage{amsmath}

\usepackage{graphicx}
\usepackage{subcaption}
\usepackage{amsmath}
\usepackage{enumitem}
\usepackage{caption}
\setcellgapes{1pt}\makegapedcells    

\renewcommand{\arraystretch}{1.10}   

\usepackage[hidelinks]{hyperref} 

\let\Lo\L
\newcommand{\mo}{}

\newcommand{\dist}{\operatorname{dist}}
\renewcommand{\L}{\mathcal{L}}

\newcommand{\cL}{c_{\L}}
\newcommand{\ymax}{\bar{y}_{\max}}
\newcommand{\dom}{\operatorname{dom}}

\title{An Adaptive Augmented Lagrangian Method for Deterministic and Stochastic Nonconvex Optimization}
\author{Tianzhu Liu\thanks{Department of Statistics \& Operations Research,
UNC–Chapel Hill. Email: \href{tianzhu@unc.edu}{tianzhu@unc.edu}}
\and
Michael J. O'Neill\thanks{Department of Statistics \& Operations Research,
UNC–Chapel Hill. Email:
\href{mikeoneill@unc.edu}{mikeoneill@unc.edu}}}
\date{\today}

\usepackage{enumitem}
\begin{document}

\maketitle

\begin{abstract}
We present an inexact Augmented Lagrangian algorithm for solving nonlinear, non‑convex optimization problems. Unlike most recently proposed Augmented Lagrangian methods with worst-case complexity guarantees, we utilize adaptive penalty parameter updates and full dual stepsizes. We show that the method matches the best known worst-case complexity results for Augmented Lagrangian methods (up to logarithmic factors) when both the function and constraints are deterministic, when the function is stochastic and the constraints are deterministic, and when both are stochastic. Experiments on CUTEst test problems confirm the practical advantages of the proposed approach over Augmented Lagrangian methods with non-adaptive penalty parameters and/or short dual step sizes in the deterministic setting. Numerical results on stochastic constrained optimization problems in machine learning also confirm these findings.
\end{abstract}

\section{Introduction}
We consider the following constrained optimization problem: 
\begin{equation}
    \min_{x\in \mathbb{R}^n} f(x) + g(x) \quad \text{s.t.} \quad c(x) = 0, \label{eq:probdef}
\end{equation}
where $f : \mathbb{R}^n \rightarrow \mathbb{R}$ is a continuously-differentiable non-convex function, and $g : \mathbb{R}^n \rightarrow \mathbb{R}\cup \{\infty\}$ is convex but possibly non-differentiable. In addition, the constraints $c: \mathbb{R}^n \rightarrow \mathbb{R}^m$ are smooth and continuously differentiable. \mo{We also consider cases in which the problem may be stochastic i.e. $f(x) = \mathbb{E}_{\xi}[F_0(x;\xi)]$ and/or $c(x) = \mathbb{E}_{\xi}[C(x;\xi)]$.}

Applications of deterministic nonlinearly constrained optimization are widespread, including optimal control, PDE-constrained optimization \cite{gnegel2021solution}, network flow, and resource allocation \cite{doostmohammadian20221st}. In the stochastic setting, applications include physics-informed neural networks\cite{mowlavi2023optimal, lu2021physics}, Neyman–Pearson classification\cite{Li2023, tong2016survey}, and many others.

\mo{A popular method for solving constrained optimization problems of the form \eqref{eq:probdef} is the Augmented Lagrangian method, which adds a penalty term to the Lagrangian function:
\begin{equation}
\label{augmented_lagrangian_function}
    \mathcal{L}_\beta(x,y) =  f(x) + c(x)^Ty + \frac{\beta}{2} \|c(x)\|^2
\end{equation}
where $y \in \mathbb{R}^m$ are Lagrange multipliers and $\beta>0$ is the penalty parameter. Then, a basic Augmented Lagrangian method consists of alternating primal and dual steps of the form:
\begin{equation} \label{eq:basicAL}
    x_{k+1} \in \underset{x \in \mathbb{R}^n}{\argmin} \{ \mathcal{L}_{\beta_k}(x, y_k) + g(x)\}, \quad \quad y_{k+1} = y_k + \beta_k c(x_{k+1}).
\end{equation}
In recent years, many works have proposed Augmented Lagrangian algorithms with worst-case complexity guarantees \cite{Sahin2019,Li2020,Li2023,Xu2021,Xie2019,Birgin2020,grapiglia2021complexity}. For our purposes,} we call a primal-dual pair $(x,y)$ an $\epsilon$-approximate stationary solution if
\begin{equation} \label{eq:approximate_optimality_condition}
    \dist \left( -\nabla_x \mathcal{L}(x, y), \partial g(x) \right) + \|c(x)\| \leq \epsilon,
\end{equation}
where $\mathcal{L}(x,y)$ denotes the standard Lagrangian function.

\mo{Most Augmented Lagrangian methods with complexity guarantees rely upon two major methodological changes when compared with standard Augmented Lagrangian routines: increasing the penalty parameter at each iteration via a fixed schedule and decoupling the dual stepsize and the penalty parameter. The penalty parameter schedule is chosen to guarantee that no more than $\mathcal{O}(\log(\epsilon^{-1}))$ outer iterations are required to find a point satisfying \eqref{eq:approximate_optimality_condition}. However, this procedure may increase the penalty parameter far too quickly, causing the subproblems to become highly ill-conditioned and difficult to solve. We empirically validate this deficiency in Section \ref{sec:detexperiments}, where our method which adaptively adjusts the penalty parameter significantly outperforms Augmented Lagrangian methods with fixed schedules.}

\mo{The second change, that of a decoupled dual stepsize, is arguably the one of much greater theoretical consequence. This change is made exclusively to use a sufficiently short dual stepsize to provide a global upper bound on the norm of the dual variables \cite{Sahin2019}.\footnote{Some papers \cite{Li2023} allow the upper bound on the norm of the dual variables to increase at a $\mathcal{O}(k)$ rate. Our method and analysis can be easily extended to this case, however we stick with a fixed upper bound for clarity of presentation.} In contrast, popular Augmented Lagrangian algorithms simply impose a bound on the dual variables by projecting them onto a norm ball of a user-specified radius \cite{Andreani2008}. We adopt this projection based approach, which is sufficient for deriving complexity guarantees. One driving factor for preferring the penalty parameter as the dual stepsize as opposed to a short dual stepsize comes from the derivation of the Augmented Lagrangian method. It is well known that the Augmented Lagrangian method is equivalent to applying the proximal point method to the dual problem \cite{rockafellar1976augmented}. However, once the penalty parameter is decoupled from the dual stepsize, this equivalence no longer holds. In other words, any Augmented Lagrangian method that uses an alternative dual stepsize is \textit{not proximal point on the dual.}\footnote{Of course, the method here only satisfies this interpretation on iterations on which the projection is not triggered.}}

\mo{Taken together, these changes suggest that most Augmented Lagrangian methods with complexity guarantees are much closer in spirit and theory to pure penalty approaches which suffer from severely ill-conditioned subproblems. Indeed, the standard complexity analysis relies on simply bounding away the contribution of the dual variables and showing convergence via a sufficiently large penalty parameter, exactly as is done in the analysis of pure penalty methods \cite{Sahin2019, lin2022complexity}. Some exceptions exist; however, these results rely on (near) initial feasibility conditions \cite{Xie2019}. One major exception is the analysis of ALGENCAN, which obtains a complexity result under the condition that the adaptively updated penalty parameter remains bounded \cite{Birgin2020}. Our proposed algorithm is heavily inspired by ALGENCAN and we extend their analysis to prove that a complexity result always holds; either by their analysis when the penalty parameter remains below a threshold dependent on our desired accuracy or by an analysis akin to \cite{Sahin2019} when the penalty parameter grows beyond this threshold. In this way, we combine these techniques to prove that our adaptive method obtains the same worst-case complexity as that of the non-adaptive methods, regardless of the behavior of the penalty parameter. }

\textbf{Contributions}:
We \mo{propose an adaptive  Augmented Lagrangian method which is closely based on ALGENCAN \cite{Andreani2008}}. In the deterministic setting, under standard smoothness and a mild regularity condition, we establish an iteration bound that matches the best known worst-case complexity (up to logarithmic factors), \mo{regardless of the behavior of the penalty parameter}. For stochastic objectives and/or stochastic constraints, we \mo{combine our adaptive Augmented Lagrangian method with a simple stochastic inner loop and provide complexity results that hold with high probability} under natural oracle assumptions. \mo{We theoretically prove that the cost of this adaptivity is only logarithmic in the relevant factors and thus our methods are less reliant on accurate parameter settings when compared with prior work.} Extensive experiments on PyCUTEst and a Neyman–Pearson classification task corroborate the advantages of the proposed update scheme over short-step and non-adaptive baselines, \mo{as well as the value of an adaptive inner loop in stochastic settings.}

\section{Related Work}
One classical approach that can effectively solve problem $\left(\ref{eq:probdef}\right)$ is the quadratic penalty method \cite{kong2019complexity, lan2013iteration, lin2022complexity, luenberger1984linear, Nocedal2006} , which iteratively solves \mo{a sequence of unconstrained subproblems involving the original} objective plus a quadratic term for the constraint violation, \mo{scaled by a chosen penalty parameter}. \mo{In \cite{lan2013iteration}, Lan and Monteiro established a worst-case complexity result of roughly $\tilde{\mathcal{O}}(\epsilon^{-1})$ for a first-order quadratic penalty method for solving convex problems.} \mo{In the nonconvex setting,} Kong et al. \cite{kong2019complexity} applied an inexact proximal point method to solve the subproblems and obtained an approximate stationary point in at most $O(\epsilon^{-3})$ \mo{gradient evaluations}. This result was improved by Lin et al. \cite{lin2022complexity} to $O(\epsilon^{-\frac{5}{2}})$ \mo{when the constraints are convex and Slater's condition holds}.

While simple, \mo{pure penalty methods commonly suffer from highly ill-conditioned subproblems, as the penalty parameter must become very large to drive the algorithm towards feasibility} \cite{Nocedal2006}. The Augmented Lagrangian Method, originally introduced by Hestenes \cite{Hestenes1969} and Powell \cite{Powell1969}, mitigates this difficulty \mo{through the addition of Lagrange multipliers to help drive the method towards feasible solutions}. Hong \cite{hong2016decomposing} proposed a proximal primal dual algorithm (Prox-PDA) for smooth, nonconvex problems with linear constraints and showed the convergence to first-order  stationary points with a sublinear rate. Lan and Monteiro \cite{Lan2016} investigated the iteration complexity of a first-order ALM \mo{in the convex setting and proved} a complexity result of $O(\epsilon^{-1})$ \mo{gradient evaluations}. Similarly, Xu \cite{Xu2021} analyzed the convergence properties of a first-order ALM \mo{and proved} an improved complexity result of $O(\epsilon^{-\frac{1}{2}}|\log\epsilon|)$ \mo{when the problem is strongly convex}. The same rate was also obtained by Li et al. \cite{Li2020} \mo{when strong convexity holds}, and, \mo{in the nonconvex setting}, a complexity result \mo{of} $O(\epsilon^{-\frac{5}{2}}|\log\epsilon|)$ \mo{was proven under} Slater’s condition.

In nonconvex settings, both stochastic and deterministic methods \mo{have been} studied extensively. Table \ref{table-comparison} summarizes a \mo{number of relevant} works \mo{with} their assumptions and complexity results. \mo{In \cite{Birgin2020}, the authors study the complexity of the ALGENCAN method \cite{Andreani2008} and demonstrate an outer-iteration worst-case complexity result of $\tilde{\mathcal{O}}(\log(\epsilon^{-1}))$ under the assumption that the penalty parameter is bounded above by a constant. Our proposed algorithm is largely inspired by ALGENCAN and we demonstrate that one can obtain an outer iteration complexity of $\tilde{\mathcal{O}}(\epsilon^{-1})$ when the penalty parameter does not obey such a bound.} Sahin et al. \cite{Sahin2019} proposed an ALM framework with a \mo{short dual stepsize, which guarantees uniformly bounded dual variables}. They obtained a complexity result of $\tilde{O}(\epsilon^{-4})$ \mo{to first-order stationary points}, and $\tilde{O}(\epsilon^{-5})$ \mo{to second-order stationary points}. Li et al. \cite{Li2021} later improved this result by proposing an inexact proximal-point method (IPPM) and obtained a \mo{complexity of $\tilde{O}(\epsilon^{-\frac{5}{2}})$ for nonconvex objectives} with \mo{linear} equality constraints and $\tilde{O}(\epsilon^{-3})$ when \mo{the constraints are nonconvex}. \mo{Similarly to \cite{Sahin2019}, a short dual step size is imposed to bound the dual variables}. Xie and Wright \cite{Xie2019} analyzed a proximal ALM for nonconvex equality constrained problems and proved a worst-case outer-iteration complexity of \mo{$\tilde{O}(\epsilon^{-(2-\eta)})$ for $\eta \in [0,2]$, where $
\eta$ captures the proximal term and penalty growth trade-off}.

When the objective function is stochastic and the constraints are deterministic, Lu et al. \cite{Lu2024} proposed a single loop variance reduced scheme and obtained a complexity of $\tilde{O}(\epsilon^{-4})$ under \mo{a mean squared smoothness condition}. Bollapragada et al. \cite{Bollapragada2023} introduced an adaptive sampling technique and showed a \mo{sample complexity} of $O(\epsilon^{-2})$ for specific choice of penalty parameter and when the objective function is convex \mo{and the constraints are linear equalities}. Li et al. \cite{Li2023} studied stochastic objective function with stochastic constraints using a momentum-based variance-reduced method and \mo{proved} a convergence result of $O(\epsilon^{-5})$ for nonconvex functions.

For our proposed method, we can show that under the same assumption, the convergence property depends on the inner-solver and we can still match the best-known complexity result for both deterministic and stochastic setting.

\begin{table}[h!]

\centering
\small
\begin{tabularx}{\textwidth}{@{}L{0.20\textwidth} C{0.17\textwidth} C{0.17\textwidth} C{0.16\textwidth} R{0.20\textwidth}@{}}
\toprule
\textbf{Method} & \textbf{Objective} & \textbf{Constraint} & \textbf{Assumptions} & \textbf{Complexity} \\
\midrule

{iALM \cite{Sahin2019}}
  & Deterministic & Deterministic & 1, 2 & $\widetilde{\mathcal{O}}(\epsilon^{-4})$ \\
\midrule

{iPPM \cite{Li2021}}
  & Deterministic & Deterministic & 1, 2 & $\widetilde{\mathcal{O}}(\epsilon^{-3})$ \\
\midrule

{ALGENCAN \cite{Birgin2020}}
  & Deterministic & Deterministic & 1, 2
  & \makecell[r]{Outer loop\\ $\mathcal{O}(\log(\epsilon^{-1}))$} \\
\midrule

\makecell[l]{Stoc-iALM \cite{Li2023} }
  & Stochastic & Stochastic & 1, 2, 3, 4, 5 & $\widetilde{\mathcal{O}}(\epsilon^{-5})$ \\
\addlinespace[3pt] 

\midrule

\makecell[l]{\textcolor{red}{\textbf{Proposed}}}
  & \makecell[c]{Deterministic}
  & \makecell[c]{Deterministic}
  & \makecell[c]{1, 2}
  & \makecell[r]{$\widetilde{\mathcal{O}}(\epsilon^{-3})$} \\

\makecell[l]{\\ \\ }
  & \makecell[c]{Stochastic \\Stochastic}
  & \makecell[c]{Deterministic \\Stochastic}
  & \makecell[c]{1, 2, 3, 4 \\1, 2, 3, 4}
  & \makecell[r]{$\widetilde{\mathcal{O}}(\epsilon^{-6})$\\$\widetilde{\mathcal{O}}(\epsilon^{-7})$} \\

\makecell[l]{\\}
  & \makecell[c]{Stochastic\\Stochastic}
  & \makecell[c]{Deterministic\\Stochastic}
  & \makecell[c]{1, 2, 3, 4, 5\\1, 2, 3, 4, 5}
  & \makecell[r]{$\widetilde{\mathcal{O}}(\epsilon^{-4})$\\$\widetilde{\mathcal{O}}(\epsilon^{-5})$} \\

\bottomrule

\end{tabularx}
\caption{Comparison of major related methods under different assumptions and settings. Assumptions are listed and discussed further in the following section.}
\label{table-comparison}
\end{table}

\section{Preliminaries}

\subsection{Notation}
We use $\|.\|$ to denote the Euclidean norm. For the convex function $g : \mathbb{R}^n \rightarrow \mathbb{R}\cup \{\infty\}$, the subdifferential set at a point $x$ is denoted by $\partial g(x)$. The distance function from a vector $x$ to a set $\mathcal{X}$ is denoted by $\text{dist}(x,\mathcal{X}) = \displaystyle \min_{z\in \mathcal{X}} \|x-z\|$. For an operator $c : \mathbb{R}^n \rightarrow \mathbb{R}^m$, $J(x) \in \mathbb{R}^{m\times n}$ denotes the Jacobian of $c$, where the $i^\text{th}$ row of $J(x)$ is the vector $\nabla c_i(x)\in \mathbb{R}^n$. For integers $k_1 \leq k_2$, we use the notation $[k_1 : k_2] = \{k_1, ..., k_2\}$.

\subsection{Assumptions}
Next, we present our assumptions. The first two assumptions will be used throughout the paper while the latter three are specific to the stochastic setting. 

\begin{assumption}[Smoothness and Boundedness] \label{assum:smoothness}
We assume that $f : \mathbb{R}^n \rightarrow \mathbb{R}$ and $c : \mathbb{R}^n \rightarrow \mathbb{R}^m$,
are smooth and bounded, i.e., there exists $L_f, L_J, L_c \geq 0$ such that
\begin{equation*}
    \|\nabla f(x) - \nabla f(v)\| \leq L_f \|x-v\|, \quad \|J(x) - J(v)\| \leq L_J \|x-v\|,\quad \|c(x) - c(v)\| \leq L_c \|x-v\| \quad   \forall x, v \in \mathbb{R}^n,
\end{equation*}
and $\lambda_f', \lambda_J' \geq 0$ such that
\begin{equation*}
    \|\nabla f(x)\| \leq \lambda_f', \quad \quad \|J(x)\| \leq \lambda_J', \quad   \forall x \in \mathbb{R}^n.
\end{equation*}
\mo{In addition, we assume that $f(x)$ and $g(x)$ are lower bounded by $f_{\min}$ and $g_{\min}$ for all $x \in \mathbb{R}^n$.}
\end{assumption}

This assumption guarantees that the \mo{gradient of the objective}, \mo{the constraints} and \mo{the Jacobian of the constraints} are Lipschitz continuous, \mo{which is a very standard assumption in the literature}.

\begin{assumption}[Regularity] \label{assum:regularity}
For all iterates $k$ of our algorithm, we assume that
\begin{equation} \label{eq:regularity}
    \nu\|c(x_{k+1})\| \leq \dist (-J(x_{k+1})^T c(x_{k+1}), \frac{\partial g(x_{k+1})}{\beta_k})
\end{equation}
for some $\nu > 0$.
\end{assumption}
The second assumption is also relatively standard, \mo{and has been used in a number of recent works on Augmented Lagrangian methods \cite{Sahin2019,Li2021,Li2023}. In addition, in the case where $g(x) = 0$, this condition is equivalent to assuming the Polyak-\Lo ojasiewicz condition on the nonlinear least-squares feasibility problem, $\|c(x)\|^2 = 0$, and always holds when the linearly independent constraint qualification (LICQ) holds at each iteration.} Thus, one can view this assumption as a constraint qualification which prevents degeneracy.

Next, we switch to stochastic setting, where we assume access to appropriate stochastic oracles, which we denote with a slight abuse of notation.

\begin{assumption}[Stochastic Oracle] \label{assum:stochastic_first_oracle}
For the stochastic version of problem (\ref{eq:probdef}), a stochastic first-order oracle can be accessed. That is, for any $x \in \dom(g)$, the oracle can obtain a sample $\xi$ and return $(\nabla f(x, \xi), c(x, \xi), J(x, \xi))$.
\end{assumption}

\mo{Now, we consider two standard assumptions on the stochastic oracles.}

\begin{assumption}[Unbiasedness and Bounded Variance] \label{assum:Unbiasedness_and_Bounded_Variance}
For any $x \in \dom(g)$, the \mo{stochastic gradient estimates satisfy,}
\begin{equation*} 
    \mathbb{E}_\xi[\nabla f(x, \xi)] = \nabla f(x), \quad \mathbb{E}_\xi[\|\nabla f(x, \xi) - \nabla f(x)\|^2] \leq \sigma_g^2,
\end{equation*}
for some $\sigma_f > 0$. \mo{Furthermore, there exists $\sigma_c \geq 0$, such that,}
\begin{align*}
&\mathbb{E}_\xi[c(x, \xi)] = c(x), \\
&\mathbb{E}_\xi[J(x, \xi)] = J(x), \\
&\mathbb{E}_\xi[\|J(x, \xi) - J(x)\|^2] \leq \sigma_c^2; \\
&\mathbb{E}_{\xi_1,\xi_2 }[\|J(x, \xi_1)^T c(x, \xi_2) - J(x)^T c(x)\|^2] \leq \sigma_c^2,
\end{align*}
where $\xi_1$ and $\xi_2$ are independent and follow the same distribution as $\xi$. 
\end{assumption}

\mo{We note here that we intentionally consider the possibility that $\sigma_c = 0$, which corresponds to the case where the objective is stochastic but the constraints remain deterministic, while not requiring additional analysis specific to this setting.}

\mo{Finally, we occasionally consider a stronger condition on the stochastic gradient and Jacobian estimates, in which they are sufficiently smooth functions for any fixed sample.}

\begin{assumption}[Mean-squared Smoothness] \label{assum:Mean-squared_Smoothness}
For all $x,v \in \dom(g)$, \mo{the stochastic gradient estimates satisfy}:
\begin{equation*}
    \mathbb{E}_\xi[\|\nabla f(x, \xi) - \nabla f(v, \xi)\|^2] \leq L_f^2\|x-v\|^2,
\end{equation*}
for some $L_f > 0$. \mo{Furthermore,}
\begin{align*}
    &\mathbb{E}_\xi[\| J(x, \xi) - J(v, \xi)\|^2] \leq L_J^2\|x-v\|^2 \\
    &\mathbb{E}_{\xi_1,\xi_2 }[\|J(x, \xi_1)^T c(x, \xi_2) - J(v, \xi_1)^T c(v, \xi_2)\|^2] \leq L_J^2\|x-v\|^2,
\end{align*}
where $L_J > 0$ and $\xi_1$ and $\xi_2$ are independent and follow the same distribution as $\xi$. 
\end{assumption}

This last assumption extends smoothness to the stochastic case by bounding the mean-squared variation of the relevant derivatives along sample paths. Notice that this assumption is more restrictive than that of Assumption \ref{assum:Unbiasedness_and_Bounded_Variance}. \mo{This assumption is necessary for algorithms based on variance reduction techniques, which in turn achieve stronger complexity guarantees.}

\section{Algorithm}

\mo{In this section, we present our proposed adaptive Augmented Lagrangian algorithm as well as a simple inner loop procedure for the stochastic setting.}

\subsection{Main Algorithm}

\begin{algorithm}
\caption{Inexact Augmented Lagrangian Algorithm}\label{alg:al}
\begin{algorithmic}[1]
\Require $\epsilon\in \mathbb{R}_{>0}$; $\beta_1 \in \mathbb{R}_{>0}$; $x_1 \in \mathbb{R}^n$; $y_1 \in \mathbb{R}^m$; $y_{\max} \in \mathbb{R}_{>0}$; $\gamma > 1$; $\tau \in (0,1)$; $\{\eta_k'\}_{k\geq 1}$ a decreasing, positive sequence that converging to $0$.
\bigskip
\For{$k = 1,2,\dots$}
\State $\eta_{k} = \min\left\{\frac{1}{\beta_k}, \eta_{k}'\right\}$.
\smallskip
\State Find $x_{k+1}$ as an approximate solution to 
\[
\underset{x \in \mathbb{R}^n}{\min} \ \mathcal{L}_{\beta_k}(x, y_k) + g(x) \quad \text{such that} \quad \dist\left(-\nabla_x \mathcal{L}_{\beta_k}(x_{k+1}, y_k), \partial g(x_{k+1})\right) \leq \eta_{k}.
\]
\State Compute $\hat{y}_{k+1} = y_k + \beta_{k}c(x_{k+1})$; Quit with $x_{k+1}$ and $\hat{y}_{k+1}$ when
\[
\dist \left( -\nabla_x \mathcal{L}(x_{k+1}, \hat{y}_{k+1}), \partial g(x_{k+1}) \right) + \|c(x_{k+1})\| \leq \epsilon;
\] 

\State If $\|c(x_{k+1})\| \leq \tau \|c(x_{k})\|$, then $\beta_{k+1} = \beta_k$; otherwise, $\beta_{k+1} = \gamma \beta_k$.
\smallskip
\State $y_{k+1} = \text{proj}_{[-y_{\max}, y_{\max}]}(\hat{y}_{k+1}); \quad\ k = k + 1$.
\EndFor
\end{algorithmic}
\end{algorithm}

\mo{Our main algorithm is presented in Algorithm \ref{alg:al}, which follows the general structure of the classical Augmented Lagrangian framework with a few modifications. In general, this structure is largely similar to that of ALGENCAN \cite{Birgin2020}, except that explicit inequality constraints are replaced with the addition of the nonsmooth term $g$, which allows for the possibility of general inequality constraints via a slack variable reformulation.}

\mo{At each iteration, we set the tolerance for the inner solver, $\eta_{k}$, as
\begin{equation*}
    \eta_{k} = \min\left\{\frac{1}{\beta_k},\, \eta_k'\right\},
\end{equation*}
which incorporates the forcing sequence $\eta_k'$, which can be chosen to guarantee that the solution to the inner solver is sufficiently accurate, regardless of the size of $\beta_k$. This tolerance is then used for the inexact minimization to find $x_{k+1}$ satisfying
\begin{equation*}
    \operatorname{dist}\!\left(-\nabla_x \mathcal{L}_{\beta_k}(x_{k+1}, y_k),\, \partial g(x_{k+1})\right) \le \eta_k.
\end{equation*}}

\mo{The dual multiplier is computed in the standard manner, 
\begin{equation*}
    \hat{y}_{k+1} = y_k + \beta_k c(x_{k+1}),
\end{equation*}
with the penalty parameter $\beta_k$ serving as the dual stepsize. For the purposes of obtaining a complexity result, we also project this update onto the $\ell_\infty$ norm ball of radius $y_{\max}$. \footnote{In the analysis, we rely on a bound on $\|y\|_2$, which implicitly includes a $\sqrt{m}$ factor due to equivalence of norms. This could be easily avoided by projecting onto the $\ell_2$ norm ball instead, but we choose the $\ell_\infty$ ball as it is used in prior work \cite{Birgin2020}.}}

The penalty parameter is updated adaptively according to the constraint violation:
\[
\beta_{k+1} =
\begin{cases}
\beta_k, & \text{if } \|c(x_{k+1})\| \le \tau \|c(x_k)\|\\[3pt]
\gamma \beta_k, & \text{otherwise}
\end{cases}
\]
where $0 < \tau < 1$ and $\gamma > 1$. This mechanism increases $\beta_k$ only when the constraint violation fails to decrease appropriately, preventing over-penalization and maintaining well-conditioned subproblems.

\subsection{Inner Loop Solver for Stochastic Problems}

\begin{algorithm}
\caption{Inner Loop Solver using Stochastic Methods}
\label{alg:innersolve}
\begin{algorithmic}[1]
\Require Fixed $y_k \in \mathbb{R}^m $; starting $x_k\in \mathbb{R}^n $ ; tolerance $\eta_k>0$; initial epochs $T_{k,1}\in\mathbb{N}_{>0}$; $r>1$; stochastic proximal method $\mathcal{M}$.
\For{$p = 1,2,... $}

\State Find $x_{k,p}$ as an approximate solution to $\min_x \mathcal{L}_{\beta_k}(x,y_k) +g(x)$ with $T_{k,p}$ \mo{iterations} of the 
\Statex \hspace{\algorithmicindent} inner-loop using stochastic method $\mathcal M$.

\If $\dist\left(-\nabla_x \mathcal{L}_{\beta_k}(x_{k,p}, y_k), \partial g(x_{k,p})\right) > \eta_k$ \mo{ or $\L_{\beta_k}(x_{k,p}, y_k) + g(x_{k,p}) > \L_{\beta_k}(x_{k}, y_k) + g(x_{k})$}
\State $T_{k,p+1}=\ r T_{k,p}$$.$
\Else
\State \textbf{Break}.
\EndIf
\EndFor
\State Return $x_{k+1} = x_{k,p}$.
\end{algorithmic}
\end{algorithm}

{Algorithm \ref{alg:innersolve} is designed to approximately minimize the augmented Lagrangian subproblem in the inexact ALM framework when the objective or constraint functions involve stochastic components or large-scale data. Unlike prior work, our proposed algorithm enforces that the approximate solution $x_{k+1}$ satisfies the inexactness condition
\[
\operatorname{dist}\!\left(-\nabla_x \mathcal{L}_{\beta_k}(x_{k+1}, y_k),\, \partial g(x_{k+1})\right) \le \eta_k,
\]
as opposed to simply hoping that the parameters are set such that the stochastic method makes the above result hold in expectation. In practice, it is often difficult to determine the correct number of inner-loop epochs required to achieve this level of accuracy, especially in the presence of stochastic noise. Algorithm \ref{alg:innersolve} provides a simple yet effective adaptive mechanism to address this issue. \mo{In Section \ref{subsec:stochasticexperiments}, we validate the efficacy of this approach. Additionally, we require that our solution is monotonic with respect to the previous iterate $x_k$, which enables us to obtain tight bounds on the function value gap at each outer loop iteration, which is critical for our complexity results. This condition allows us to avoid additional unnecessary assumptions, such as boundedness of the domain of $g$, which has been used in prior work \cite{Li2021,Li2023}, but precludes the possibility of important problem classes, such as inequality constrained problems via a slack variable reformulation.}

Algorithm \ref{alg:innersolve} starts from an initial number of epochs $T_{k,0}$ and an initial point $x_k$. At each iteration $p$, it applies a stochastic optimization method $\mathcal{M}$ (such as stochastic gradient descent, SVRG, or momentum-based stochastic methods) to perform $T_{k,p}$ epochs of the chosen algorithm to minimize $\mathcal{L}_{\beta_k}(x, y_k) + g(x)$. Here we also assume that we have access to the full batch quantities. If the output is not sufficiently accurate, then the number of epochs is increased by a factor $r > 1$ and the loop is repeated. In Section \ref{sec:stochasticanalysis}, we show that this algorithm terminates, with high probability, in a comparable number of epochs to standard stochastic methods.

\section{Convergence Analysis}

In this section, we will derive our main results, beginning with a complexity result for Algorithm \ref{alg:al}.

\subsection{Outer-loop Complexity Result}

We start our analysis by establishing an upper bound on the constraint violation across all iterations. 

\begin{lemma}\label{lem:grad_upperbound}
Let Assumption \ref{assum:smoothness} and Assumption \ref{assum:regularity} hold  and let $\bar y_{\max} = \sup\{\|y_k\|_2\}_{k\geq 1}$ \mo{(which is guaranteed to exist by the construction of Algorithm \ref{alg:al})}. Then, we have that: 
\begin{enumerate}
    \item[(i)] $\|c(x_{k+1})\| \leq \frac{\lambda_f' + \lambda_J'\bar y_{\max} + \frac{1}{\beta_k}}{\nu \beta_k}$ for all $k \geq 1$.
    \item[(ii)] There exists $ c_\text{big} := \max\{ \|c(x_1)\|, \frac{\lambda_f' + \lambda_J'\bar y_{\max} + \frac{1}{\beta_1}}{\nu \beta_1}\}$ such that $\|c(x_k)\| \leq c_\text{big}$ for all $k$.
\end{enumerate}
\end{lemma}

\begin{proof} 
From step 3 of Algorithm \ref{alg:al}, the stopping criterion for the inner loop subproblem ensures
\begin{equation*}
    \dist (-\nabla_x \mathcal{L}_{\beta_k}(x_{k+1}, y_k), \partial g(x_{k+1})) \leq \eta_k \leq \frac{1}{\beta_k}.
\end{equation*}
By the definition of the gradient of the Augmented Lagrangian function,
\begin{equation*}
    \dist (-\nabla f(x_{k+1}) - J(x_{k+1})^T y_k - \beta_k J(x_{k+1})^T c(x_{k+1}), \partial g(x_{k+1}))  \leq \frac{1}{\beta_k}.    
\end{equation*}
Applying the triangle inequality and rearranging the terms, we have
\begin{equation*}
    \dist (-\beta_k J(x_{k+1})^T c(x_{k+1}), \partial g(x_{k+1}))  \leq \|\nabla f(x_{k+1})\| + \|J(x_{k+1})^T y_k\| +  \frac{1}{\beta_k},
\end{equation*}
and therefore,
\begin{align*}
 \dist (-J(x_{k+1})^T c(x_{k+1}) , \frac{\partial g(x_{k+1})}{\beta_k})  
 &\leq \frac{\|\nabla f(x_{k+1})\| + \|J(x_{k+1})^T y_k\| +  \frac{1}{\beta_k}}{\beta_k} \\
 & \leq \frac{\lambda_f' + \lambda_J'\bar y_{\max} + \frac{1}{\beta_k}}{\beta_k}. 
\end{align*}
By Assumption \ref{assum:regularity}, we have that $\|c(x_{k+1})\| \leq \frac{\lambda_f' + \lambda_J'\bar y_{\max} + \frac{1}{\beta_k}}{\nu \beta_k}$ for all $k \geq 1$. The second result follows by noting that $\beta_k \geq \beta_1$ for all $k$.
\end{proof}

The uniform bound from the previous lemma allows us to control the constraint violation in terms of the penalty parameter. Using this, we can now derive sufficient conditions on the penalty parameter under which Algorithm \ref{alg:al} is guaranteed to terminate with an approximately feasible solution.

\begin{lemma}\label{lem:beta_upperbound5.2}
Let Assumptions \ref{assum:smoothness} and \ref{assum:regularity} hold. If in any iteration $k$, we have \begin{equation*}
    \beta_k \geq \beta_\textrm{upper} := \frac{\nu + \lambda_f' + \lambda_J'\bar y_{\max} + \sqrt{4\epsilon \nu + (\nu + \lambda_f' + \lambda_J'\bar y_{\max})^2}}{2\epsilon \nu},
\end{equation*} then Algorithm \ref{alg:al} terminates at iteration $k+1$ with $x_{k+1}$ and $\hat{y}_{k+1}$ which satisfy the approximate optimality condition \eqref{eq:approximate_optimality_condition}.
\end{lemma}

\begin{proof}

By step 4 of Algorithm \ref{alg:al}, the method terminates when
\begin{equation*}
    \dist (-\nabla_x \mathcal{L}(x_{k+1}, \hat{y}_{k+1}), \partial g(x_{k+1})) + \|c(x_{k+1})\|\leq \epsilon.
\end{equation*}
We consider two terms on the left-hand side of the inequality separately. By Lemma \ref{lem:grad_upperbound} part $(i)$, we have that $\|c(x_{k+1})\| \leq \frac{\lambda_f' + \lambda_J'\bar y_{\max} + \frac{1}{\beta_k}}{\nu \beta_k}$ holds for all $k$. Now, by step 3 of Algorithm \ref{alg:al}, the stopping criterion for the inner loop subproblem is set as
\begin{equation*}
    \dist (-\nabla_x \mathcal{L}_{\beta_k}(x_{k+1}, y_k), \partial g(x_{k+1})) \leq \eta_k.
\end{equation*}
Recalling that $\hat{y}_{k+1} = y_k + \beta_{k}c(x_{k+1})$, we have
\begin{equation*}
    \nabla_x \mathcal{L}_{\beta_k}(x_{k+1}, y_k)  = \nabla f(x_{k+1}) +  J(x_{k+1})^T\hat{y}_{k+1} = \nabla_x \mathcal{L}(x_{k+1}, \hat{y}_{k+1}).
\end{equation*}
Therefore, it follows that
\begin{equation*}
    \dist (-\nabla_x \mathcal{L}(x_{k+1}, \hat{y}_{k+1}), \partial g(x_{k+1})) = \dist (-\nabla_x \mathcal{L}_{\beta_k}(x_{k+1}, y_k), \partial g(x_{k+1})) \leq \eta_k \leq \frac{1}{\beta_k}.
\end{equation*}

Now, combining both quantities in the stopping criteria together, we have that
\begin{equation*}
    \dist (-\nabla_x \mathcal{L}(x_{k+1}, \hat{y}_{k+1}), \partial g(x_{k+1})) + \|c(x_{k+1})\|\leq \frac{1}{\beta_k} + \frac{\lambda_f' + \lambda_J'\bar y_{\max} + \frac{1}{\beta_k}}{\nu \beta_k}.
\end{equation*}

\mo{Thus, Algorithm \ref{alg:al} terminates whenever}
\begin{equation*}
    \frac{1}{\beta_k} + \frac{\lambda_f' + \lambda_J'\bar y_{\max} + \frac{1}{\beta_k}}{\nu \beta_k} \leq \epsilon.
\end{equation*}

Solving this inequality for $\beta_k$ yields a quadratic equation from which we obtain our desired result.
\end{proof}

The previous two lemmas are mild modifications of Theorem 4.1 of \cite{Sahin2019}, and are sufficient to guarantee convergence when $\beta_k$ becomes large. We now combine this with a modification of the argument in \cite{Birgin2020} to derive a complexity result for Algorithm \ref{alg:al}. 

\begin{theorem}\label{thm:complexity_outer5.3} (Outer-loop Complexity)
Let Assumptions \ref{assum:smoothness} and \ref{assum:regularity} hold. Suppose that there exists $N(\epsilon) \in \{1, 2, ...\}$ such that for all $k\geq N(\epsilon)$, we have $\eta_k \leq \frac{\epsilon}{2}$. Then, 

\smallskip
$(i)$ Suppose that $\|c(x_{k+1})\|\leq \tau\|c(x_{k})\|$ happens at least \[\ceil[\Bigg]{\max \Biggl\{N(\epsilon), \frac{\log(\frac{2c_\text{big}}{\epsilon})}{|\log(\tau)|} \Biggl\}}\] consecutive iterations, then we reach the optimality for $c_\text{big}$ defined in the Lemma $\ref{lem:grad_upperbound}$ part $(ii)$;

\smallskip
$(ii)$ Otherwise, it takes at most \[ \ceil[\Bigg]{\max \Biggl\{N(\epsilon), \frac{\log(\frac{2c_\text{big}}{\epsilon})}{|\log(\tau)|} \Biggl\} -1} \times \ceil[\Bigg]{\frac{\log(\frac{\beta_{\textrm{upper}}}{\beta_1})}{\log(\gamma)}} \] iterations to reach condition (\ref{eq:approximate_optimality_condition}) with $x_{k+1}$, $\hat{y}_{k+1}$, and $\beta_{\textrm{upper}}$ from Lemma $\ref{lem:beta_upperbound5.2}$.  
\end{theorem}

\begin{proof}
We first consider case $(i)$.
By step 4 of Algorithm \ref{alg:al}, the algorithm terminates when \[ \dist (-\nabla_x \mathcal{L}(x_{k+1}, \hat{y}_{k+1}), \partial g(x_{k+1})) + \|c(x_{k+1})\|\leq \epsilon.\] 
In the proof of Lemma \ref{lem:beta_upperbound5.2}, we obtain a relationship between the gradient of the Augmented Lagrangian function and the Lagrangian function. Combining the relationship with the assumption on $N(\epsilon)$, we have that \[ \dist (-\nabla_x \mathcal{L}_{\beta_k}(x_{k+1}, y_{k}), \partial g(x_{k+1})) = \dist (-\nabla_x \mathcal{L}(x_{k+1}, \hat{y}_{k+1}), \partial g(x_{k+1})) \leq \eta_k \leq \frac{\epsilon}{2}\] for some $x_{k+1}$ and $\hat{y}_{k+1}$. 

Meanwhile, suppose $\|c(x_{k+1})\|\leq \tau\|c(x_{k})\|$ happens for at least $T$ consecutive iterations. By Lemma \ref{lem:grad_upperbound}, we have that $\|c(x_{k+1})\| \leq c_\text{big}$ holds for all $k$. 

Thus, the algorithm terminates once when $\max\{N(\epsilon), \tau^T c_\text{big}\} \leq \frac{\epsilon}{2}$. Since $0 < \tau < 1 $, taking the natural logarithm from the second portion inside the max function yields \[ T \geq \frac{\log(\frac{2c_\text{big}}{\epsilon})}{|\log(\tau)|}.\]

We now consider case $(ii)$. In the worst-case scenario, $\|c(x_{k+1})\|\leq \tau\|c(x_{k})\|$ happens for exactly
\begin{equation} \label{eq:numconsec}
    \max \Biggl\{N(\epsilon), \frac{\log(\frac{2c_\text{big}}{\epsilon})}{|\log(\tau)|} \Biggl\} -1
\end{equation}
consecutive iterations. 

Under such circumstances, it is possible that we fail to reach the optimality from consecutive reduction in constraint violation. \mo{However, by Lemma \ref{lem:beta_upperbound5.2}, once $\beta_k$ exceeds $\beta_{upper}$, the algorithm terminates at $k+1$. In addition, by step 5 of Algorithm \ref{alg:al}, $\beta_{k+1} = \gamma^j \beta_1$, where $j$ is the number of iterations for which $\|c(x_{k+1})\| > \tau \|c(x_k)\|$. Thus, we have that $j \leq \frac{\log(\frac{\beta_{upper}}{\beta_1})}{\log(\gamma)}$. Combining this with \eqref{eq:numconsec} yields the result.} 

\end{proof}

The previous theorem \mo{demonstrates two modes of convergence}. First, if the \mo{constraint violation decreases sufficiently quickly}, then the algorithm terminates \mo{without needing to excessively increase the penalty parameter}. However, \mo{when the constraint violation is not well-behaved, the only option available is to increase penalty parameter significantly, which in turn yields convergence. We note that this behavior is reminiscent to classical convergence results for Augmented Lagrangian methods e.g. \cite[Proposition 4.2.3]{bertsekas1999nonlinear}, in which convergence can be driven \textbf{either} by good dual variable estimates or a sufficiently large penalty parameter. Most prior work (e.g. \cite{Sahin2019,Li2021,Li2023}) only focus on a single mode of convergence, that is, when the penalty parameter is sufficiently large. In the worst case, the outer loop requires at most $O(\log^2(\epsilon^{-1}))$ iterations to reach approximate feasibility and stationarity. This logarithmic dependence is a mild additional cost imposed by the adaptive nature of the algorithm.}

\subsection{Deterministic Complexity Result}

\mo{In order to prove a complete complexity result, we now consider the cost of solving the Augmented Lagrangian subproblem at each iteration. To this end, we prove the following result about the gradient of the Augmented Lagrangian.} 

\begin{lemma}\label{lem:Lipschitz5.4}
\mo{The gradient of the Augmented Lagrangian function, $\nabla_x \mathcal{L}_{\beta_k}(x,y)$, is Lipschitz continuous in $x$ with parameter $L_k = L_f + \bar{y}_{\max} L_J + \beta_k L_J c_{\text{big}} + \beta_k \lambda_J' L_c$,} where, $c_{\text{big}}$ and $\bar{y}_{\max}$ are defined in Lemma \ref{lem:grad_upperbound}. 
\end{lemma}

\begin{proof}
By the triangle inequality, we have 
\begin{equation*}
 \|\nabla \mathcal{L}_{\beta_k}(x,y) - \nabla \mathcal{L}_{\beta_k}(v,y)\| \leq \|\nabla f(x) - \nabla f(v)\| + \bar{y}_{\max} \|J(x) - J(v)\| + \beta_k\|J(x)^T c(x) - J(v)^T c(v)\|.
\end{equation*}
We arrive at the result for the first two terms by applying Assumption \ref{assum:smoothness} directly. For the final term in the last inequality, we use Assumption \ref{assum:smoothness} and the upper bound on $\|c(x)\|$ from Lemma \ref{lem:grad_upperbound} to obtain 
\begin{align*}
 \beta_k\|J(x)c(x) - J(v)c(v)\|  
 & \leq \beta_k\|J(x) - J(v)\|\|c(x)\|+ \beta_k\|J(v)\|\|c(x)-c(v)\|\\
 & \leq (\beta_k L_J c_{\text{big}} + \beta_k \lambda_J' L_c)\|x-v\|,
\end{align*}
which proves the result.
\end{proof}

The next lemma plays a key role in connecting proximal gradient updates with the stopping criterion of Algorithm \ref{alg:al}. The proof is a minor twist on \cite[Lemma 9.5]{wright2022optimization}, in the sense that we do not assume convexity of our smooth function.

\begin{lemma}\label{lem:proximal_grad5.5}
Fix an outer iteration $k$. Let the inner subproblem at iteration $k$ be $\phi_k(x)=\mathcal{L}_{\beta_k}(x,y_k)+g(x).$ For any $\alpha_k>0$, define the proximal operator of the nonsmooth portion $g(x)$ as \[\operatorname{prox}_{\alpha_k g}(x) = \displaystyle\arg\min_u\left\{g(u)+\frac{1}{2\alpha_k}\|u-x\|^2\right\};\]

then for any $\alpha_k>0$, let the proximal gradient step be \[ \psi_{\alpha_k}(x)=\operatorname{prox}_{\alpha_k g}\bigl(x-\alpha_k \nabla_x \mathcal{L}_{\beta_k}(x,y_k)\bigr); \]
and the proximal gradient mapping as \[ G_{\alpha_k}(x)=\frac{1}{\alpha_k}\bigl(x-\psi_{\alpha_k}(x)\bigr). \]
Then the following hold:
\begin{enumerate}
    \item[(i)] 
    \[ G_{\alpha_k}(x)-\nabla_x\mathcal{L}_{\beta_k}(x,y_k)\in \partial g\bigl(\psi_{\alpha_k}(x)\bigr). \]

    \item[(ii)] For any $z\in \dom(g)$ and $\alpha_k\in(0,\frac{1}{L_k}]$, \[ \phi_k\bigl(\psi_{\alpha_k}(x)\bigr) \le \mathcal{L}_{\beta_k}(x,y_k)+g(z) -\langle \nabla_x\mathcal{L}_{\beta_k}(x,y_k),x-z\rangle +\langle G_{\alpha_k}(x),x-z\rangle -\frac{\alpha_k}{2}\|G_{\alpha_k}(x)\|^2.\]

    \item[(iii)] Moreover, when $\alpha_k=\frac{1}{L_k}$, the stationarity residual at the proximal gradient iterate  $\psi_{\alpha_k}(x)$ is bounded by the gradient mapping; namely, \[
    \operatorname{dist}\!\left( -\nabla_x\mathcal{L}_{\beta_k}\bigl(\psi_{\alpha_k}(x),y_k\bigr), \partial g\bigl(\psi_{\alpha_k}(x)\bigr) \right) \le
    2\|G_{\alpha_k}(x)\|.\]
    Consequently, given the inner subproblem tolerance $\eta_k$, if
    $\|G_{\alpha_k}(x)\|\le \frac{\eta_k}{2}$,
    then the stopping criterion of the inner subproblem in Algorithm \ref{alg:al} is satisfied.
\end{enumerate}
\end{lemma}

\begin{proof}
By the first-order optimality condition for the proximal operator, we have 
\[\psi_{\alpha_k}(x) = \operatorname{prox}_{\alpha_k g}\bigl(x-\alpha_k \nabla_x \mathcal{L}_{\beta_k}(x,y_k)\bigr) =  \displaystyle\arg\min_u\left\{g(u)+\frac{1}{2\alpha_k}\|u- (x-\alpha_k \nabla_x \mathcal{L}_{\beta_k}(x,y_k))\|^2\right\}. \]
Thus, we have \[ 0\in \partial g(\psi_{\alpha_k}(x)) + \frac{1}{\alpha_k}\bigl(\psi_{\alpha_k}(x)-(x-\alpha_k \nabla_x \mathcal{L}_{\beta_k}(x,y_k))\bigr). \]
Plugging in  the definition of $G_{\alpha_k}(x)$ yields part $(i)$.

Next, since $\nabla_x \mathcal{L}_{\beta_k}(x,y_k)$ is $L_k$-Lipschitz by Lemma \ref{lem:Lipschitz5.4}, for $ \alpha_k \in (0,\frac{1}{L_k}]$, we have
\begin{align*}
\mathcal{L}_{\beta_k}(\psi_{\alpha_k}(x),y_k) &\le\mathcal{L}_{\beta_k}(x,y_k)+\langle \nabla_x \mathcal{L}_{\beta_k}(x,y_k),\psi_{\alpha_k}(x)-x\rangle+\frac{L_k}{2}\|\psi_{\alpha_k}(x)-x\|^2\\
& \le \mathcal{L}_{\beta_k}(x,y_k)- \alpha_k \langle  \nabla_x \mathcal{L}_{\beta_k}(x,y_k),G_{\alpha_k}(x)\rangle +\frac{\alpha_k}{2}\|G_{\alpha_k}(x)\|^2
\end{align*}

Meanwhile, by convexity of $g$ and result of $(i)$, for any $z\in \dom(g)$, we have 
\[g(\psi_{\alpha_k}(x)) \le g(z)+\langle G_{\alpha_k}(x) -\nabla_x \mathcal{L}_{\beta_k}(x,y_k),\psi_{\alpha_k}(x)-z\rangle.\]

Adding the two inequalities and plugging in $\psi_{\alpha_k}(x) = x - \alpha_k G_{\alpha_k}(x)$ yields part $(ii)$.

Finally, we establish the connection between the stopping criterion in Algorithm \ref{alg:al} and the proximal gradient mapping $G_{\alpha_k}(x)$. By part $(i)$, we can choose
\[
p:= G_{\alpha_k}(x)-\nabla_x\mathcal{L}_{\beta_k}(x,y_k)\in \partial g\bigl(\psi_{\alpha_k}(x)\bigr),
\]
then it follows that
\begin{align*}
\operatorname{dist}\!\left( -\nabla_x\mathcal{L}_{\beta_k}\bigl(\psi_{\alpha_k}(x),y_k\bigr), \partial g\bigl(\psi_{\alpha_k}(x)\bigr) \right)
&= \inf_{p \in \partial g\bigl(\psi_{\alpha_k}(x)\bigr)} \| -\nabla_x \mathcal{L}_{\beta_k}(\psi_{\alpha_k}(x),y_k)-p\| \\
&\le \|\nabla_x \mathcal{L}_{\beta_k}(x,y_k)-\nabla_x \mathcal{L}_{\beta_k}(\psi_{\alpha_k}(x),y_k)-G_{\alpha_k}(x)\| \\
&\le \|\nabla_x \mathcal{L}_{\beta_k}(x,y_k)-\nabla_x \mathcal{L}_{\beta_k}(\psi_{\alpha_k}(x),y_k)\|+\|G_{\alpha_k}(x)\|.
\end{align*}
By Lemma \ref{lem:Lipschitz5.4},
\[\|\nabla_x\mathcal{L}_{\beta_k}(x,y_k) - \nabla_x \mathcal{L}_{\beta_k}(\psi_{\alpha_k}(x),y_k)\| \le L_k\|x-\psi_{\alpha_k}(x)\| =
L_k\alpha_k\|G_{\alpha_k}(x)\|. \]
Choose $\alpha_k= \frac{1}{L_k}$, we obtain
\[ \operatorname{dist}\bigl(-\nabla_x\mathcal{L}_{\beta_k}\bigl(\psi_{\alpha_k}(x),y_k\bigr), \partial g\bigl(\psi_{\alpha_k}(x)\bigr)  \le
(1+L_k\alpha_k)\|G_{\alpha_k}(x)\| = 2\|G_{\alpha_k}(x)\|. \]
Therefore, if \[ \|G_{\alpha_k}(x)\|\le \frac{\eta_k}{2}, \]
then \[ \operatorname{dist}\bigl(-\nabla_x\mathcal{L}_{\beta_k}\bigl(\psi_{\alpha_k}(x),y_k\bigr), \partial g\bigl(\psi_{\alpha_k}(x)\bigr)\le \eta_k. \]
This shows that it suffices to drive the gradient mapping norm below $\frac{\eta_k}{2}$ in order to satisfy the subproblem stopping criterion in Algorithm \ref{alg:al}.

\end{proof}

\mo{Next, we prove a uniform upper bound on $\L_{\beta_k}(x_k, y_k)$, which will appear in our final complexity result.}
\begin{lemma} \label{lem:Lupperbound}
    Let Assumptions \ref{assum:smoothness} and \ref{assum:regularity} hold. In addition, assume that
    \begin{equation*}
        \L_{\beta_k}(x_{k+1}, y_k) + g(x_{k+1}) \leq \L_{\beta_{k}}(x_k, y_k) + g(x_k)
    \end{equation*}
    holds at every iteration $k$. Let
    \begin{equation*}
        \hat{N}(\epsilon) := \ceil[\Bigg]{\max \Biggl\{N(\epsilon), \frac{\log(\frac{2c_\text{big}}{\epsilon})}{|\log(\tau)|} \Biggl\} -1} \quad \text{and} \quad \cL := \frac{(4 \nu \ymax + \gamma (\lambda_f' + \lambda_J' \ymax + \frac{1}{\beta_1})) (\lambda_f' + \lambda_J' \ymax + \frac{1}{\beta_1})}{2\beta_1\nu^2(1-\gamma^{-1})},
    \end{equation*}
    where $\ymax$ is defined in Lemma \ref{lem:grad_upperbound}. Then, for all $k$,
    \begin{equation*}
        \L_{\beta_k}(x_k, y_k) + g(x_k) \leq  \L_{\beta_1}(x_1,y_1) + g(x_1) + \cL \hat{N}(\epsilon).
    \end{equation*}
\end{lemma}

\begin{proof}
    By the definition of $\L_{\beta_k}(x_k,y_k)$, the statement of the lemma, and the results of Lemma \ref{lem:grad_upperbound},
    \begin{align*}
        \L_{\beta_k}(x_k,y_k) + g(x_k) &= \L_{\beta_{k-1}}(x_k,y_{k-1}) + g(x_k) + (y_k-y_{k-1})^T c(x_k) + \frac{\beta_k - \beta_{k-1}}{2} \|c(x_k)\|^2 \\
        &\leq \L_{\beta_{k-1}}(x_{k-1},y_{k-1}) + g(x_{k-1}) + \frac{(4 \nu \ymax + \gamma (\lambda_f' + \lambda_J' \ymax + \frac{1}{\beta_1})) (\lambda_f' + \lambda_J' \ymax +  \frac{1}{\beta_1})}{2\nu^2 \beta_{k-1}}.
    \end{align*}
    Applying this bound recursively, we have that
    \begin{equation*}
        \L_{\beta_k}(x_k, y_k) + g(x_k) \leq  \L_{\beta_1}(x_1, y_1) + g(x_1) + \frac{(4 \nu \ymax + \gamma (\lambda_f' + \lambda_J' \ymax + \frac{1}{\beta_1})) (\lambda_f' + \lambda_J' \ymax +  \frac{1}{\beta_1})}{2\nu^2} \sum_{j=1}^{k-1} \frac{1}{\beta_{j-1}}.
    \end{equation*}
    By Theorem \ref{thm:complexity_outer5.3}, we know that the maximum number of consecutive iterations before we increase $\beta_k$ or terminate is bounded by $\hat{N}(\epsilon)$. Thus, for all $j$, we have that $\beta_j \geq \beta_1 \gamma^{\floor{k/\hat{N}(\epsilon)}}$ and thus,
    \begin{align*}
        \L_{\beta_k}(x_k, y_k) + g(x_k) &\leq  \L_{\beta_1}(x_1, y_1) + g(x_1) \\
        &\quad+ \frac{\hat{N}(\epsilon)(4 \nu \ymax + \gamma (\lambda_f' + \lambda_J' \ymax + \frac{1}{\beta_1})) (\lambda_f' + \lambda_J' \ymax +  \frac{1}{\beta_1})}{2 \beta_1\nu^2} \sum_{j=1}^{\floor{(k-1)/\hat{N}(\epsilon)}} \frac{1}{\gamma^j} \\
        &\leq  \L_{\beta_1}(x_1, y_1) + g(x_1) + \frac{\hat{N}(\epsilon)(4 \nu \ymax + \gamma (\lambda_f' + \lambda_J' \ymax + \frac{1}{\beta_1})) (\lambda_f' + \lambda_J' \ymax + \frac{1}{\beta_1})}{2\beta_1\nu^2(1-\gamma^{-1})},
    \end{align*}
    where the final inequality follows by recognizing the summation as a geometric series with parameter $\gamma^{-1}$.
\end{proof}

The mild additional assumption that, the inner subproblem solution is at least as good in function value as the previous iterate can be easily enforced by any monotonic inner solver started at $x_k$. It should be noted that if the algorithm did not adaptively select $\beta_k$ and instead updated it at each iteration, then $\hat{N}(\epsilon) = 1$ and the previous result would be independent of $\epsilon$. Thus, the fact that $\L_{\beta_k}(x_k, y_k) = \mathcal{O}(\log(\epsilon^{-1}))$ is a direct consequence of the adaptivity of the algorithm, which we regard as a minor price to pay for the additional benefits of an adaptive method.

Next, we prove that the subproblem is uniformly lower bounded at each $k$, independently of $\beta_k$.

\begin{lemma} \label{lem:Llowerbound}
    Let Assumption \ref{assum:smoothness} hold. 
    Then,
    \begin{equation*}
        \underset{x \in \dom(g)}{\inf} \ \L_{\beta_k}(x, y_k) + g(x) \geq f_{\min} + g_{\min} - \frac{1}{2\beta_1} \ymax^2.
    \end{equation*}
\end{lemma}

\begin{proof}
    By the definition of the Augmented Lagrangian,
    \begin{equation*}
        \underset{x \in \dom(g)}{\inf} \mathcal{L}_{\beta_k}(x, y_k) + g(x) = \underset{x \in \dom(g)}{\inf} f(x) + g(x) + y_k^T c(x) + \frac{\beta_k}{2} \|c(x)\|^2 \geq f_{\min} + g_{\min} - \frac{1}{2 \beta_1} \ymax^2,
    \end{equation*}
    where we used Young's inequality, the definition of $\ymax$, and $\beta_k \geq \beta_1$ in the final inequality.
\end{proof}

Given these results, we are prepared to obtain an overall complexity guarantee for the deterministic setting.

\begin{theorem}\label{complexity_inner_deter} (Overall Complexity for Deterministic Setting)

Suppose that the conditions of Theorem \ref{thm:complexity_outer5.3} hold.  \mo{If, at every outer iteration $k$, the subproblem in Step 3 of Algorithm \ref{alg:al} is solved via the proximal gradient method with $\alpha_k = 1/L_k$, then,} the total number of proximal gradient iterations required by Algorithm \ref{alg:al} to return an $\epsilon$-approximate feasible solution is
\[ \widetilde{\mathcal O}(\epsilon^{-3}). \]
\end{theorem}

\begin{proof}

Consider the behavior of the inner subproblem solver for a fixed inner iteration $k$. By Lemma \ref{lem:proximal_grad5.5}$(ii)$ applied to the inner subproblem $\phi_k$ with $z=x$, for every step of proximal gradient, we have
\begin{equation*}
\phi_k(\psi_{\alpha_k}(x))\le \phi_k(x)-\frac{\alpha_k}{2}\|G_{\alpha_k}(x)\|^2.
\end{equation*}
We note that this inequality clearly proves monotonicity of the iterates in the inner solver, and thus the results of Lemma \ref{lem:Lupperbound} hold. Summing the above inequality for subproblem iterations $t=1,\dots,T$, we have
\begin{equation*}
    \phi_k(x_k^{T+1}) \le \phi_k(x_k^1)-\frac{\alpha_k}{2}\sum_{t=1}^{T}\|G_t\|^2,
\end{equation*}
where we use $G_t$ as shorthand to denote $G_{\alpha_k}(x^t_k)$. Rearranging this inequality and 
using the results of Lemmas \ref{lem:Lupperbound} and \ref{lem:Llowerbound},
\begin{equation*}
    \frac{\alpha_k}{2}\sum_{t=1}^{T}\|G_t\|^2 \le \phi_k(x_k^1)-\phi_k(x_k^{T+1}) \le \L_{\beta_1}(x_1, y_1) + g(x_1) - f_{\min} - g_{\min} + \cL \hat{N}(\epsilon) + \frac{1}{2 \beta_1} \ymax^2,
\end{equation*}

Using $\alpha_k=\frac{1}{L_k}$, we have that
\begin{equation*}
    \min_{1\le t\le T}\|G_t\|^2 \le \frac{2L_k(\L_{\beta_1}(x_1, y_1) + g(x_1) - f_{\min} - g_{\min} + \cL \hat{N}(\epsilon) + \frac{1}{2 \beta_1} \ymax^2)}{T}.
\end{equation*}
Hence, if
\begin{equation*}
    T_k \ge \frac{8L_k(\L_{\beta_1}(x_1, y_1) + g(x_1) - f_{\min} - g_{\min} + \cL \hat{N}(\epsilon) + \frac{1}{2 \beta_1} \ymax^2)}{\eta_k^2},
\end{equation*}
then there exists some $t\in\{1,\dots,T_k\}$ such that \[ \|G_t\|\le \frac{\eta_k}{2}. \]
By part $(iii)$ of Lemma \ref{lem:proximal_grad5.5}, the corresponding iterate $x_k^{t+1}$ satisfies
\[
\operatorname{dist}\bigl(-\nabla_x L_{\beta_k}(x_k^{t+1},y_k),\partial g(x_k^{t+1})\bigr)\le \eta_k.
\]

Therefore, the number of proximal gradient iterations required by the $k$-th inner solve is
\[
T_k=O\!\left(\frac{L_k \hat{N}(\epsilon)}{\eta_k^2}\right).
\]
By Lemma \ref{lem:Lipschitz5.4}, we have that $L_k=O(\beta_k)$, while Lemma \ref{lem:beta_upperbound5.2} proves that $\beta_k=O(\epsilon^{-1})$. Additionally, by the definition of $\hat{N}(\epsilon)$, $\hat{N}(\epsilon) = \mathcal{O}(\log(\epsilon^{-1}))$ while Algorithm \ref{alg:al} sets $\eta_k=\Theta(\epsilon)$. Therefore, each subproblem solve requires at most
\[
O\!\left(\frac{\log(\epsilon^{-1})/\epsilon}{\epsilon^2}\right)
=
O(\log(\epsilon^{-1})\epsilon^{-3})
\]
proximal gradient iterations. Combining this with the logarithmic outer-loop bound from Theorem \ref{thm:complexity_outer5.3} yields the overall complexity up to logarithmic factors.

\end{proof}

\subsection{Stochastic Complexity Results}
\label{sec:stochasticanalysis}

We next turn to the stochastic setting. The first step is to define a stochastic gradient mapping and establish its unbiasedness, smoothness, and variance properties. Meanwhile, we showed the Lipschitz property still holds for the Augmented Lagrangian function under the stochastic setting, \mo{provided Assumption \ref{assum:Mean-squared_Smoothness} is satisfied}. These are formalized in the following lemma, which we adopted from \cite{Li2023}. We include the proof for completeness. 

\begin{lemma}\label{lem:Lipschitz_stochastic5.7}
Let Assumptions \ref{assum:smoothness}, \ref{assum:stochastic_first_oracle} and \ref{assum:Unbiasedness_and_Bounded_Variance} hold and fix $y\in \mathbb{R}^m$ \mo{ such that $\|y\| \leq \ymax$.} Let \[\Delta_\beta(x, y, \zeta) = \nabla f(x, \xi_1)+ J(x, \xi_1)^T y + \beta J(x, \xi_1)^T c(x, \xi_2),\] where $\zeta = (\xi_1, \xi_2)$ and $\xi_1, \xi_2$ are independent and follow the same distribution as $\xi$. Then for any $x,v \in \dom(g)$, it holds:
\begin{equation*}
    \mathbb{E}_\zeta[\Delta_\beta(x, y, \zeta)] = \nabla \mathcal{L}_{\beta}(x,y) \quad \text{and} \quad \mathbb{E}_\zeta[\|\Delta_\beta(x, y, \zeta) - \nabla \mathcal{L}_{\beta}(x,y)\|^2] \leq \sigma_{\mathcal{L}_{\beta}}^2,
\end{equation*}
where $\sigma_{\L_{\beta}} = \sqrt{3\sigma_f^2 + 3\sigma_c^2\mo{\ymax^2} + 3\beta^2 \sigma_c^2}.$ If, in addition, Assumption \ref{assum:Mean-squared_Smoothness} holds, then,
\begin{equation*}
    \mathbb{E}_\zeta[\|\Delta_\beta(x, y, \zeta) - \Delta_\beta(v, y, \zeta)\|^2] \leq L_{\L_{\beta}}^2\|x-v\|^2,
\end{equation*}
holds for any $x,v \in \dom(g)$, where $L_{\L_{\beta}} = \sqrt{3L_f^2 + 3L_J^2\mo{\ymax^2} + 3\beta^2 L_J^2}.$
\end{lemma}

\begin{proof}
Since $\xi_1$ and $\xi_2$ are independent, it holds
\[ \mathbb{E}_{\zeta}[\Delta_\beta(x, y, \zeta)] = \mathbb{E}_{\xi_1}[\nabla f(x, \xi_1)] + \mathbb{E}_{\xi_1}[J(x, \xi_1)^T y] + \beta \mathbb{E}_{\xi_1}[ J(x, \xi_1)]^T \mathbb{E}_{\xi_2}[c(x, \xi_2)].\] 

Since $\xi_1$ and $\xi_2$ follow the same distribution of $\xi$. We have that \[ \mathbb{E}_{\zeta}[\Delta_\beta(x, y, \zeta)] = \nabla f(x) + J(x)^{T}y + \beta J(x)^{T} c(x) = \nabla \mathcal{L}_{\beta}(x,y). \]
In addition, by the triangle inequality,
\begin{align*}
 \|\Delta_\beta(x, y, \zeta) - \nabla \L_\beta(x,y)\|^2 
 & \leq  3\|\nabla f(x, \xi_1) - \nabla f(x)\|^2\\
 &  + 3\|J(x, \xi_1) - J(x)\|^2 \mo{\|y\|^2} \\
 &  + 3\beta^2\|J(x, \xi_1)^T c(x, \xi_2) - J(x)^T c(x)\|^2\\
 & \leq 3(\sigma_f^2 + \sigma_c^2\mo{\ymax^2} + \beta^2\sigma_c^2). 
\end{align*}

Similarly, we can repeat the above proof when Assumption \ref{assum:Mean-squared_Smoothness} holds to derive the bound for $L_{\L_{\beta}}$.

\end{proof}

\mo{Given the previous lemma, we now provide the following result about the stochastic proximal gradient method. Due to the nature of our inner loop solver, Algorithm \ref{alg:innersolve}, we adapt standard results to prove a high probability convergence result, which will be sufficient to show that Algorithm \ref{alg:innersolve} terminates in a reasonable number of iterations with high-probability.}

\begin{lemma}
\label{lem:aug_lag_grad_upperbound_stochastic5.8}

Let Assumptions \ref{assum:smoothness}, \ref{assum:stochastic_first_oracle} and \ref{assum:Unbiasedness_and_Bounded_Variance} hold and fix outer iteration $k$ for Algorithm \ref{alg:al} and inner iteration $p$ for Algorithm \ref{alg:innersolve}. Meanwhile, in Algorithm \ref{alg:innersolve}, let the stochastic method $\mathcal{M}$ be the stochastic proximal gradient method applied to $ \phi_k(x):=L_{\beta_k}(x,y_k)+g(x)$ for $T_p$ iterations, with stepsize $\alpha_{k,p} > 0$, i.e. at stochastic proximal gradient iteration $t \in \{1, \dots, T_p\}$,
\begin{equation*}
    x_{t+1, k, p} = \operatorname{prox}_{\alpha_{k,p} g}\!\bigl(x_{t, k, p}-\alpha_{k,p} \Delta_{\beta_k}(x_{t, k, p}, y_k, \zeta_t)\bigr),
\end{equation*}
where $\Delta_\beta(x,y,\zeta)$ is defined in Lemma \ref{lem:Lipschitz_stochastic5.7}.
Sample $t^\star$ uniformly from $\{1, \dots, T_p\}$ and assume that $\L_{\beta_k}(x, y_k) + g(x)$ is lower bounded for all $x \in \dom(g)$. Let $L_k$ be defined in Lemma \ref{lem:Lipschitz5.4}, $\sigma_{\L_{\beta_k}}$ be defined in Lemma \ref{lem:Lipschitz_stochastic5.7}, and let $\Omega_{k} := \L_{\beta_k}(x_{k}, y_k) + g(x_{k}) - \underset{x \in \dom(g)}{\min}( \L_{\beta_k}(x, y_k) + g(x))$ denote the function value gap.
Define the failure event as
\begin{equation*}
    F_{k,p} := \left\{\dist\left(-\nabla_x \mathcal{L}_{\beta_k}(x_{k,p}, y_k), \partial g(x_{k,p})\right) ^2 > \eta_k^{2}\right\} \cup \left\{ \mathcal{L}_{\beta_k}(x_{k,p}, y_k) + g(x_{k,p}) > \mathcal{L}_{\beta_k}(x_{k}, y_k) + g(x_{k})\right\}.
\end{equation*}
Then, when $\alpha_{k,p} = \min\left\{\frac{1}{2 L_k}, \sqrt{\frac{\Omega_{k}}{L_k \sigma_{\L_{\beta_k}}^2 T_p}}\right\}$ and $x_{k,p} = x_{t^*,k,p}$,
\begin{equation*}
    \mathbb P\!\left(F_{k,p}\right)\le\frac{72(1+\frac{1}{\sqrt{2}})^2}{\eta_k^2}\max\left\{\frac{2 L_k \Omega_{k}}{T_{k,p}}, \sigma_{\L_{\beta_k}} \sqrt{\frac{2 L_k \Omega_{k}}{T_{k,p}}} \right\}.
\end{equation*}

\end{lemma}

\begin{proof}

By Lemma \ref{lem:Lipschitz_stochastic5.7}, the estimator $\Delta_{\beta_k}(x,y_k,\zeta)$ is unbiased for $\nabla_x \mathcal{L}_{\beta_k}(x,y_k)$ and has variance bounded by $\sigma_{\L_{\beta_k}}^2$. Moreover, $\nabla_x \mathcal{L}_{\beta_k}(\cdot,y_k)$ is $L_k$-Lipschitz by Lemma \ref{lem:Lipschitz5.4} and $g$ is convex, thus $\phi_k$ is $L_k$-weakly convex. Hence, the assumptions of \cite[Corollary 3.6]{Davis2019} apply to $\phi_k(x)$ and thus, $\mathbb{E}[\phi_k(x_{t^*,k,p}) \leq \phi_k(x_{x_k})]$,
\begin{equation*}
    \mathbb{E}[\|G_{1/L_k}(x_{t^*,k,p})\|^2] \leq 18(1+\frac{1}{\sqrt{2}})^2 \max\left\{\frac{2 L_k \Omega_{k}}{T_{k,p}}, \sigma_{\L_{\beta_k}} \sqrt{\frac{2 L_k \Omega_{k}}{T_{k,p}}} \right\},
\end{equation*}
holds, by using \cite[Theorem 4.5]{drusvyatskiy2018error} to connect the norm of the proximal gradient mapping and the norm of the gradient of the Moreau envelope. We note that the two events in $F_{k,p}^c$ can be simultaneously implied by the same event via a mild modification of the proof of \cite[Theorem 4.5]{drusvyatskiy2018error}, and thus combining this inequality with \ref{lem:proximal_grad5.5} $(iii)$ and Markov's inequality yields the result.
\end{proof}

Given the probabilistic bound in the previous lemma, we are now able to establish worst-case complexity bounds for Algorithm \ref{alg:innersolve} which hold with high probability. In the following, we use the terminology ``full oracle evaluations" to denote computation of $\nabla f(x)$ (and possible $c(x)$ and $J(x)$, when the constraints are stochastic) as opposed to a stochastic approximation of these quantities.

\begin{lemma}\label{lem:complexity_inner_stochastic5.9} (Inner-loop Complexity for Stochastic Proximal Gradient)
Let Assumptions \ref{assum:smoothness}, \ref{assum:stochastic_first_oracle} and \ref{assum:Unbiasedness_and_Bounded_Variance} hold and fix outer iteration $k$ for Algorithm \ref{alg:al}. Let the conditions of Lemma \ref{lem:aug_lag_grad_upperbound_stochastic5.8} hold. Let $C_k := 72(1+\frac{1}{\sqrt{2}})^2\max\left\{2 L_k \Omega_{k}, \sigma_{\L_{\beta_k}} \sqrt{2 L_k \Omega_{k}} \right\}$. Then, with probability 1, the inner-loop solver terminates finitely. In addition, with probability $1-\hat\delta$, the total number of evaluations required to satisfy the stopping criteria of Algorithm \ref{alg:innersolve} is no more than
\begin{align*}
   &\max\left\{\left\lceil\frac{r^2 C_k^2 - T_1 \eta_k^4 \hat{\delta}^2}{\eta_k^4 \hat{\delta}^2(r-1)}\right\rceil,T_1\right\} \text{\quad \; stochastic oracle evaluations,} \\
    &\max\left\{\left\lceil\frac{2}{\log r}\,\log\!\Big(\tfrac{C_k}{\eta_k^2\hat\delta\sqrt{T_1}}\Big)\right\rceil, 1\right\} \text{ \quad full oracle evaluations.}
\end{align*}
\end{lemma}
\begin{proof}

By the update rule of Algorithm \ref{alg:innersolve}, the number of proximal stochastic iterations at iteration $p$ is given by $T_{k,p} = r^p T_1$. By Lemma \ref{lem:aug_lag_grad_upperbound_stochastic5.8}, we have 
\begin{equation} \label{eq:failureprob}
    \mathbb{P}(F_{k,p}) \le \frac{C_k}{\eta_k^2 \sqrt{T_{k,p}}}
    = \frac{C_k}{\eta_k^2 \sqrt{T_1}}\, r^{-\frac{p}{2}}.
\end{equation}
Then,
\begin{equation*}
    \sum_{p=0}^{\infty}\mathbb{P}(F_{k,p})
    \le \frac{C_k}{\eta_k^2\sqrt{T_1}}\sum_{p=0}^{\infty} r^{-\frac{p}{2}}<\infty,
\end{equation*}
and thus by the Borel-Cantelli Lemma, we have that $\mathbb{P}(F_p\ \text{infinitely often})=0$. Therefore, the algorithm is guaranteed to converge when $T_{p,k}$ is sufficiently large, which yields the first result.

Next, recalling \eqref{eq:failureprob}, it follows that the probability of success is at least $1 - \hat{\delta}$ whenever
\begin{equation*}
    p \;\ge\; \left\lceil\frac{2}{\log r}\,\log\!\Big(\tfrac{C_k}{\eta_k^2\hat\delta\sqrt{T_1}}\Big)\right\rceil := p^*.
\end{equation*}
We note that this is the number of full oracle evaluations, as we require one full evaluation at each iteration of Algorithm \ref{alg:innersolve} to check the stopping criteria. Next, the total number of proximal stochastic gradient iterations needed to satisfy the stopping criteria of Algorithm \ref{alg:innersolve}, with probability at least $1- \hat{\delta}$, is no more than
\begin{equation*}
    T_1 \sum_{p=0}^{p^*} r^p = T_1 \frac{r^{p^*+1}-1}{r-1} \leq \frac{r^2 \frac{C_k^2}{\eta_k^4 \hat{\delta}^2} - T_1}{r-1} = \frac{r^2 C_k^2 - T_1 \eta_k^4 \hat{\delta}^2}{\eta_k^4 \hat{\delta}^2(r-1)}.
\end{equation*}
\end{proof}

If, in addition to Assumptions \ref{assum:stochastic_first_oracle} and \ref{assum:Unbiasedness_and_Bounded_Variance}, we have the stronger assumption \ref{assum:Mean-squared_Smoothness}, then we can obtain an analogous result when PStorm \cite{Li2023} is used as the stochastic method in Algorithm \ref{alg:innersolve}.





\begin{lemma}\label{pstorm_inner_stochastic}(Inner-loop Complexity for PStorm)
Let Assumptions \ref{assum:smoothness}, \ref{assum:stochastic_first_oracle}, \ref{assum:Unbiasedness_and_Bounded_Variance} and \ref{assum:Mean-squared_Smoothness} hold and fix outer iteration $k$ for Algorithm \ref{alg:al}. Let $\phi_k$, $\Omega_k$, and $F_{k,p}$ be defined as in Lemma \ref{lem:aug_lag_grad_upperbound_stochastic5.8}, and let $\sigma_{\L_{\beta}}$ and $L_{\L_{\beta}}$ be defined as in lemma \ref{lem:Lipschitz_stochastic5.7}. Let the conditions of Lemma \ref{lem:aug_lag_grad_upperbound_stochastic5.8} hold. Let $A_k := 1920^\frac{3}{2}3^{\frac{3}{2}}(\frac{24^2}{10})^{\frac{1}{2}}(\max\{\Omega_k, 1\})$ denotes the numerical constant in Lemma $3$ from \cite{Li2023}. Then, with probability 1, the inner-loop solver terminates finitely. In addition, with probability at least $1-\hat\delta$, the total number of evaluations required to satisfy the stopping criteria of Algorithm \ref{alg:innersolve} is no more than
\begin{align*}
   &\max\left\{\left\lceil\frac{r^2 A_k\sigma_{\L_{\beta}}L_{\L_{\beta}} - T_1 \eta_k^3 \hat{\delta}^\frac{3}{2}}{\eta_k^3 \hat{\delta}^\frac{3}{2}(r-1)}\right\rceil,T_1\right\} \text{\quad \; stochastic oracle evaluations,} \\
    &\max\left\{\left\lceil\frac{1}{\log r}\log\left(\frac{A_k\,\sigma_{\L_{\beta_k}}\,L_{\L_{\beta_k}}}{\eta_k^3\,\hat{\delta}^\frac{3}{2}\,T_1}\right)\right\rceil, 1\right\} \quad \text{ full oracle evaluations.}
\end{align*}

\end{lemma}

\begin{proof}
As in the previous Lemma, we take the update rule of Algorithm \ref{alg:innersolve} as $T_{k,p} = r^p T_1$. By Lemma 3 of \cite{Li2023} and the definition
of $A_k$, the output $x_{k,p}$ of the $p$-th PStorm call
satisfies
\[\mathbb{E}\!\left[\operatorname{dist}\!\left(-\nabla_x\L_{\beta_k}(x_{k,p},y_k),\partial g(x_{k,p})\right)^2\right]\leq \left(\frac{A_k\,\sigma_{\L_{\beta_k}}\, L_{\L_{\beta_k}}}{T_{k,p}}\right)^{2/3}.\]

In addition, the condition 
\[\L_{\beta_k}(x_{k,p},y_k)+g(x_{k,p})\leq\L_{\beta_k}(x_k,y_k)+g(x_k)\]
is assumed to hold for every $p$, so the failure event
in $F_{k,p}$ cannot occur. Hence, by Markov's inequality,
\[
\mathbb{P}(F_{k,p})
\leq
\frac{1}{\eta_k^2}
\left(\frac{A_k\,\sigma_{\L_{\beta_k}}\,L_{\L_{\beta_k}}}{T_{k,p}}\right)^{\frac{2}{3}} = \left(\frac{A_k\,\sigma_{\L_{\beta_k}}\,L_{\L_{\beta_k}}}{\eta_k^3T_{1}}\right)^{\frac{2}{3}}r^{\frac{-2p}{3}}.
\]
Since $r>1$, by similar argument as the previous Lemma, we have that $\mathbb{P}\bigl(F_{p}\;\text{infinitely often}\bigr)=0,$
and the inner-loop solver terminates finitely almost surely.

Next, it is sufficient to choose $p$ so that 
\[ \left(\frac{A_k\, \sigma_{\L_{\beta_k}}\,L_{\L_{\beta_k}}}{\eta_k^3T_1r^{p}}\right)^{2/3} \leq \hat{\delta}. \]

Thus, the number of PStorm calls needed that the probability of success is at least $1-\hat{\delta}$ whenever 
\[ p \geq \left\lceil \frac{1}{\log r} \log\left( \frac{A_k\,\sigma_{\L_{\beta_k}}\,L_{\L_{\beta_k}}}{\eta_k^3\hat{\delta}^{\frac{3}{2}}T_1}\right)\right\rceil:= p^*\]

Finally, the total number of PStorm iterations needed to satisfy the stopping criteria of Algorithm \ref{alg:innersolve}, with probability at least $1- \hat{\delta}$, is no more than
\begin{equation*}
    T_1 \sum_{p=0}^{p^*} r^p \leq \frac{r^2 A_k\sigma_{\L_{\beta}}L_{\L_{\beta}} - T_1 \eta_k^3 \hat{\delta}^\frac{3}{2}}{\eta_k^3 \hat{\delta}^\frac{3}{2}(r-1)}.
\end{equation*}

\end{proof}



Finally, we combine the outer-loop complexity from Theorem \ref{thm:complexity_outer5.3} with the inner-loop analysis from Lemma \ref{lem:complexity_inner_stochastic5.9} to obtain the overall complexity of our proposed algorithm in the stochastic setting.

\begin{theorem}\label{thm:overall-epochs} (Overall Complexity for Stochastic Results using Stochastic Proximal Gradient)
Let Assumptions \ref{assum:smoothness}, \ref{assum:regularity}, \ref{assum:stochastic_first_oracle}, and \ref{assum:Unbiasedness_and_Bounded_Variance} hold. Let the conditions of Lemma \ref{lem:aug_lag_grad_upperbound_stochastic5.8} hold. Then, if the objective is stochastic and the constraints are deterministic, with probability at least $1-\delta$, the total number of evaluations required by Algorithm \ref{alg:al} to return an $\epsilon$-approximate feasible solution is 
\begin{equation*}
   \tilde{\mathcal O}\!\left(\delta^{-2}\epsilon^{-6}\right) \text{\; stochastic oracle evaluations and } \mathcal{O}\left(\log(\epsilon^{-1})\log(\delta^{-1}\epsilon^{-1})\right) \text{ full oracle evaluations.}
\end{equation*}

If both the objective and the constraints are stochastic, then with probability at least $1-\delta$, the total number of evaluations required by Algorithm \ref{alg:al} to return an $\epsilon$-approximate feasible solution is 
\begin{equation*}
   \tilde{\mathcal O}\!\left(\delta^{-2}\epsilon^{-7}\right) \text{\; stochastic oracle evaluations and } \mathcal{O}\left(\log(\epsilon^{-1})\log(\delta^{-1}\epsilon^{-1})\right) \text{ full oracle evaluations.}
\end{equation*}
\end{theorem}

\begin{proof}
Let $j$ denote the number of times Algorithm \ref{alg:al} calls the inner solver, and let $j_{\max}$ be the maximum number invokes the inner solver. By Theorem \ref{thm:complexity_outer5.3}, we have $j \leq j_{\max}$ and 
\[
    j_{\max} = \mathcal{O}(\log(\epsilon^{-1})).
\]

Now consider one call of the inner solver, corresponding to outer iteration $k$. By Lemma \ref{lem:complexity_inner_stochastic5.9}, with probability at least $1-\hat{\delta}$, the inner solver terminates in no more than
\begin{equation*}
    \left\lceil\frac{r^2 C_k^2-T_1\eta_k^4 \hat{\delta^2}}{\eta_k^4 \hat{\delta}^2(r-1)}\right\rceil,
\end{equation*}
stochastic oracle evaluations. Since, $\eta_k = \Omega(\epsilon)$, $C_k = \mathcal{O}(\max\{L_k \Omega_k, \sigma_{L_{\beta_k}} \sqrt{L_k \Omega_k}\})$ and $\Omega_k = \mathcal{O}(\log(\epsilon^{-1}))$ by Lemmas  \ref{lem:Lupperbound} and \ref{lem:Llowerbound}, a single call to the inner solver requires at most
\begin{equation*}
    \tilde{\mathcal{O}}\left(\frac{ \max\{L_k^2, \sigma_{L_{\beta_k}}^2 L_k\}}{\epsilon^4 \hat{\delta}^2}\right),
\end{equation*}
with probability at least $1-\hat{\delta}$ stochastic oracle evaluations.

To ensure that the probability of failure for the output of Algorithm \ref{alg:al} across all inner solves is at most $\delta$, 
let $\hat{\delta} := \delta/j_{\max}$. Then, by the union bound,
\[
     \mathbb{P}(\text{All inner solves terminate quickly}) \geq  1-j\hat{\delta}
    \;\geq \; 1 - j_{\max}\hat{\delta} \;=\; 1-\delta.
\]

Therefore, with probability at least $1-\delta$, the total number of stochastic evaluations required by Algorithm \ref{alg:al} is bounded by
\begin{equation*}
    \tilde{\mathcal{O}}\left(\frac{j_{\max}^{3} \max\{L_k^2, \sigma_{L_{\beta_k}}^2 L_k\}}{\epsilon^4 \delta^2}\right) = \tilde{\mathcal{O}}\left(\frac{\max\{L_k^2, \sigma_{L_{\beta_k}}^2 L_k\}}{\epsilon^4 \delta^2}\right),
\end{equation*}

Finally, we consider $\max\{L_k^2, \sigma_{L_{\beta_k}}^2 L_k\}$ separately depending on whether the constraints are deterministic or stochastic. When the constraints are deterministic, $\sigma_{L_\beta}$ term is a constant, so $\max\{L_k^2, \sigma_{L_{\beta_k}}^2 L_k\} = \mathcal O (\beta_k^2) = \mathcal O (\epsilon^{-2})$, and thus the overall complexity is $\tilde{\mathcal O}(\delta^{-2}\epsilon^{-6})$. When the constraints are stochastic, $\sigma_{L_\beta} = \mathcal O (\beta_k) = \mathcal{O}(\epsilon^{-1})$, and therefore the overall complexity is $\tilde{\mathcal O}(\delta^{-2}\epsilon^{-7})$.

Finally, by Lemma \ref{lem:complexity_inner_stochastic5.9} with probability at least $1-\hat{\delta}$, the number of full oracle evaluations required for a single call of the inner solver is no more than
\begin{equation*}
    \left\lceil\frac{2}{\log r}\,\log\!\Big(\tfrac{C_k}{\eta_k^2\hat\delta\sqrt{T_1}}\Big)\right\rceil.
\end{equation*}
Recalling the definition of $\hat{\delta}$ above and $j_{\max}$, it follows that, with probability at least $1-\delta$, the total number of full oracle evaluations is no more than
\begin{equation*}
    j_{\max} \left\lceil\frac{2}{\log r}\,\log\!\Big(\tfrac{C_k}{\eta_k^2\hat\delta\sqrt{T_1}}\Big)\right\rceil = \mathcal{O}(\log(\epsilon^{-1}) \log(\epsilon^{-1} \delta^{-1})),
\end{equation*}
regardless of whether the constraints are deterministic or stochastic.
\end{proof}

\begin{corollary}\label{pstorm_Overall} (Overall Complexity for Stochastic Results using PStorm)



Let Assumptions \ref{assum:smoothness}, \ref{assum:regularity}, \ref{assum:stochastic_first_oracle}, \ref{assum:Unbiasedness_and_Bounded_Variance}, and \ref{assum:Mean-squared_Smoothness} hold. Let the conditions of Lemma \ref{pstorm_inner_stochastic} hold. Then, if the objective is stochastic and the constraints are deterministic, with probability at least $1-\delta$, the total number of evaluations required by Algorithm \ref{alg:al} to return an $\epsilon$-approximate feasible solution is 
\begin{equation*}
   \tilde{\mathcal O}\!\left(\delta^{-\frac{3}{2}}\epsilon^{-4}\right) \text{\; stochastic oracle evaluations and } \mathcal{O}\left(\log(\epsilon^{-1})\log(\delta^{-1}\epsilon^{-1})\right) \text{ full oracle evaluations.}
\end{equation*}

If both the objective and the constraints are stochastic, then with probability at least $1-\delta$, the total number of evaluations required by Algorithm \ref{alg:al} to return an $\epsilon$-approximate feasible solution is 
\begin{equation*}
   \tilde{\mathcal O}\!\left(\delta^{-\frac{3}{2}}\epsilon^{-5}\right) \text{\; stochastic oracle evaluations and } \mathcal{O}\left(\log(\epsilon^{-1})\log(\delta^{-1}\epsilon^{-1})\right) \text{ full oracle evaluations.}
\end{equation*}

\end{corollary}

\section{Numerical Experiments} \label{sec:experiments}
In this section, we present several numerical results. We evaluate the method in both deterministic and stochastic settings. The deterministic experiments investigate the performance of two choices of the Augmented Lagrangian method: whether the penalty parameter is increased consistently or adaptively, and whether a full or short dual step size is used. The stochastic experiments additionally examine the effectiveness of employing Algorithm \ref{alg:innersolve} to solve the inner subproblems versus a simple, fixed epoch strategy more commonly used in the literature \cite{Li2023}.

\subsection{Deterministic Experiments} \label{sec:detexperiments}

We evaluate the four Augmented Lagrangian variants on constrained optimization problems from the PyCUTEst library \cite{PyCUTEst2022}, which covers both convex and nonconvex objectives with equality and inequality constraints. For a problem, let $n$ and $m$ denote the numbers of variables and constraints, respectively, before introducing slack variables. To keep runtime manageable, we select problems satisfying $n +m < 10^3$. Each one-sided inequality constraint of the form $c_i(x) \leq 0$ is reformulated as $c_i(x)+s_i=0, s_i\geq 0,$ and the nonnegativity restrictions on the slack variables are imposed as bound constraints. In addition, we exclude problems without numerical objective values and certain problems which caused memory issues. All experiments were conducted on a workstation equipped with an Intel Core i9-14900KF processor and 32 GB of RAM. The computations were performed using the CPU only, without GPU acceleration. After applying these criteria, the final test set contains $485$ problems.

In the deterministic experiments, we use the L-BFGS-B algorithm as our inner-loop solver, implemented via the SciPy optimization library \cite{scipy2020}. For the outer-loop update rules, the four variants form a $2\times 2$ comparison between the penalty parameter update rule and the dual variable step size. The terms \emph{Full} and \emph{Short} refer to the dual variable step size, whereas \emph{Always} and \emph{Adaptive} refer to the penalty parameter update frequency. 

For the \emph{Full}-step variants, the multiplier update is performed in the same manner as in Algorithm \ref{alg:al}. For the \emph{Short}-step variants, the update is $y_{k+1} = y_k+\alpha_k^{\mathrm{short}} c(x_{k+1})$, with
\[ \alpha_k^{\mathrm{short}} = \min\left( \frac{\|c(x_1)\|\log^2 2} {\|c(x_{k+1})\|(k+1)\log^2(k+2)},\,1\right),\]
as proposed in \cite{Sahin2019} and variants of which are used in \cite{Li2021,Li2023}.

The two \emph{Always} variants increase the penalty parameter at every outer iteration according to $\beta_{k+1}=\gamma\beta_k$. On the other hand, the two \emph{Adaptive} variants implement the penalty parameter update as in Algorithm \ref{alg:al}.

To be more specific, we list them in order:

\begin{enumerate}[leftmargin=2em,itemsep=2pt,topsep=2pt]
  \item \textbf{Always–Full:} Update the dual variable every iteration using the full step, taking the penalty parameter $\beta$ as the dual step size; cap dual components by $10^9$ for numerical stability.
  \item \textbf{Always–Short:} Update every iteration using the short step size as in Sahin et al.\ \cite{Sahin2019}, which bounds the dual variable and yields favorable complexity guarantees.
  \item \textbf{Adaptive–Full (Algorithm \ref{alg:al}):} Use the full step ($\beta_k$ as dual step size) but update the dual variable only when constraint violation fails to decrease sufficiently.
  \item \textbf{Adaptive–Short:} As in Always–Short method, but update the dual variable using the adaptive scheme.
\end{enumerate}







We compare the methods using the performance profiles from Dolan and Mor\'e \cite{dolan2002benchmarking}.  Let \(\mathcal{P}\) denote the set of test problems and let \(\mathcal{S}\) denote the set of solver methods. For problem \(pr\in\mathcal{P}\) and solver \(s\in\mathcal{S}\), let
$t_{pr,s}$ be the number of gradient evaluations required to satisfy the termination condition for Algorithm \ref{alg:al}. For each problem solved by at least one method, define \[ t_{pr}^* := \min_{s\in\mathcal{S}} t_{pr,s}.\]

The performance ratio is defined as 
\begin{equation*}
    r_{pr,s} :=
    \begin{cases}
        \displaystyle
        \frac{t_{pr,s}}{t_{pr}^*},
        &
        
        \text{ if solve} \;s \;\text{solved problem } pr
        ;
        \\[10pt]
        100,
        &
        \text{otherwise}.
    \end{cases}
\end{equation*}
Thus, an unsuccessful run receives performance ratio of $100$ and does not count as solved for any performance factor. For $\alpha\geq 1$, define the performance profile of solver $s$ as
\begin{equation*}
\label{eq:performance-profile}
    \rho_s(\alpha)
    :=
    \frac{1}{|\mathcal{P}|}
    \left|
        \left\{
            pr\in\mathcal{P}:
            r_{pr,s}\leq\alpha
        \right\}
    \right|.
\end{equation*}
The value \(\rho_s(1)\) is the fraction of problems for which solver
\(s\) is the most efficient method, up to ties. Larger values of
\(\alpha\) measure the robustness of the solver over the test set.
We display the profiles for \(1\leq\alpha < 100\), which shows the proportion of problems each solver can solve up to 99 times of the shortest number of gradient evaluations.


Figure~\ref{fig:perfprofiles}  reports the performance profiles for four combinations of the termination tolerance and penalty-update parameters. The proposed Adaptive-Full method attains the highest performance profile throughout the displayed range in all four panels, indicating the best overall combination of efficiency and robustness on this test set. Its advantage becomes more pronounced when penalty increases are triggered more frequently and when the penalty growth factor is larger, suggesting that the method remains effective across a range of parameter choices. The Always-Full method is competitive under the looser termination tolerance, but its performance deteriorates under more aggressive penalty growth. The two short-step variants solve fewer problems within the evaluation budget and are significantly outperformed by the methods using a full dual stepsize across all experiments. Overall, these results suggest that combining a full dual step with an adaptive penalty-update rule provides a better balance between feasibility progress and inner-subproblem conditioning than either increasing the penalty parameter at every iteration or using a short dual step.

\begin{figure}[h]
  \centering

\resizebox{0.81\textwidth}{!}{%
\begin{minipage}{\textwidth}
  
  \begin{subfigure}{0.48\textwidth}
    \centering
    \includegraphics[width=\textwidth]{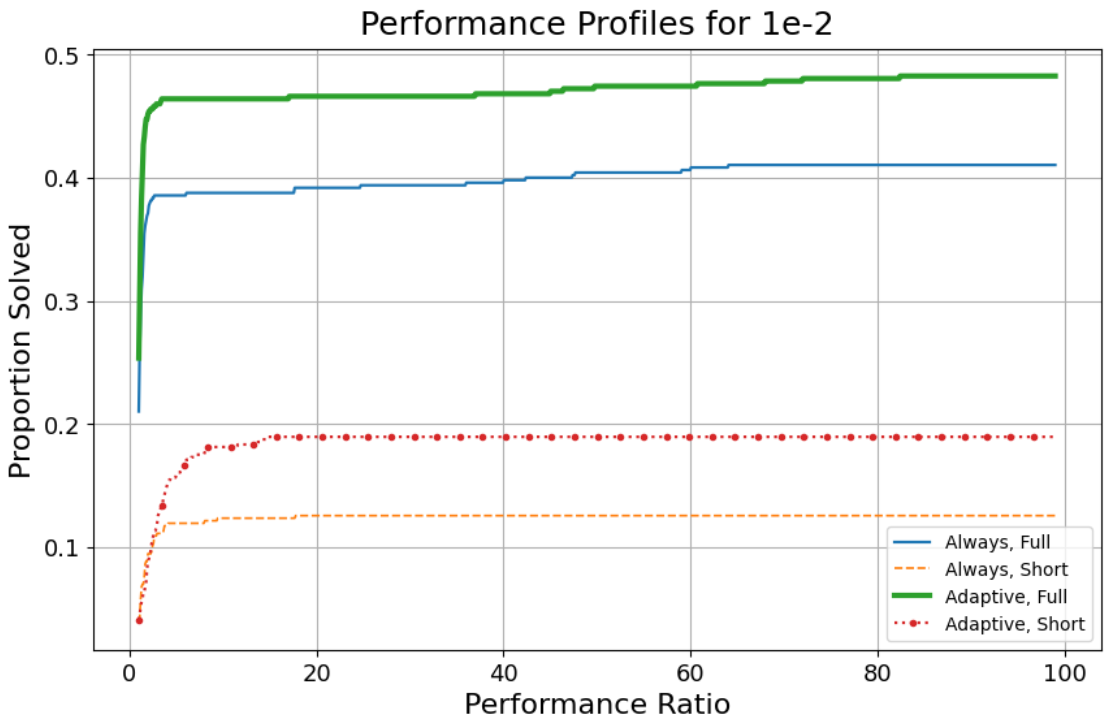}
    \caption{Tolerance $10^{-2}$ with $\tau = 0.8$ and $\gamma = 1.5$}
  \end{subfigure}\hfill
  \begin{subfigure}{0.48\textwidth}
    \centering
    \includegraphics[width=\textwidth]{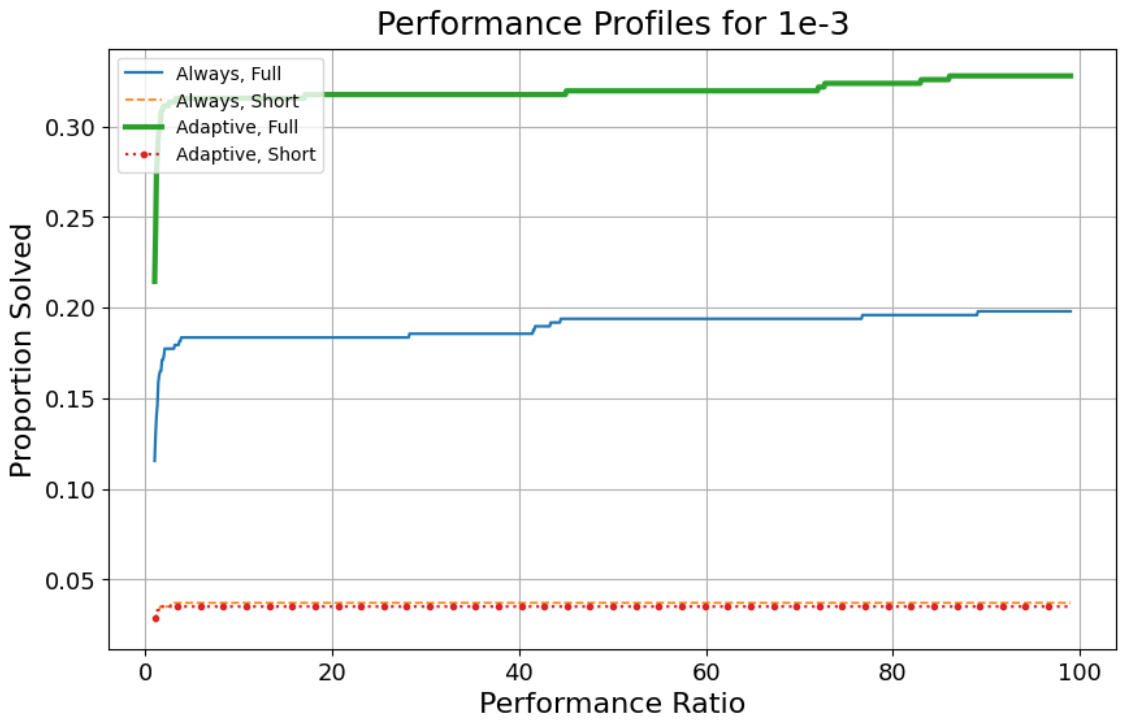}
    \caption{Tolerance $10^{-3}$ with $\tau = 0.8$ and $\gamma = 1.5$}
  \end{subfigure}

  \vspace{0.5em}

  \begin{subfigure}{0.48\textwidth}
    \centering
    \includegraphics[width=\textwidth]{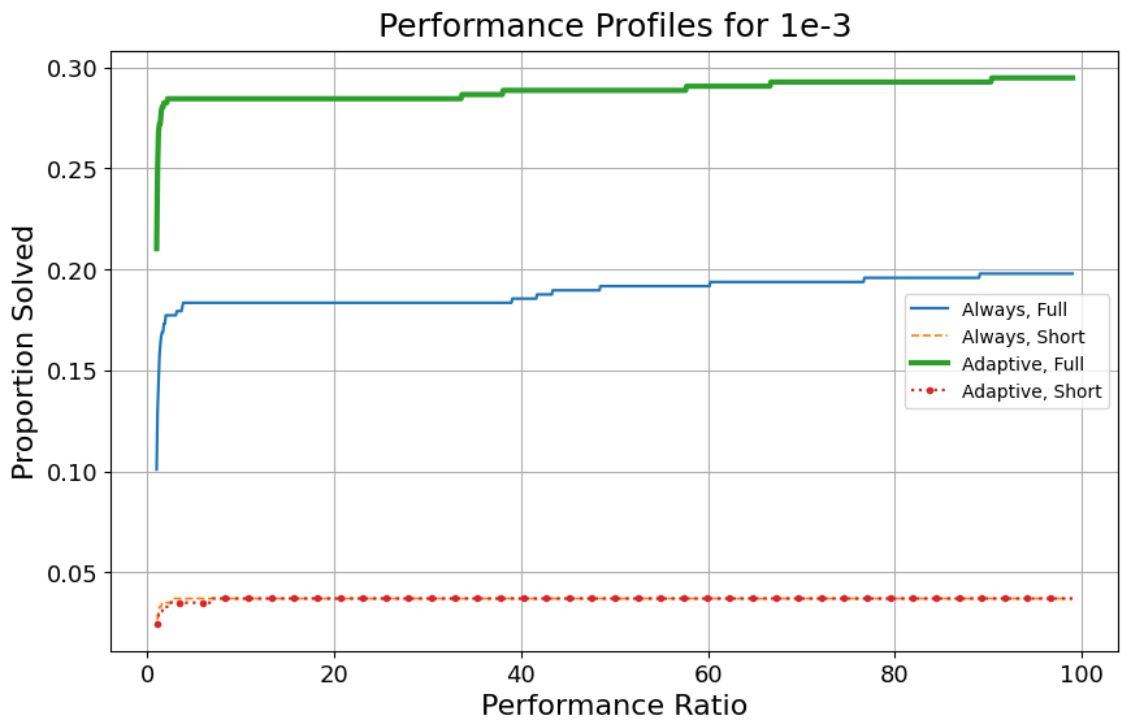}
    \caption{Tolerance $10^{-3}$ with $\tau = 0.5$ and $\gamma = 1.5$}
  \end{subfigure}\hfill
  \begin{subfigure}{0.48\textwidth}
    \centering
    \includegraphics[width=\textwidth]{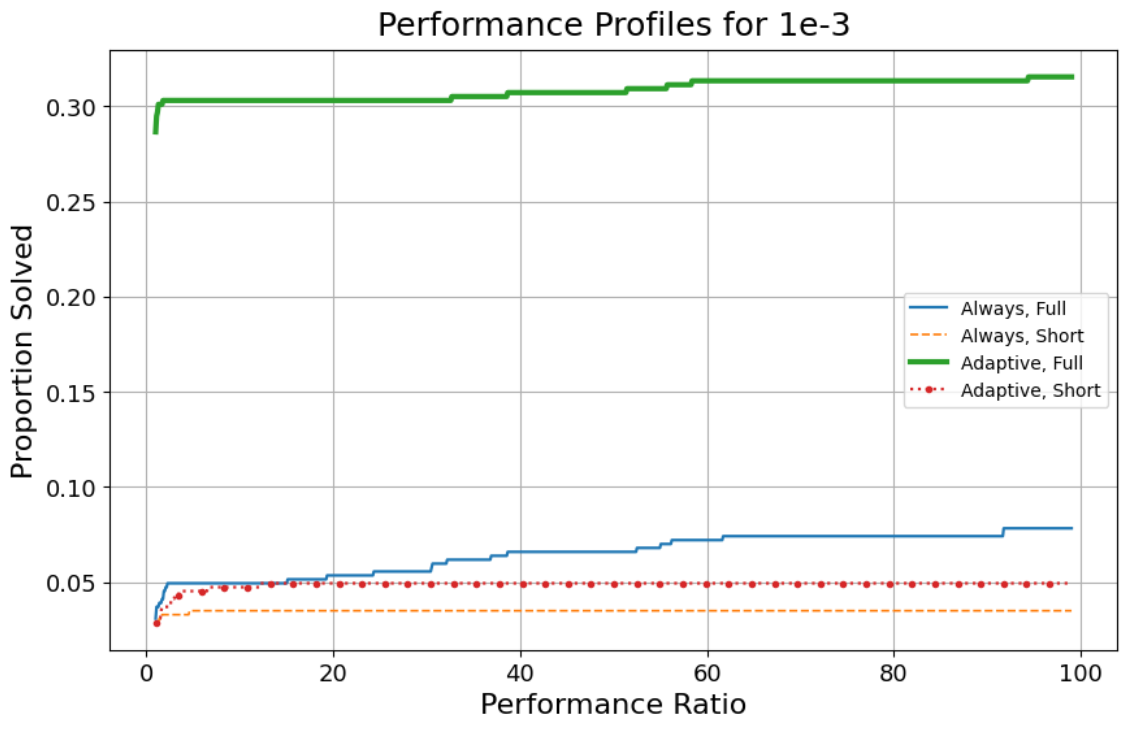}
    \caption{Tolerance $10^{-3}$ with $\tau = 0.8$ and $\gamma = 2$}
  \end{subfigure}

  \end{minipage}
}
\caption{Performance profiles $\rho_s(\alpha)$ for four dual variable update variants. The figures compare the performance profiles of four dual-variable update variants under different stopping tolerances and penalty growth factors. Panels (a)--(b) correspond to two stopping tolerances $10^{-2}$, and $10^{-3}$ respectively. (c) studies a more frequent penalty growth $\tau = 0.5$ with tolerance $10^{-3}$. (d) studies a more aggressive penalty growth factor $\gamma = 2$ with tolerance $10^{-3}$.}
  \label{fig:perfprofiles}

\end{figure}




Table \ref{tbl:final_beta} below lists a subset of the test problems on which the Always-Full method failed to reach the termination tolerances, while Algorithm \ref{alg:al} (Adaptive-Full) reached an approximate feasible stationary point. From the table, two patterns stand out: (i) the last (before termination) penalty parameter produced by the Always update method is typically several orders of magnitude larger than the penalty parameter under the Adaptive rule (see the column $\beta_{\text{Always}}/\beta_{\text{Adaptive}}\gg 1$), and (ii) the Always-Full method generally accumulates more gradient evaluations yet still fail to converge to the desired accuracy, whereas the Adaptive-Full method succeeds with generally many fewer gradient evaluations. These statistics indicate a systematic difference in how the two methods shape the inner subproblems. This phenomenon is likely due to the conditioning generated by the Augmented Lagrangian subproblem. When the penalty parameter grows aggressively, the curvature introduced by $\tfrac{\beta}{2}\,\|c(x)\|^{2}$ dominates, making the inner problem extremely difficult to solve. The Adaptive update method, by increasing the penalty parameter only when the constraint violation fails to decrease sufficiently, avoids this aggressive penalty parameter growth, which keeps the inner problems reasonably conditioned, and thereby achieves the stopping criteria on these instances. This evidence supports the incorporation of an adaptive penalty parameter update for practical Augmented Lagrangian methods.

\begin{table}[htbp]
\centering
\scriptsize
\setlength{\tabcolsep}{4pt}
\renewcommand{\arraystretch}{1.05}

\resizebox{\textwidth}{!}{%
\begin{tabular}{lccccc}
\toprule
& \multicolumn{2}{c}{Always-Full} & \multicolumn{2}{c}{Algorithm \ref{alg:al}} & \\
\cmidrule(lr){2-3}\cmidrule(lr){4-5}
Problem & Gradient Evaluations & Last $\beta$ & Gradient Evaluations & Last $\beta$ & $\beta_{\text{Always}}/\beta_{\text{Adaptive}}$ \\
\midrule
ALLINITC & 1703 & $3.41376\times 10^{22}$ & 541 & 83966.6 & $4.06562\times 10^{17}$ \\
BROWNALE & 1340 & $2.75095\times 10^{11}$ & 242 & 191.751 & $1.43465\times 10^{9}$ \\
BT4 & 1287 & $4.49548\times 10^{21}$ & 275 & 127.834 & $3.51665\times 10^{19}$ \\
BT6 & 1265 & $2.67046\times 10^{17}$ & 379 & 55977.7 & $4.77058\times 10^{12}$ \\
BT7 & 2503 & $3.36344\times 10^{25}$ & 762 & 7371.55 & $4.56273\times 10^{21}$ \\
BYRDSPHR & 1320 & $2.99699\times 10^{21}$ & 148 & 16.8341 & $1.78031\times 10^{20}$ \\
CB2 & 872 & $2.95282\times 10^{24}$ & 361 & 1456.11 & $2.02788\times 10^{21}$ \\
CB3 & 881 & $4.42922\times 10^{24}$ & 273 & 287.627 & $1.53992\times 10^{22}$ \\
CHACONN1 & 852 & $2.55411\times 10^{26}$ & 219 & 191.751 & $1.33199\times 10^{24}$ \\
CHACONN2 & 884 & $3.36344\times 10^{25}$ & 244 & 127.834 & $2.63110\times 10^{23}$ \\
\bottomrule
\end{tabular}%
}

\caption{First ten alphabetically ordered test problems with gradient evaluations and final penalty parameters where the Always-Full method fails but Algorithm \ref{alg:al} succeeds with tolerance $10^{-3}$, $\tau = 0.8$, and $\gamma = 1.5$.}
\label{tbl:final_beta}
\end{table}

\subsection{Stochastic Experiments} \label{subsec:stochasticexperiments}

We next evaluate the method in the stochastic setting. Following \cite{Li2023} and \cite{Yan2022}, we consider the finite sum Neyman-Pearson classification problem on the normalized Spambase dataset \cite{kelly_uci_ml_repository}. Let $\{a_i^{+}\}_i^{n^+}$ and $\{a_i^{-}\}_i^{n^-}$ denote the positive and negative class feature vectors, respectively. We solve 

\[
\begin{aligned}
\min_{x \in \mathbb{R}^n}\quad 
& \frac{1}{n^{+}}\sum_{i=1}^{n^{+}} 
  \log\!\Bigl(1+\exp\!\bigl(-x^{T}a_i^{+}\bigr)\Bigr) \\[4pt]
\text{s.t.}\quad 
& \frac{1}{n^{-}}\sum_{i=1}^{n^{-}}
  \log\!\Bigl(1+\exp\!\bigl(x^{\top}a_i^{-}\bigr)\Bigr) -\hat{c} \leq 0. 
\end{aligned}
\]
 The objective function is a logistic surrogate for false-negative error, whereas the constraint is a logistic surrogate to control false-positive error. Here, we consider a range of $\hat{c} \in \{0.6,0.7,0.8,0.9\}$. The dataset is divided into $80\%$ training and $20\%$ test observations using random sampling. Each feature is standardized using the mean and standard deviation computed from the training data and the same transformation is applied to the test set. Our results are averaged over 20 trials, for which the same randomly chosen $x_0$ is used, while $y_0 = 0$. For shorthand while plotting, we define \begin{equation*}
     R = \dist \left( -\nabla_x L(x, y), \partial g(x) \right) + \|c(x)\|.
 \end{equation*}
 We handle the inequality constraint via adding a slack variable $s$ and the constraint $s\geq 0$ and use projected stochastic gradient and PStorm \cite{Li2023} algorithm to solve the stochastic subproblem.

\begin{figure}[h]
    \centering

    \begin{subfigure}[b]{0.48\textwidth}
        \centering
        \includegraphics[width=\textwidth,height=4.14cm]{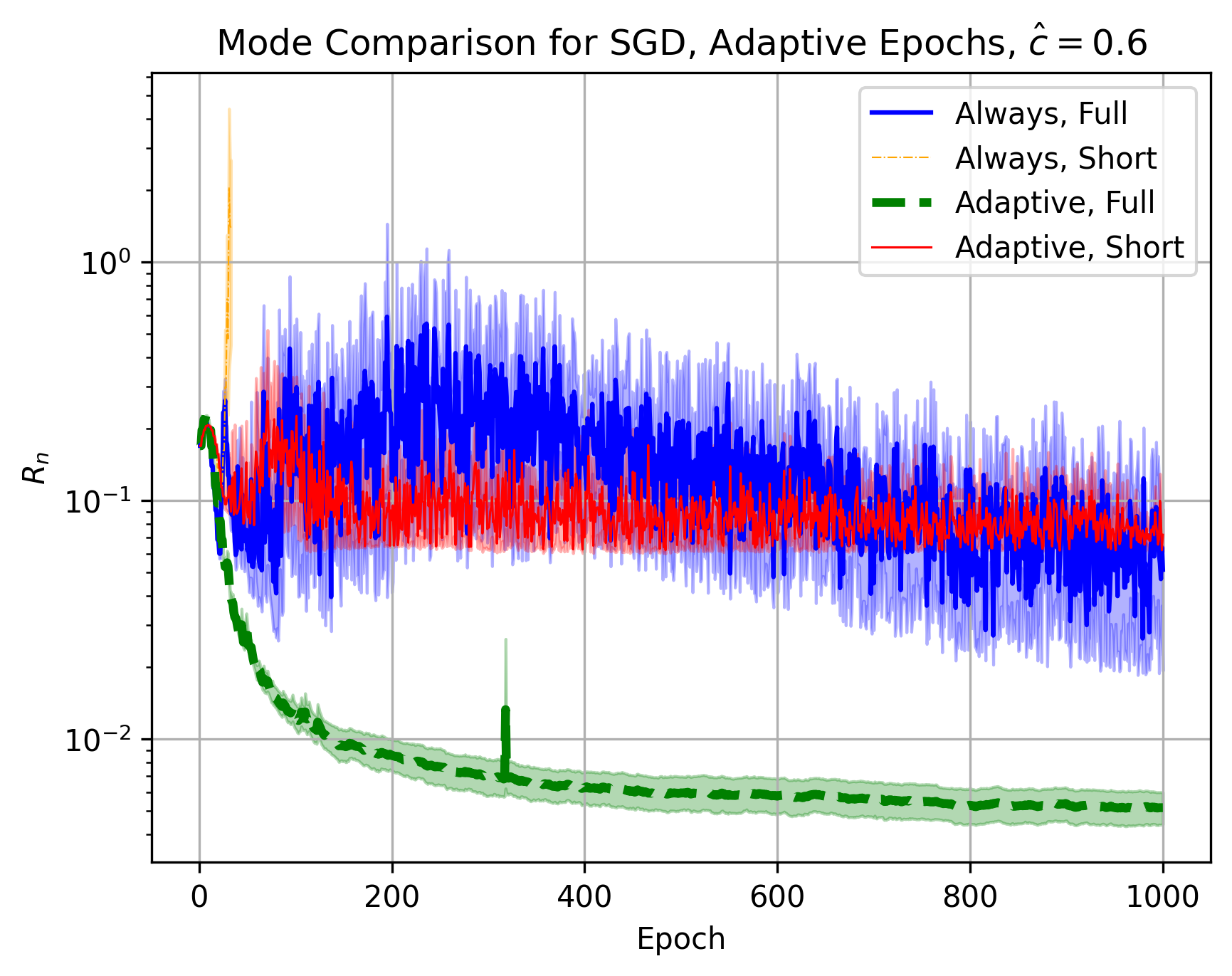}
        \caption{SGD, Adaptive epochs, Accuracy, $\hat{c} = 0.6$.}
        \label{fig:sgd_adap_acc_0.6}
    \end{subfigure}
    \hfill
    \begin{subfigure}[b]{0.48\textwidth}
        \centering
        \includegraphics[width=\textwidth,height=4.14cm]{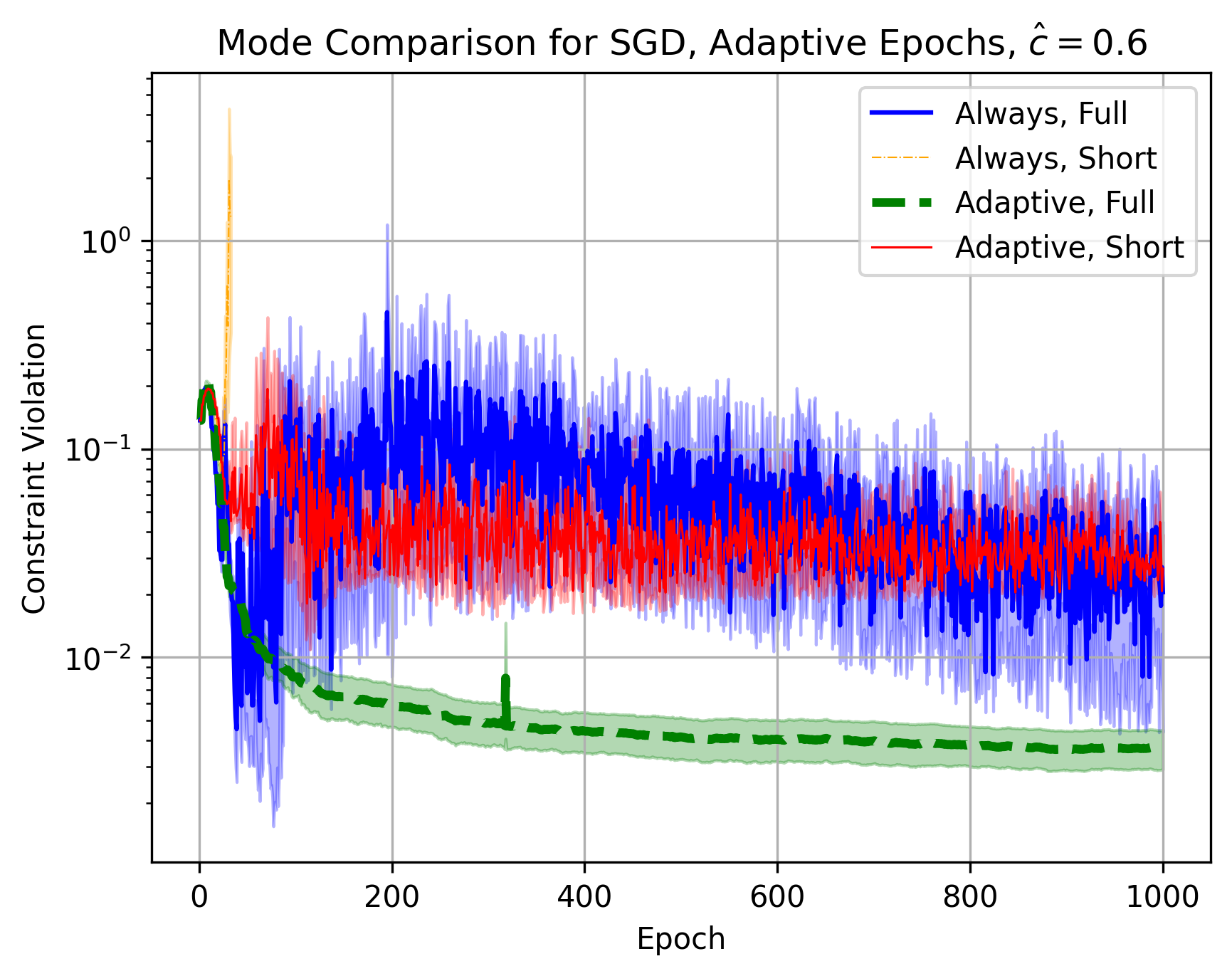}
        \caption{SGD, Adaptive epochs, Constr. violation, $\hat{c} = 0.6$.}
        \label{fig:sgd_adap_constr_0.6}
    \end{subfigure}

    \vspace{0.4cm} 

    \begin{subfigure}[b]{0.48\textwidth}
        \centering
        \includegraphics[width=\textwidth,height=4.14cm]{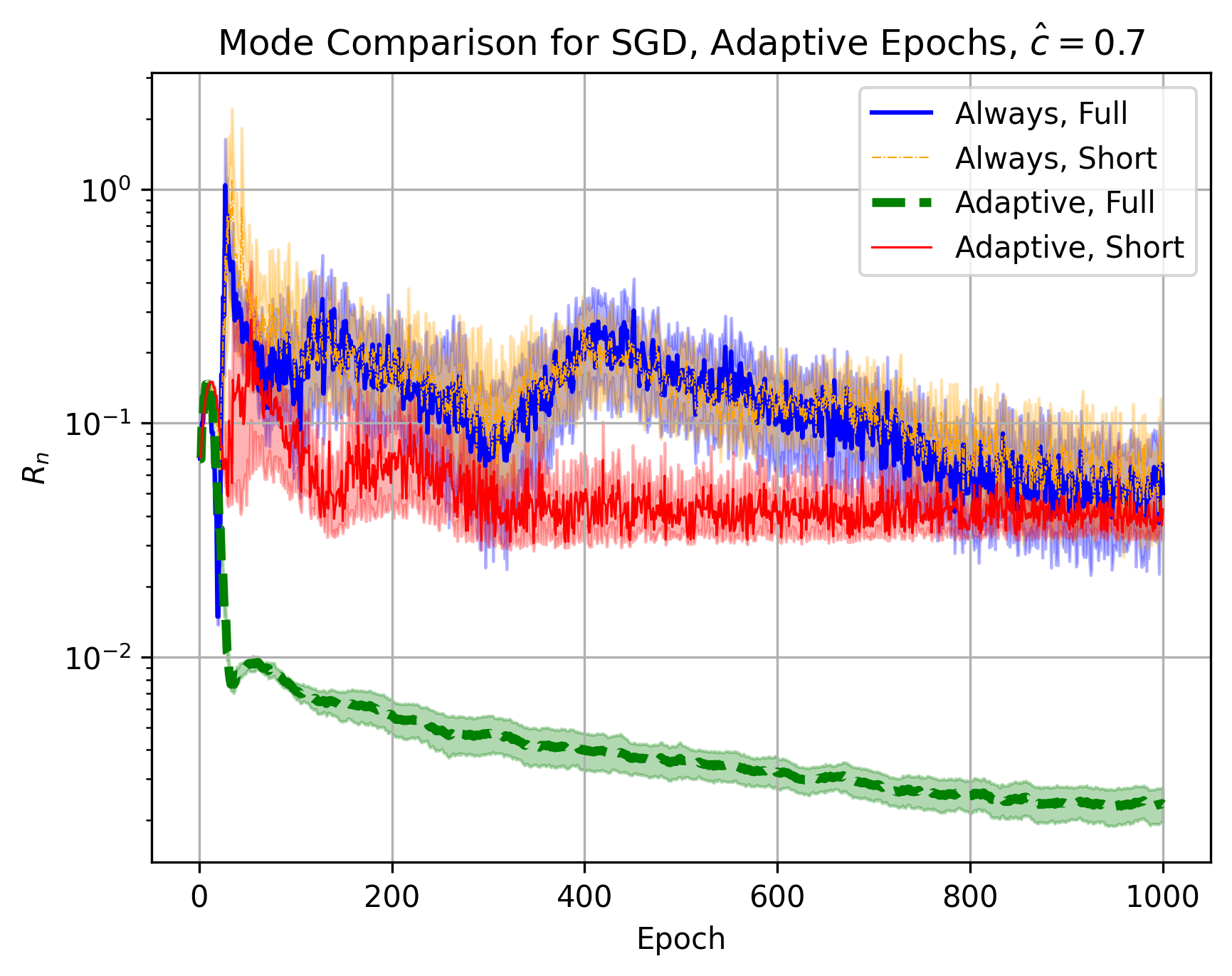}
        \caption{SGD, Adaptive epochs, Accuracy, $\hat{c} = 0.7$.}
        \label{fig:sgd_adap_acc_0.7}
    \end{subfigure}
    \hfill
    \begin{subfigure}[b]{0.48\textwidth}
        \centering
        \includegraphics[width=\textwidth,height=4.14cm]{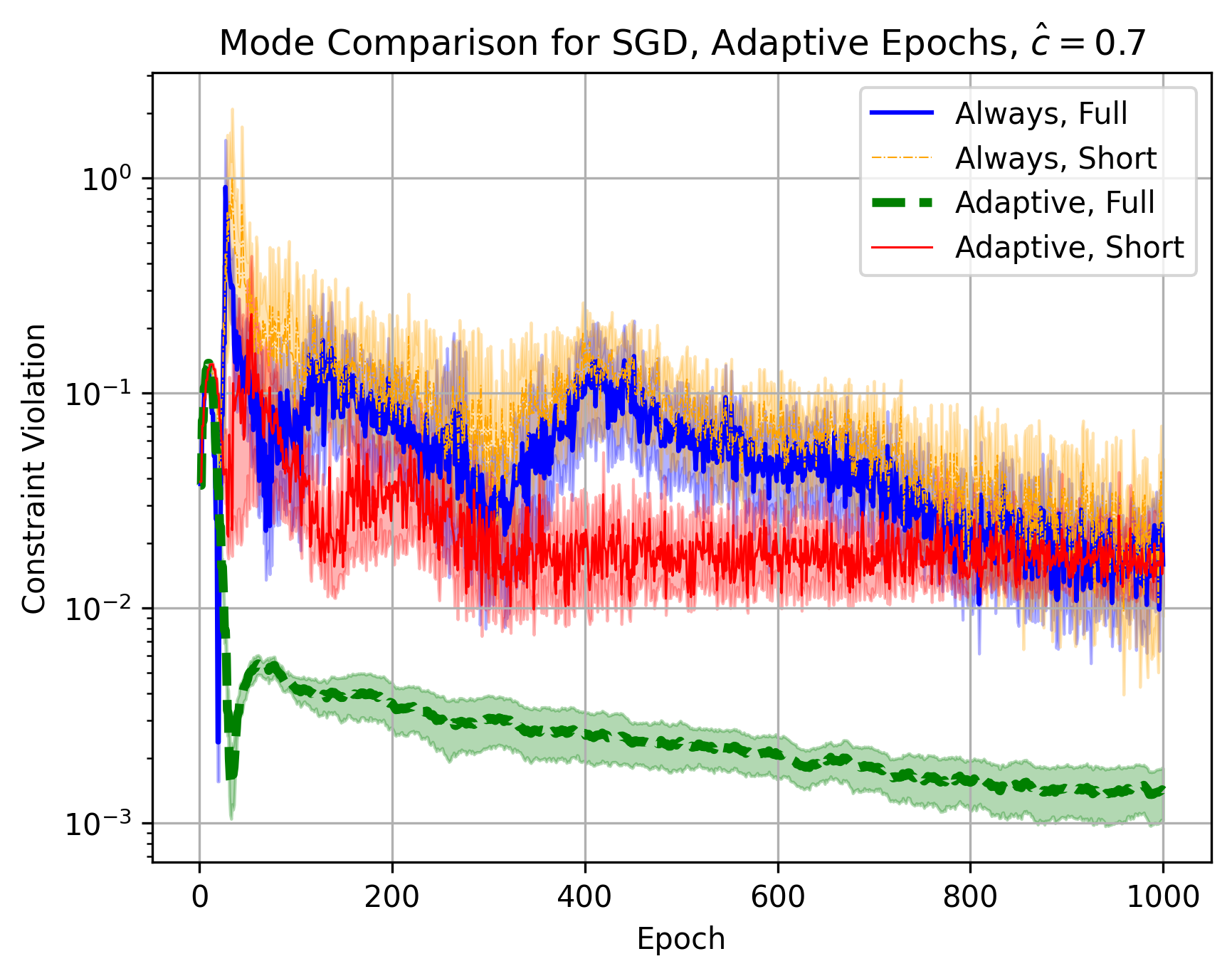}
        \caption{SGD, Adaptive epochs, Constr. violation, $\hat{c} = 0.7$.}
        \label{fig:sgd_adap_constr_0.7}
    \end{subfigure}

    \hfill

      \begin{subfigure}[b]{0.48\textwidth}
        \centering
        \includegraphics[width=\textwidth,height=4.14cm]{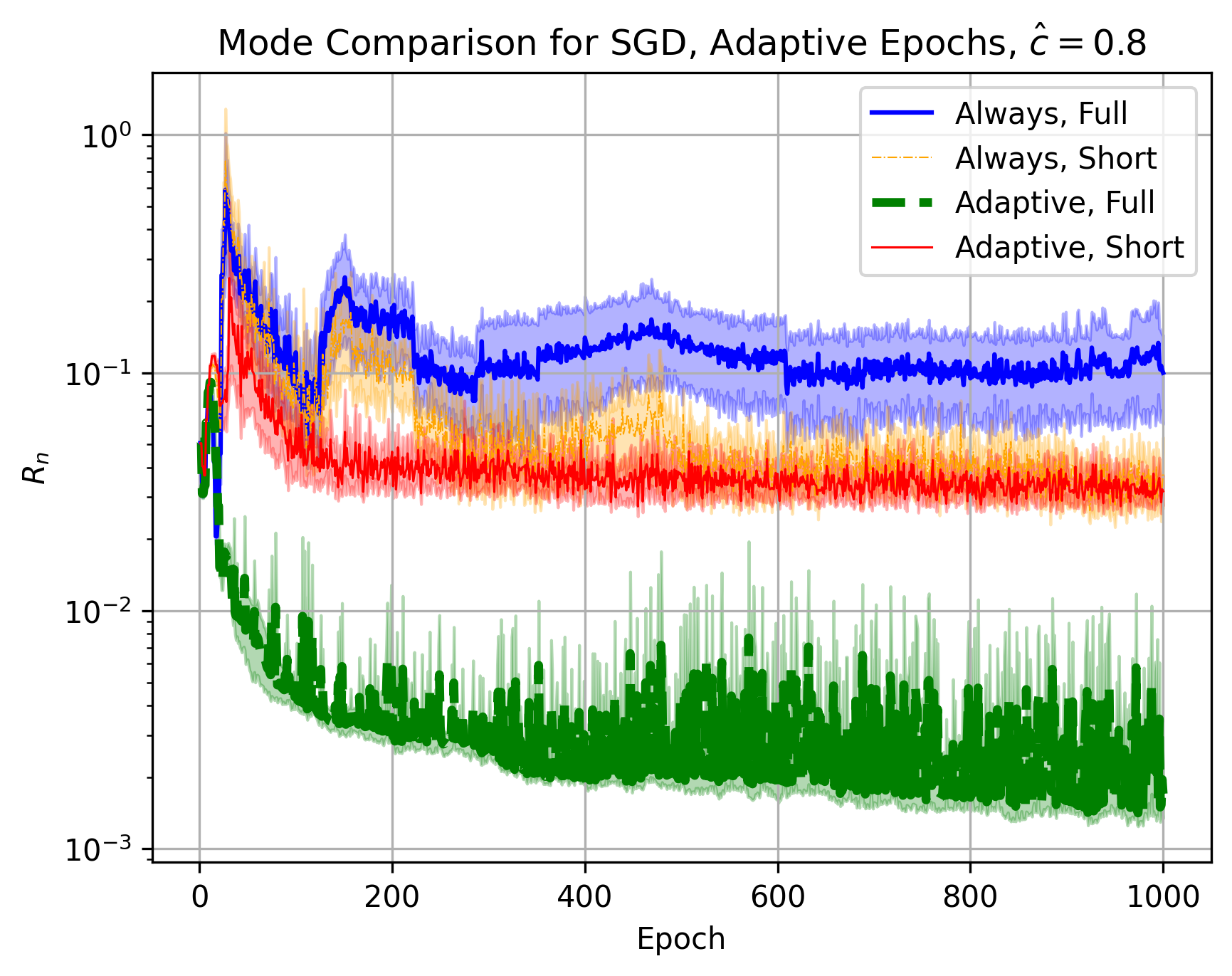}
        \caption{SGD, Adaptive epochs, Accuracy, $\hat{c} = 0.8$.}
        \label{fig:sgd_adap_acc_0.8}
    \end{subfigure}
    \hfill
    \begin{subfigure}[b]{0.48\textwidth}
        \centering
        \includegraphics[width=\textwidth,height=4.14cm]{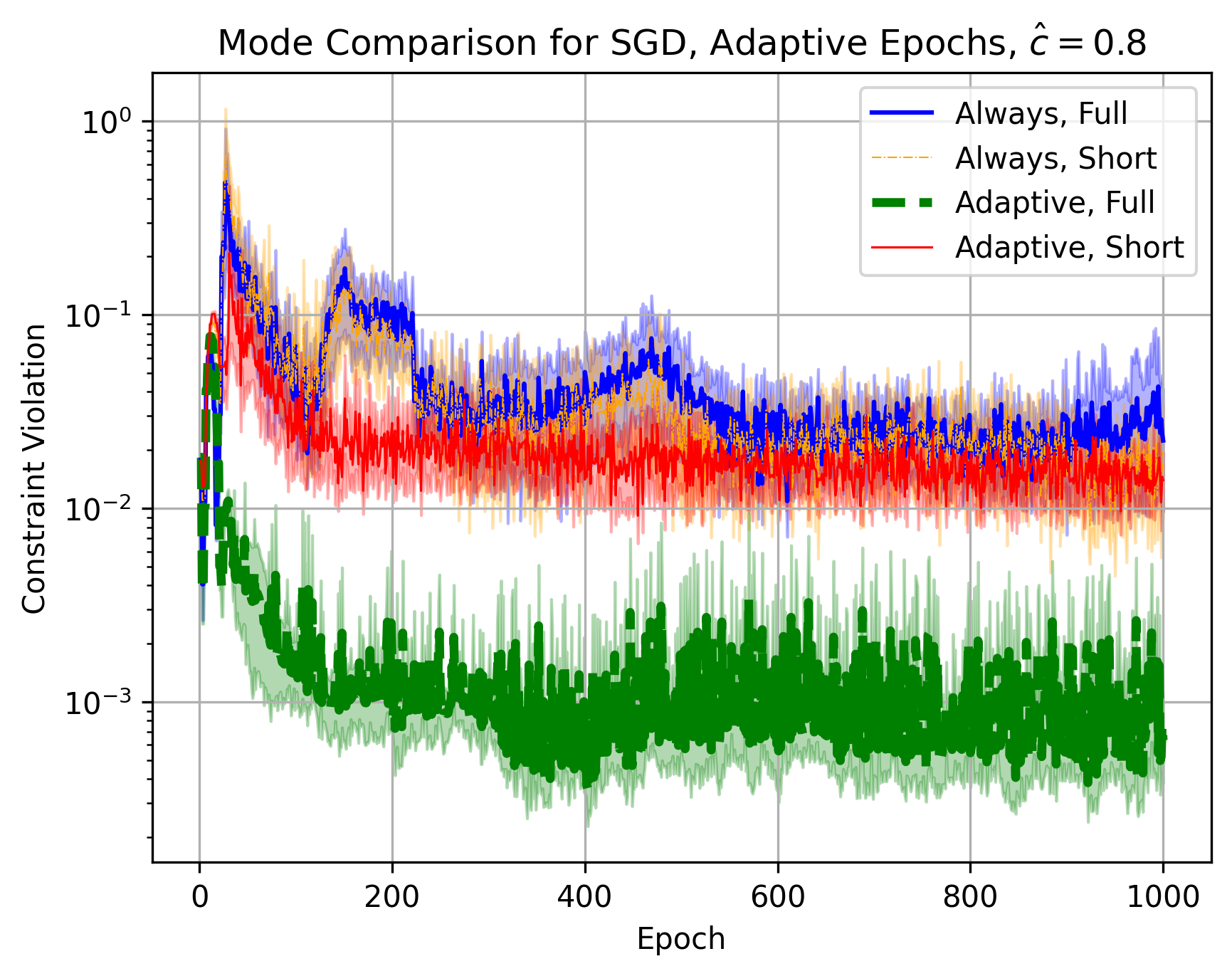}
        \caption{SGD, Adaptive epochs, Constr. violation, $\hat{c} = 0.8$.}
        \label{fig:sgd_adap_constr_0.8}
    \end{subfigure}

    \vspace{0.4cm} 

    \begin{subfigure}[b]{0.48\textwidth}
        \centering
        \includegraphics[width=\textwidth,height=4.14cm]{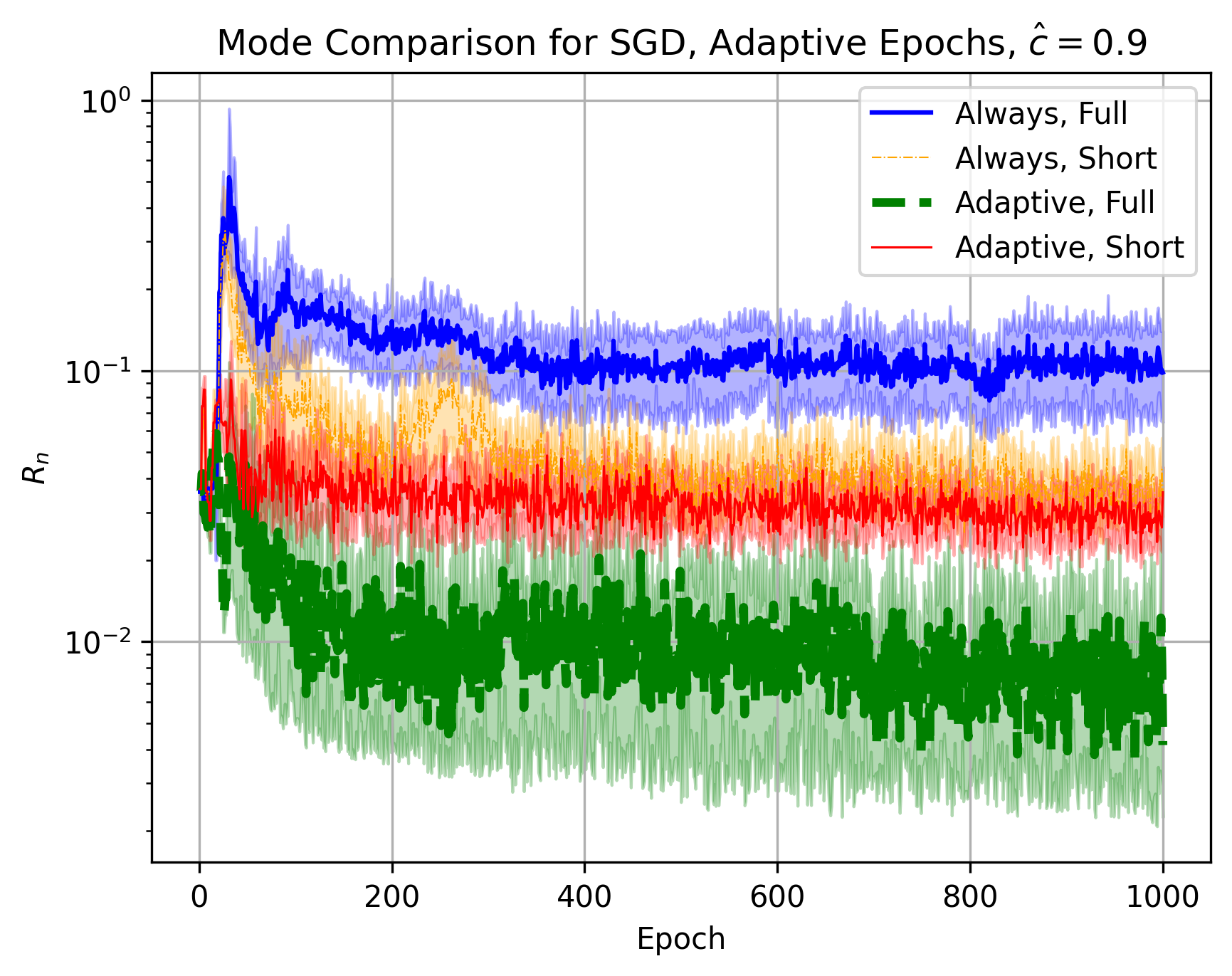}
        \caption{SGD, Adaptive epochs, Accuracy, $\hat{c} = 0.9$.}
        \label{fig:sgd_adap_acc_0.9}
    \end{subfigure}
    \hfill
    \begin{subfigure}[b]{0.48\textwidth}
        \centering
        \includegraphics[width=\textwidth,height=4.14cm]{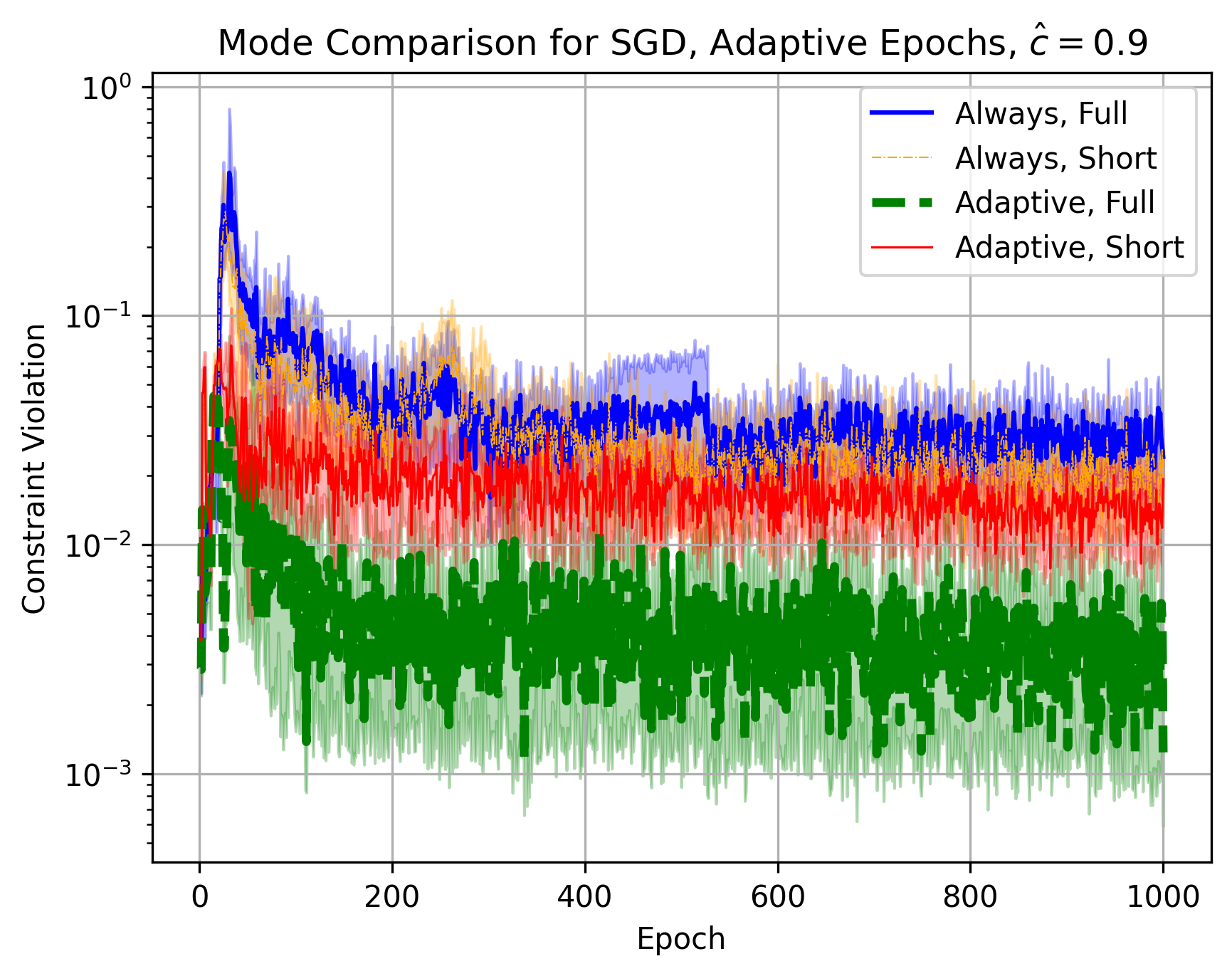}
        \caption{SGD, Adaptive epochs, Constr. violation, $\hat{c} = 0.9$.}
        \label{fig:sgd_adap_constr_0.9}
    \end{subfigure}
    
    \caption{Comparison of different outer loop modes for SGD with adaptive inner epochs varying $\hat{c}$. Average results over 20 trials with 95\% confidence intervals.}
    \label{fig:modes_comparison_all}
\end{figure}

\begin{figure}[h]
    \centering

    \begin{subfigure}[b]{0.48\textwidth}
        \centering
        \includegraphics[width=\textwidth,height=4.14cm]{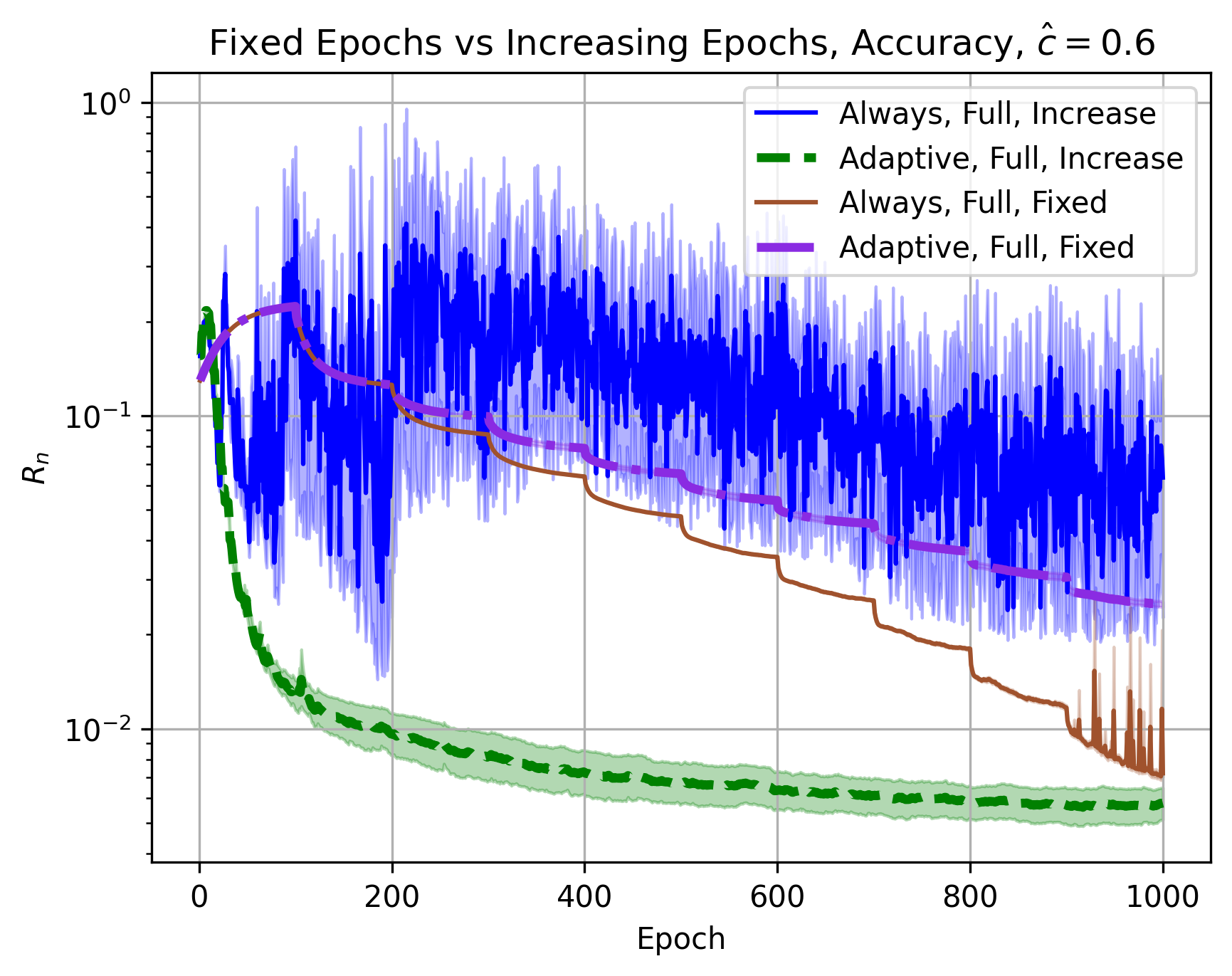}
        \caption{SGD, Adapt vs. Fixed, Accuracy, $\hat{c} = 0.6$.}
        \label{fig:sgd_comp_acc_0.6}
    \end{subfigure}
    \hfill
    \begin{subfigure}[b]{0.48\textwidth}
        \centering
        \includegraphics[width=\textwidth,height=4.14cm]{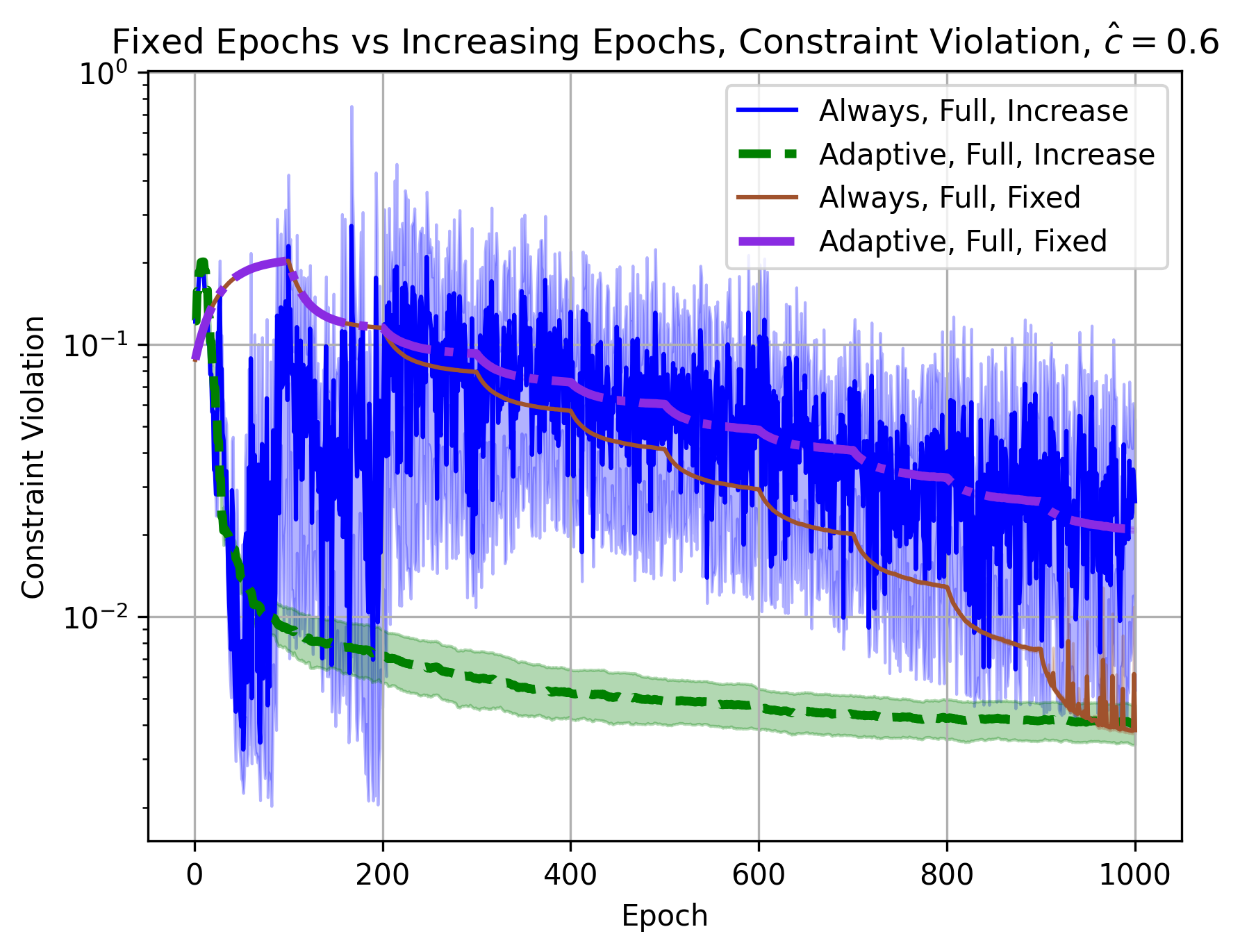}
        \caption{SGD, Adapt vs. Fixed, Constr. Violation, $\hat{c} = 0.6$.}
        \label{fig:sgd_comp_constr_0.6}
    \end{subfigure}

    \vspace{0.4cm} 

    \begin{subfigure}[b]{0.48\textwidth}
        \centering
        \includegraphics[width=\textwidth,height=4.14cm]{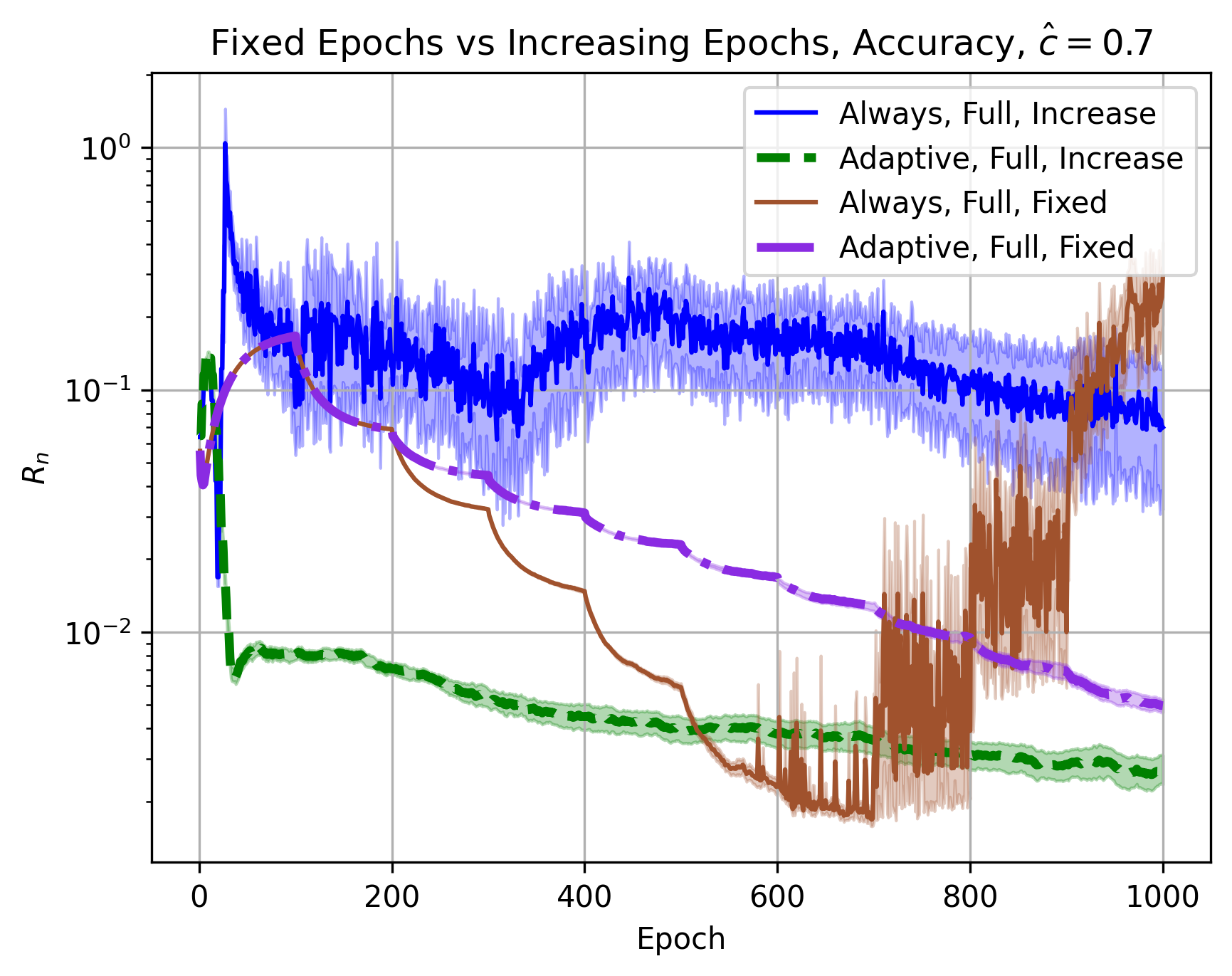}
        \caption{SGD, Adapt vs. Fixed, Accuracy, $\hat{c} = 0.7$.}
        \label{fig:sgd_comp_acc_0.7}
    \end{subfigure}
    \hfill
    \begin{subfigure}[b]{0.48\textwidth}
        \centering
        \includegraphics[width=\textwidth,height=4.14cm]{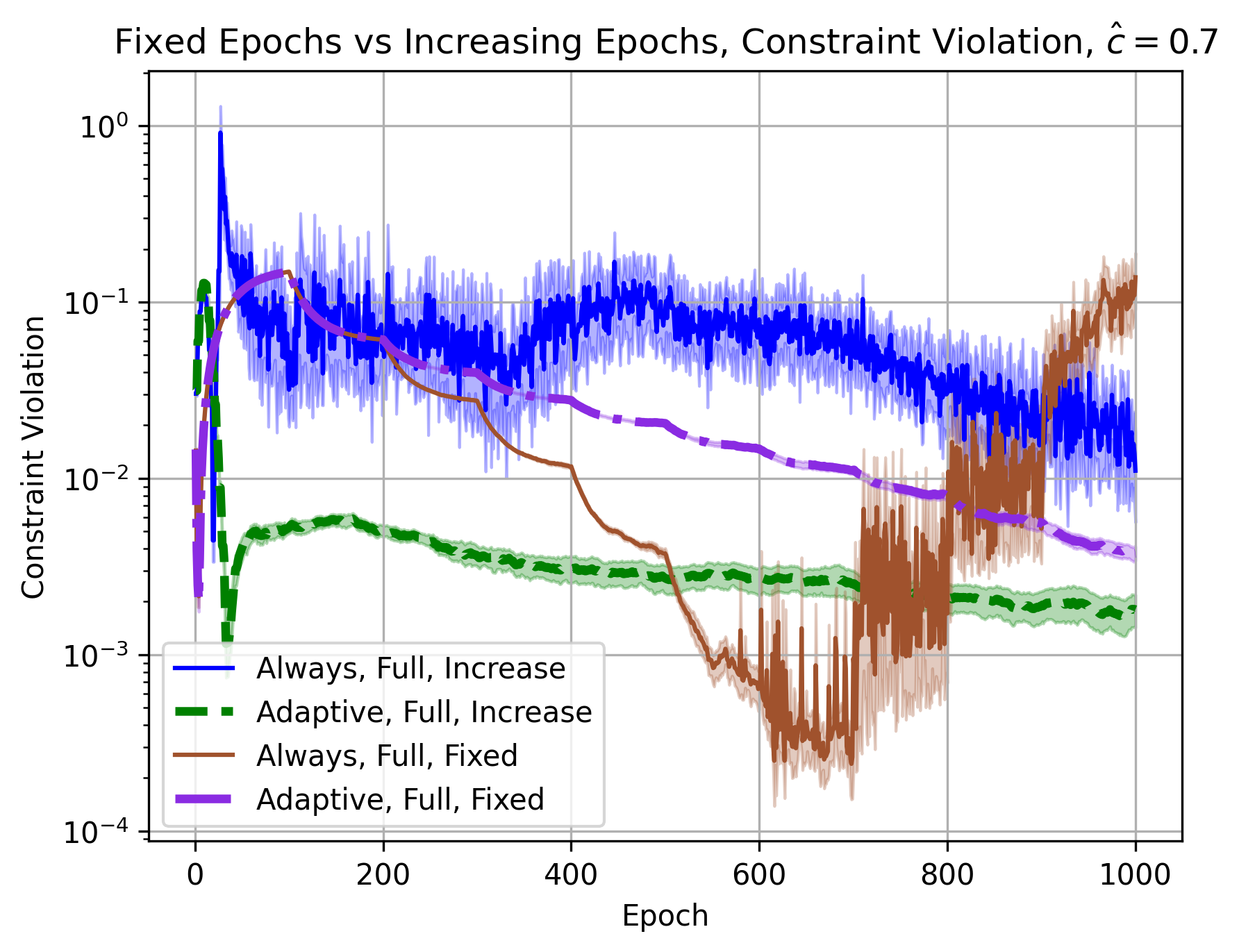}
        \caption{SGD, Adapt vs. Fixed, Constr. Violation, $\hat{c} = 0.7$.}
        \label{fig:sgd_comp_constr_0.7}
    \end{subfigure}

    \hfill

      \begin{subfigure}[b]{0.48\textwidth}
        \centering
        \includegraphics[width=\textwidth,height=4.14cm]{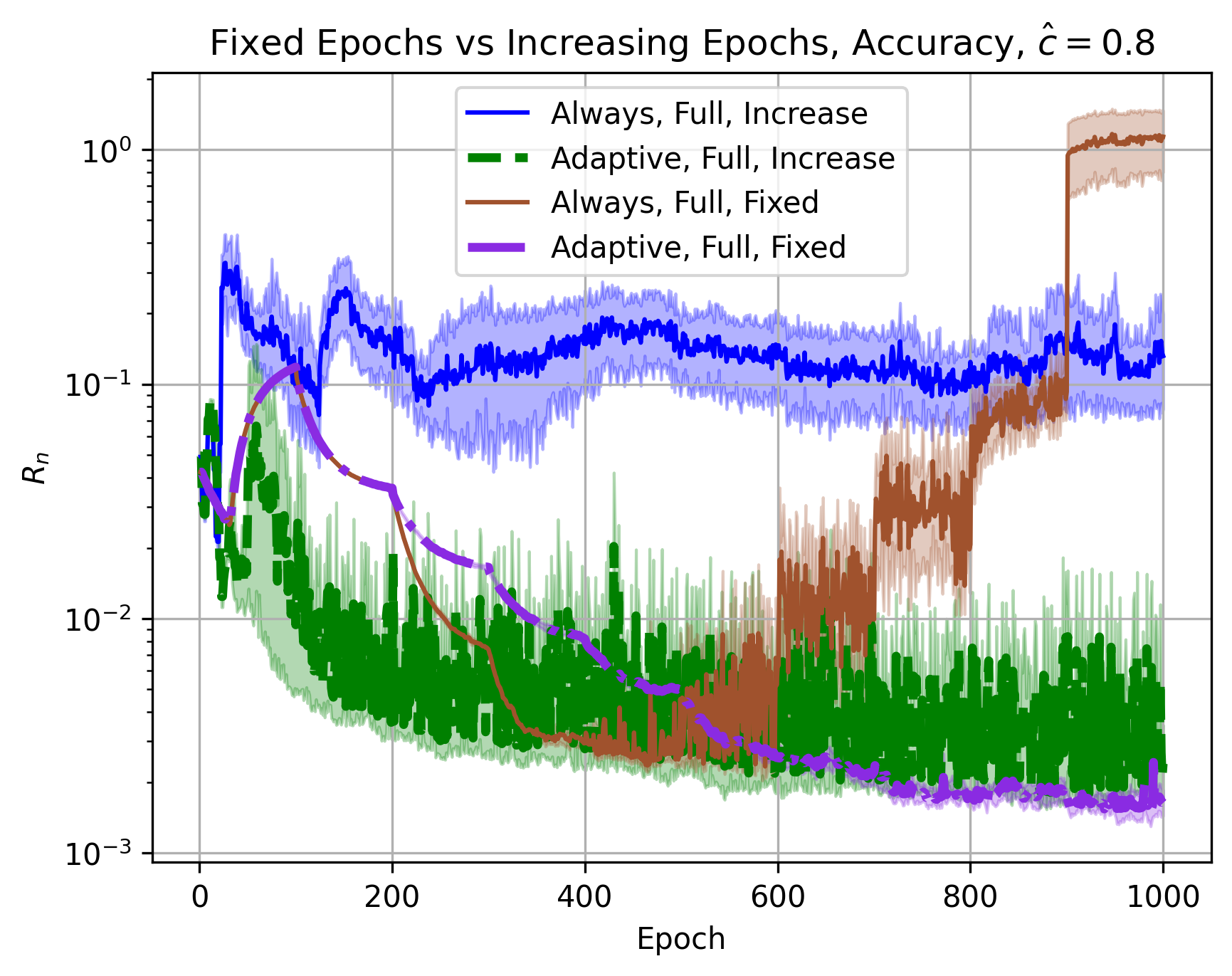}
        \caption{SGD, Adapt vs. Fixed, Accuracy, $\hat{c} = 0.8$.}
        \label{fig:sgd_comp_acc_0.8}
    \end{subfigure}
    \hfill
    \begin{subfigure}[b]{0.48\textwidth}
        \centering
        \includegraphics[width=\textwidth,height=4.14cm]{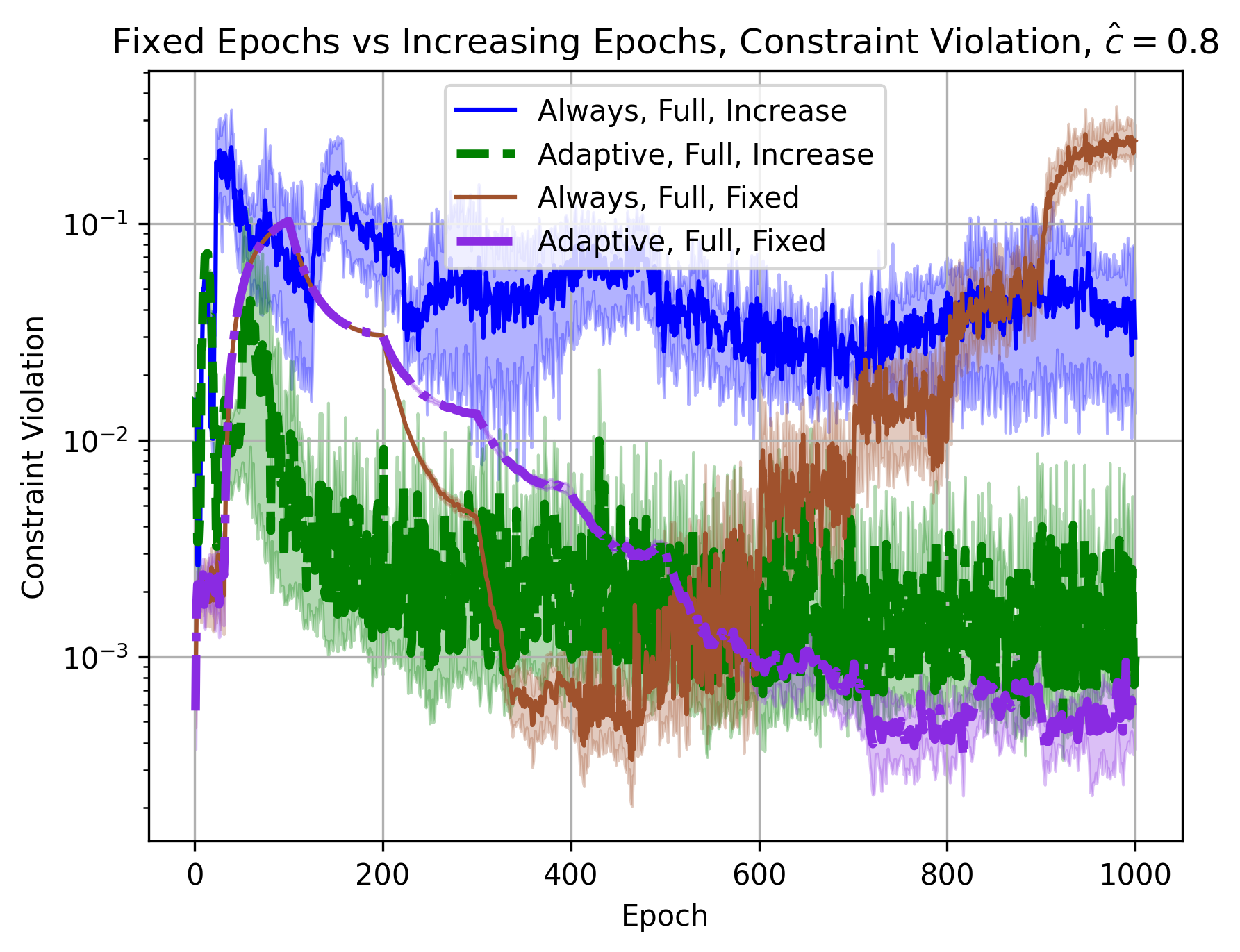}
        \caption{SGD, Adapt vs. Fixed, Constr. Violation, $\hat{c} = 0.8$.}
        \label{fig:sgd_comp_constr_0.8}
    \end{subfigure}

    \vspace{0.4cm} 

    \begin{subfigure}[b]{0.48\textwidth}
        \centering
        \includegraphics[width=\textwidth,height=4.14cm]{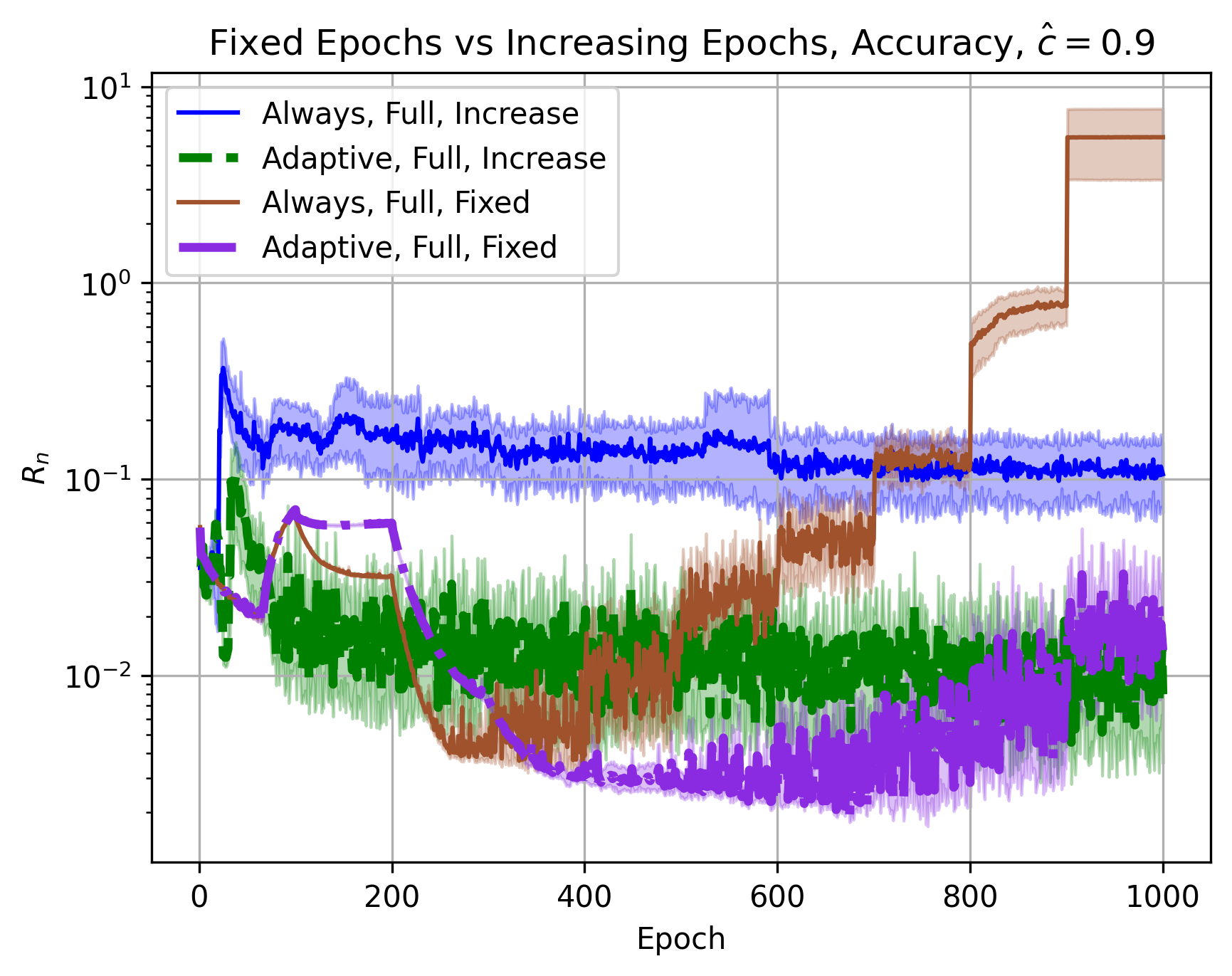}
        \caption{SGD, Adapt vs. Fixed, Accuracy, $\hat{c} = 0.9$.}
        \label{fig:sgd_comp_acc_0.9}
    \end{subfigure}
    \hfill
    \begin{subfigure}[b]{0.48\textwidth}
        \centering
        \includegraphics[width=\textwidth,height=4.14cm]{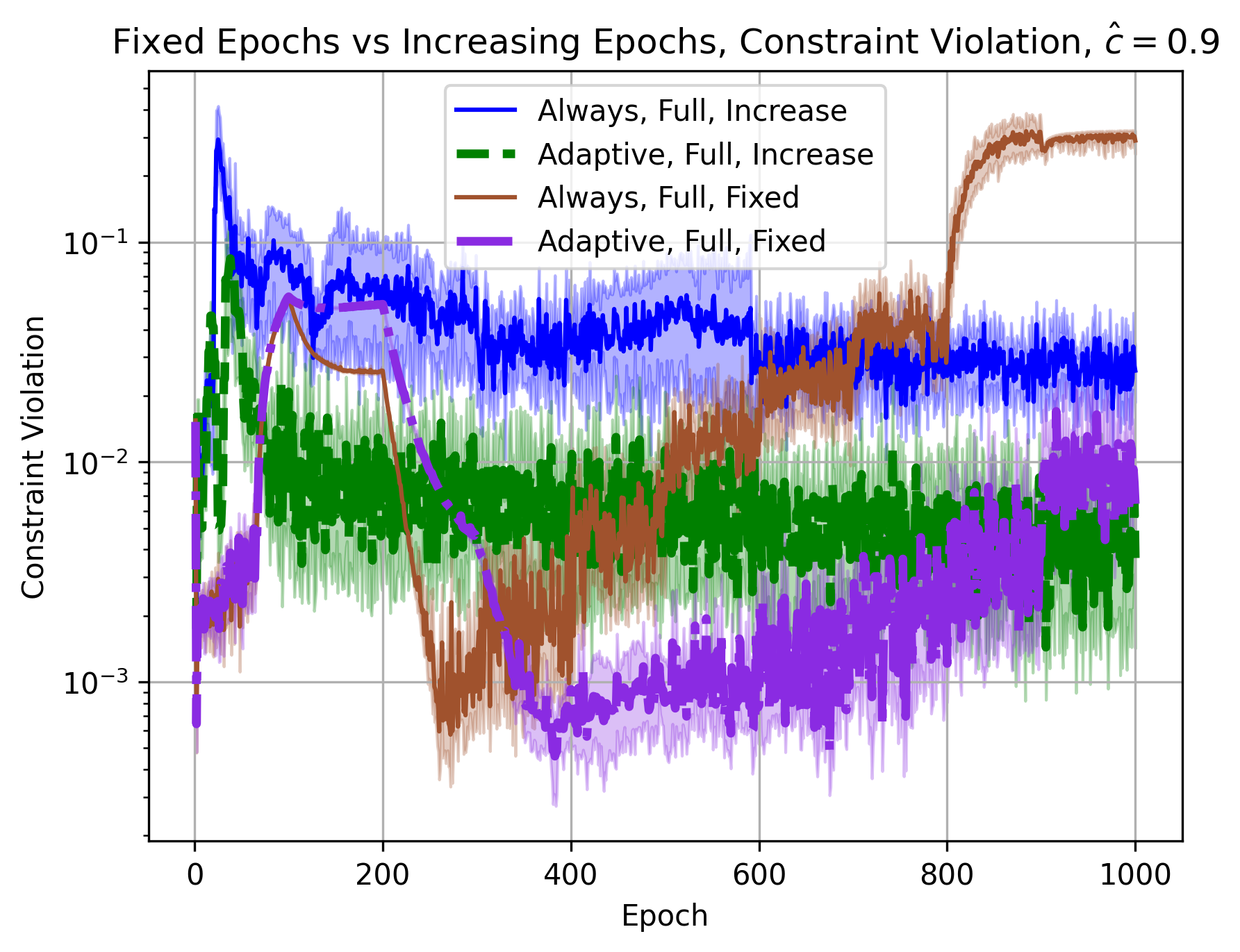}
        \caption{SGD, Adapt vs. Fixed, Constr. Violation, $\hat{c} = 0.9$.}
        \label{fig:modes_adaptive_D}
    \end{subfigure}
    
    \caption{Comparison of adaptive versus fixed inner epochs for SGD varying $\hat{c}$. Average results over 20 trials with 95\% confidence intervals.}
    \label{fig:sgd_comp_constr_0.9}
\end{figure}

The results in Figure~\ref{fig:modes_comparison_all} compare the four outer-loop variants when all methods are implemented using Algorithm \ref{alg:innersolve} and projected stochastic gradient is used to solve the subproblem. At each stochastic gradient iteration, positive and negative class mini-batches are sampled randomly with replacement with mini-batch sizes $128$, while the stepsize is set to $\frac{1}{\sqrt{T_{k,p}}}$. At the start of the algorithm, we choose $T_{0,0} = 1$ and set $r = 2$. In addition, we set $\beta_0 = 0.1$ and $\gamma = 1.5$, values found after some mild parameter tuning.

Across the values of $\hat{c}$, Adaptive-Full method generally produces the smallest combined residual and the smallest constraint violation. Meanwhile, we notice the Always-Full method initially makes progress but exhibits larger oscillations and stabilizes at a higher residual level. This difference is especially visible for the looser constraint thresholds. The comparison indicates that the full dual variable step size is more effective when it is paired with adaptive penalty growth. A full dual variable step size alone does not prevent poor numerical behavior when the penalty parameter is increased unnecessarily.\\

Figure~\ref{fig:sgd_comp_constr_0.9} compares methods implementing Algorithm \ref{alg:innersolve} with a fixed epoch strategy for the Always-Full and Adaptive-Full variants. The parameters for the algorithms using Algorithm \ref{alg:innersolve} (labeled ``increase" in the plots) are the same as the previous experiment, while the fixed epoch strategies solve each subproblem using 100 epochs with a stepsize of 0.1 and an initial penalty parameter of 1. Across the values of $\hat{c}$, Adaptive-Full method using Algorithm \ref{alg:innersolve} exhibits the most consistent convergence behavior. Overall, the method rapidly decreases to a low residual value and maintains relatively stable values of both the optimality residual and the constraint violation. The advantage is particularly clear for $\hat{c} \in \{0.6,0.7\}$, where the method reaches lower residuals than its fixed epoch counterpart and displays less variability throughout the run. In general, the increasing epoch strategy provides more stable behavior across all four constraint thresholds and is less sensitive to the choice of a single fixed epoch count. The benefit of increasing inner epoch is especially clear when compared with Always-Full under a fixed epoch, which exhibits worse convergence trajectory for several values of $\hat{c}$, which, while occasionally achieving a lower residual or constraint violation, has consistently erratic behavior across the experiments. These results suggest that progressively increasing the amount of inner-solver computation improves the reliability of the approximate subproblem solutions. The strategy begins with a relatively small computational epoch and increase inner iterations only when the current point fails the acceptance conditions in Algorithm \ref{alg:innersolve}. It therefore avoids unnecessarily expensive inner solves during the early iterations while providing greater accuracy when later augmented Lagrangian subproblems become more difficult.

Additional stochastic experiments can be found in Appendix \ref{app:pstormexperiments} where PStorm is employed as the inner subproblem solver. The experimental results are qualitatively similar to those for projected stochastic gradient, so we relegate these to the appendix.

\section{Conclusion}

In this work, we proposed a simple, adaptive Augmented Lagrangian method for nonlinear, non-convex optimization problems. We prove complexity results that match, up to logarithmic factors, state-of-the-art bounds for Augmented Lagrangian methods under mild conditions. Empirically, we demonstrate that the use of the classical dual stepsize (with a projection step, to enable complexity analysis) as well as an adaptive penalty parameter scheme far outperforms the approaches used by most Augmented Lagrangian methods with complexity guarantees. In addition, we propose a simple, adaptive inner loop for stochastic Augmented Lagrangian methods and derive high probability complexity results when this inner subproblem solver is employed. Numerically, this adaptive inner solver helps stabilize the convergence behavior of the algorithm and proves consistently effective across the set of experiments.

These results apply exclusively to ``two-loop" Augmented Lagrangian methods, though there is significant interest in ``single-loop" methods, which take simple steps, as opposed to solving a subproblem, at each Augmented Lagrangian iteration. An interesting direction of future research is to explore whether the adaptive strategies proposed here can be extended to single loop Augmented Lagrangian methods.

\clearpage

\bibliographystyle{plain} 
\bibliography{references} 

\appendix

\section{PStorm Experiments} \label{app:pstormexperiments}

In this section, we repeat the stochastic experiments of Section \ref{subsec:stochasticexperiments} using the PSTORM \cite{Li2023} subproblem solver, which incorporates recursive variance reduction. The results of these experiments can be found in Figures \ref{fig:modes_comparison_all_pstorm} and \ref{fig:pstorm_comp}. As noted in Section \ref{subsec:stochasticexperiments}, the results are qualitatively similar to those performed using projected stochastic gradient. In particular, the Adaptive, Full method consistently outperforms the other methods tested, across the different problems tested. In addition, the adaptive inner-solver is significantly more stable than the other methods tested and performs as well or better than the non-adaptive methods in most trials.

\begin{figure}[h]
    \centering

    \begin{subfigure}[b]{0.48\textwidth}
        \centering
        \includegraphics[width=\textwidth,height=4.14cm]{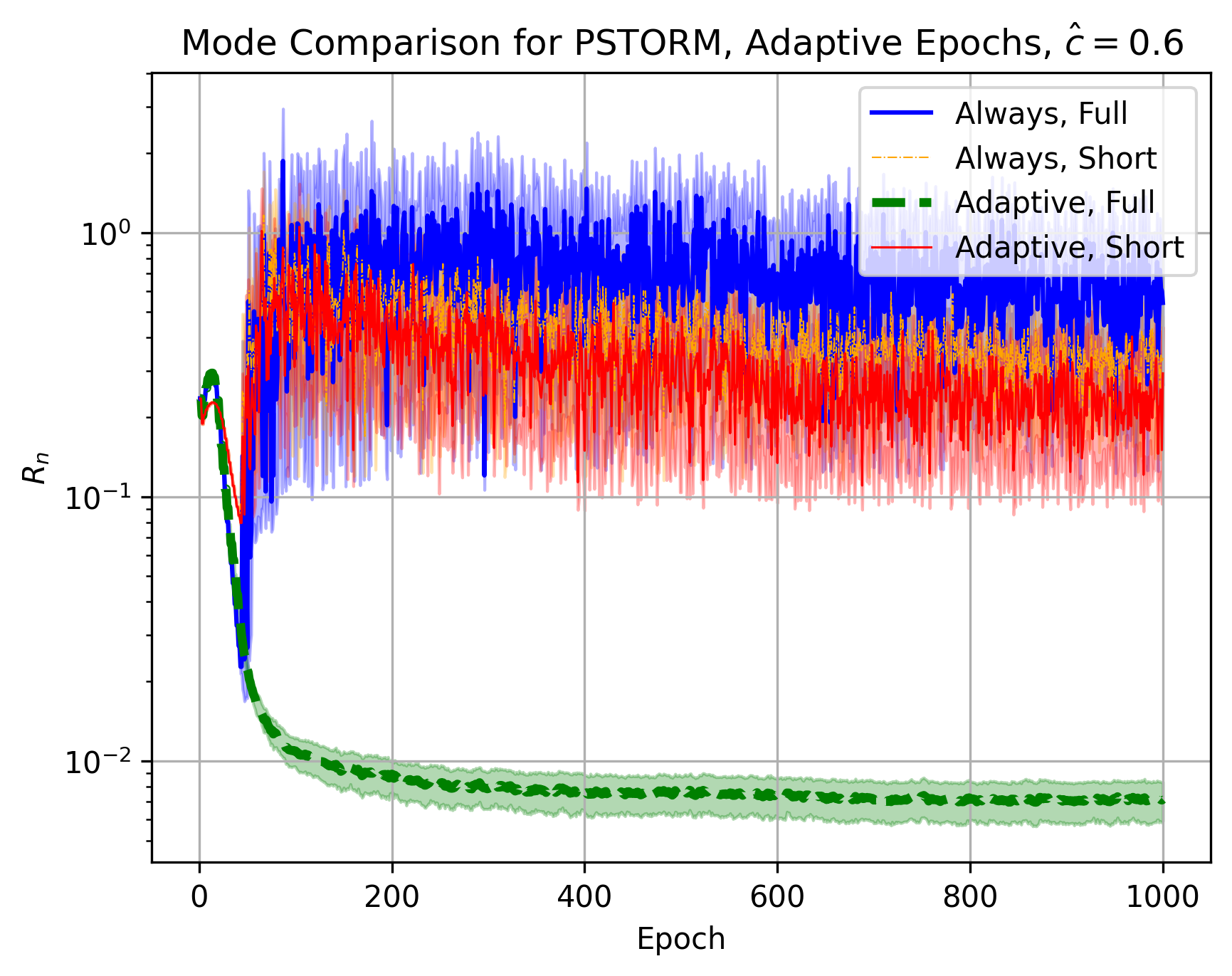}
        \caption{PStorm, Adaptive epochs, Accuracy, $\hat{c} = 0.6$.}
        \label{fig:pstorm_adap_acc_0.6}
    \end{subfigure}
    \hfill
    \begin{subfigure}[b]{0.48\textwidth}
        \centering
        \includegraphics[width=\textwidth,height=4.14cm]{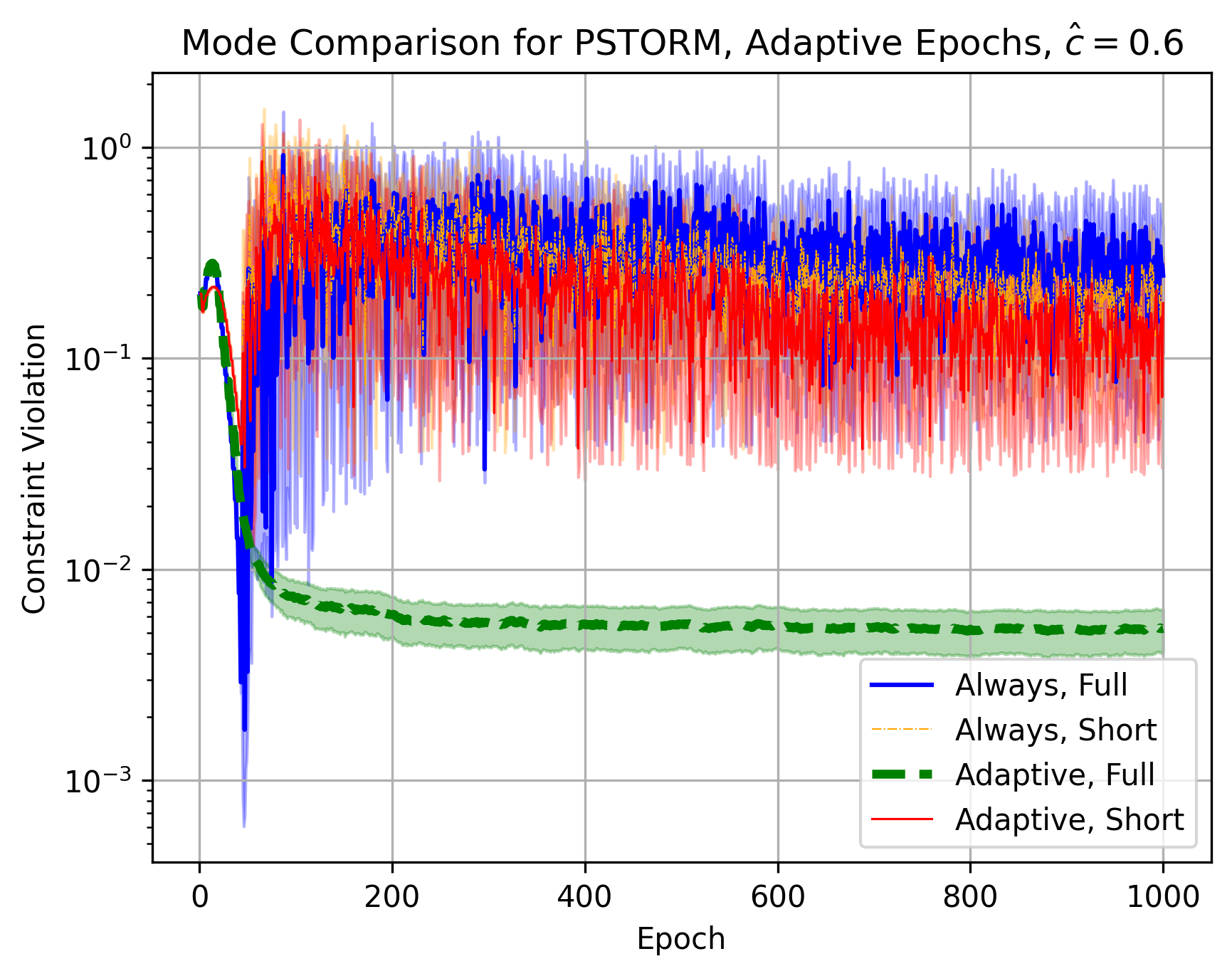}
        \caption{PStorm, Adaptive epochs, Constr. violation, $\hat{c} = 0.6$.}
        \label{fig:pstorm_adap_constr_0.6}
    \end{subfigure}

    \vspace{0.4cm} 

    \begin{subfigure}[b]{0.48\textwidth}
        \centering
        \includegraphics[width=\textwidth,height=4.14cm]{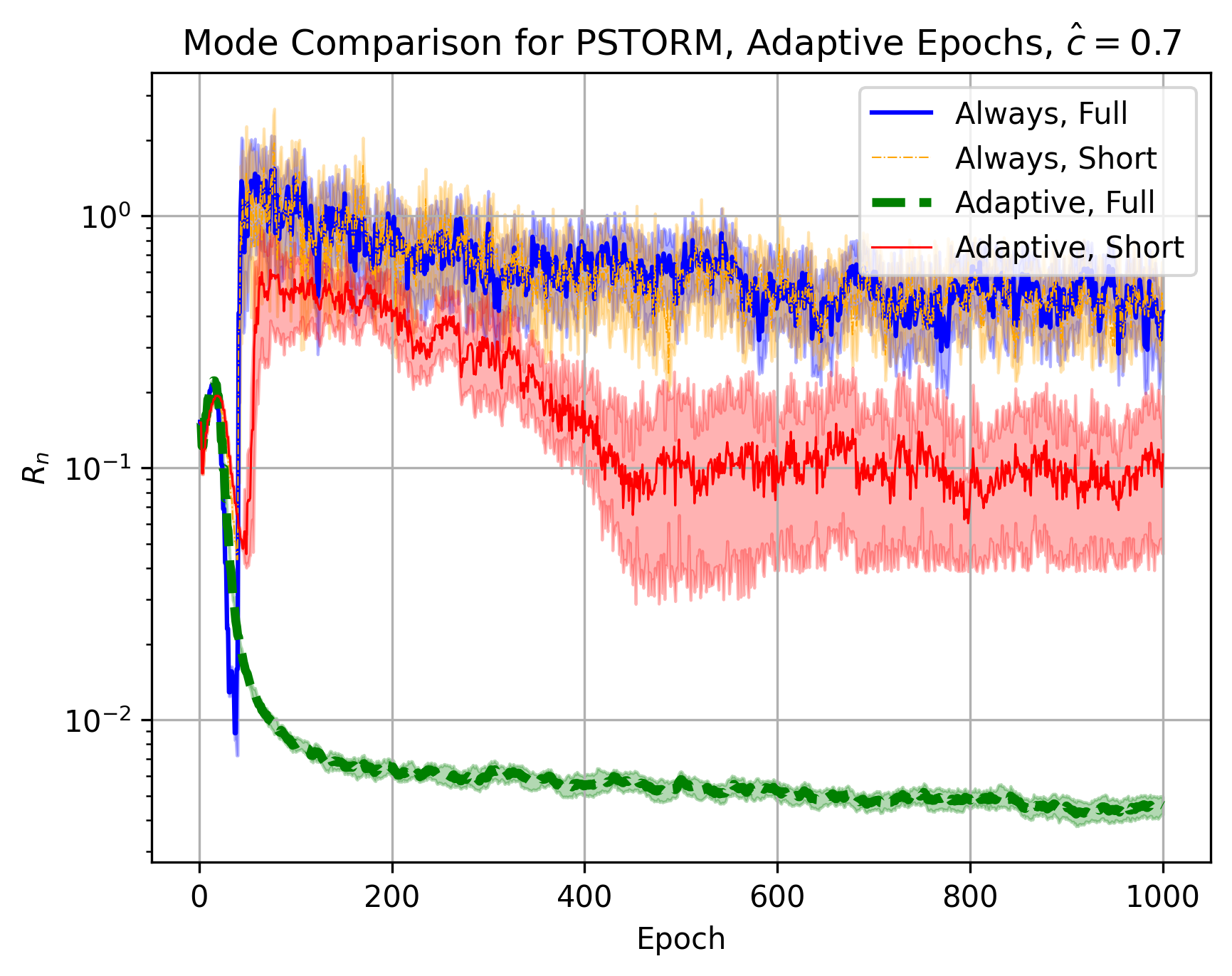}
        \caption{PStorm, Adaptive epochs, Accuracy, $\hat{c} = 0.7$.}
        \label{fig:pstorm_adap_acc_0.7}
    \end{subfigure}
    \hfill
    \begin{subfigure}[b]{0.48\textwidth}
        \centering
        \includegraphics[width=\textwidth,height=4.14cm]{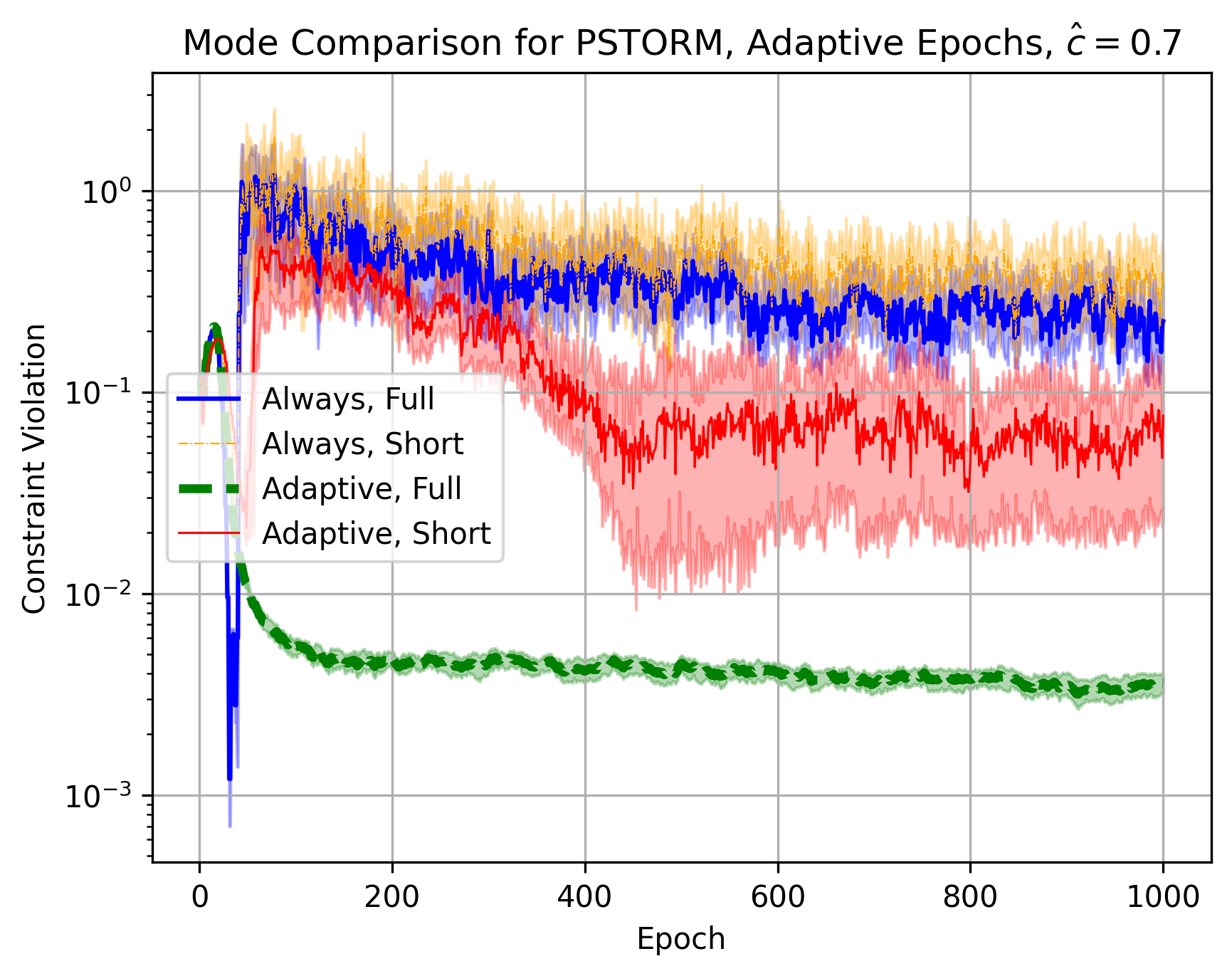}
        \caption{PStorm, Adaptive epochs, Constr. violation, $\hat{c} = 0.7$.}
        \label{fig:pstorm_adap_constr_0.7}
    \end{subfigure}

    \hfill

      \begin{subfigure}[b]{0.48\textwidth}
        \centering
        \includegraphics[width=\textwidth,height=4.14cm]{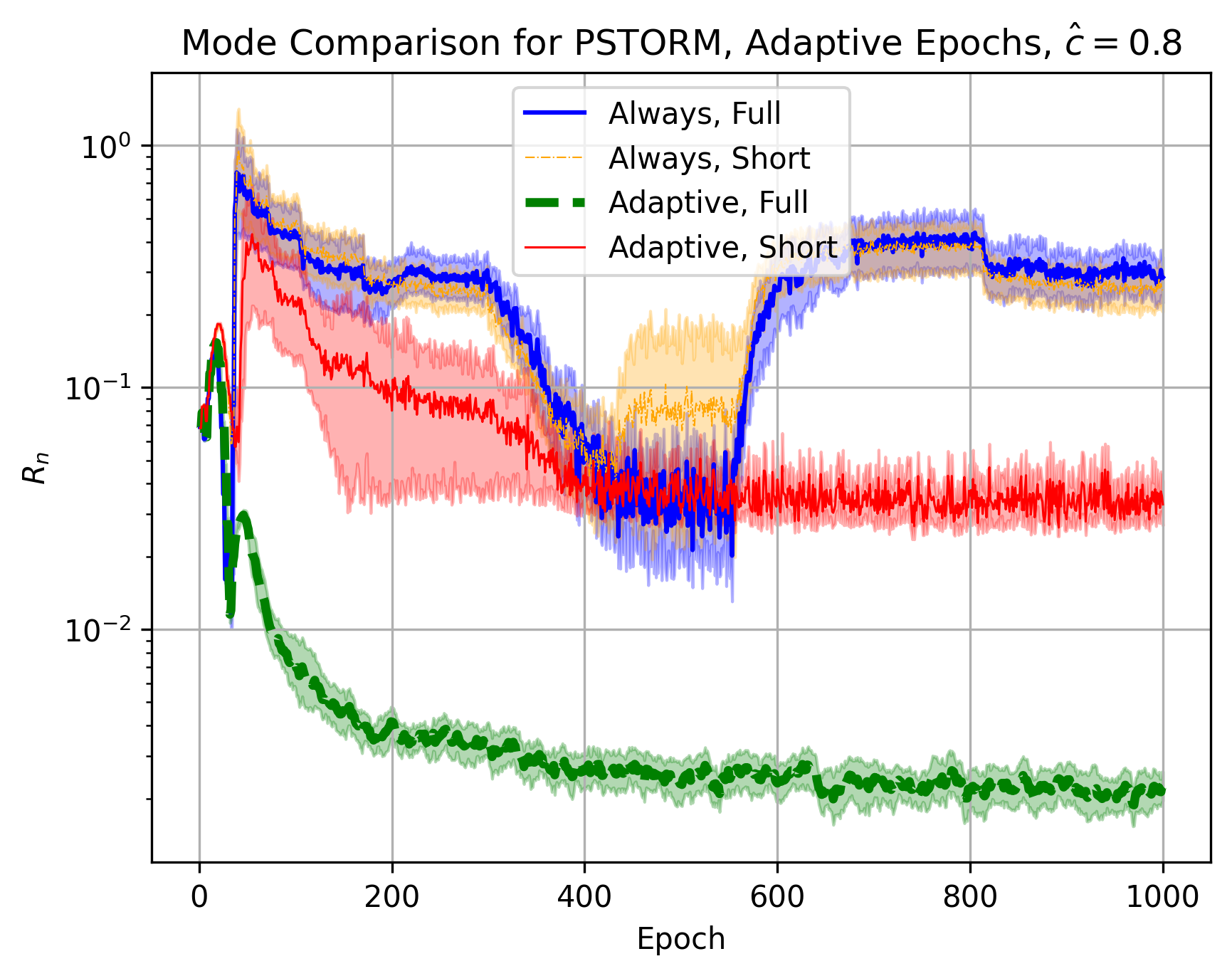}
        \caption{PStorm, Adaptive epochs, Accuracy, $\hat{c} = 0.8$.}
        \label{fig:pstorm_adap_acc_0.8}
    \end{subfigure}
    \hfill
    \begin{subfigure}[b]{0.48\textwidth}
        \centering
        \includegraphics[width=\textwidth,height=4.14cm]{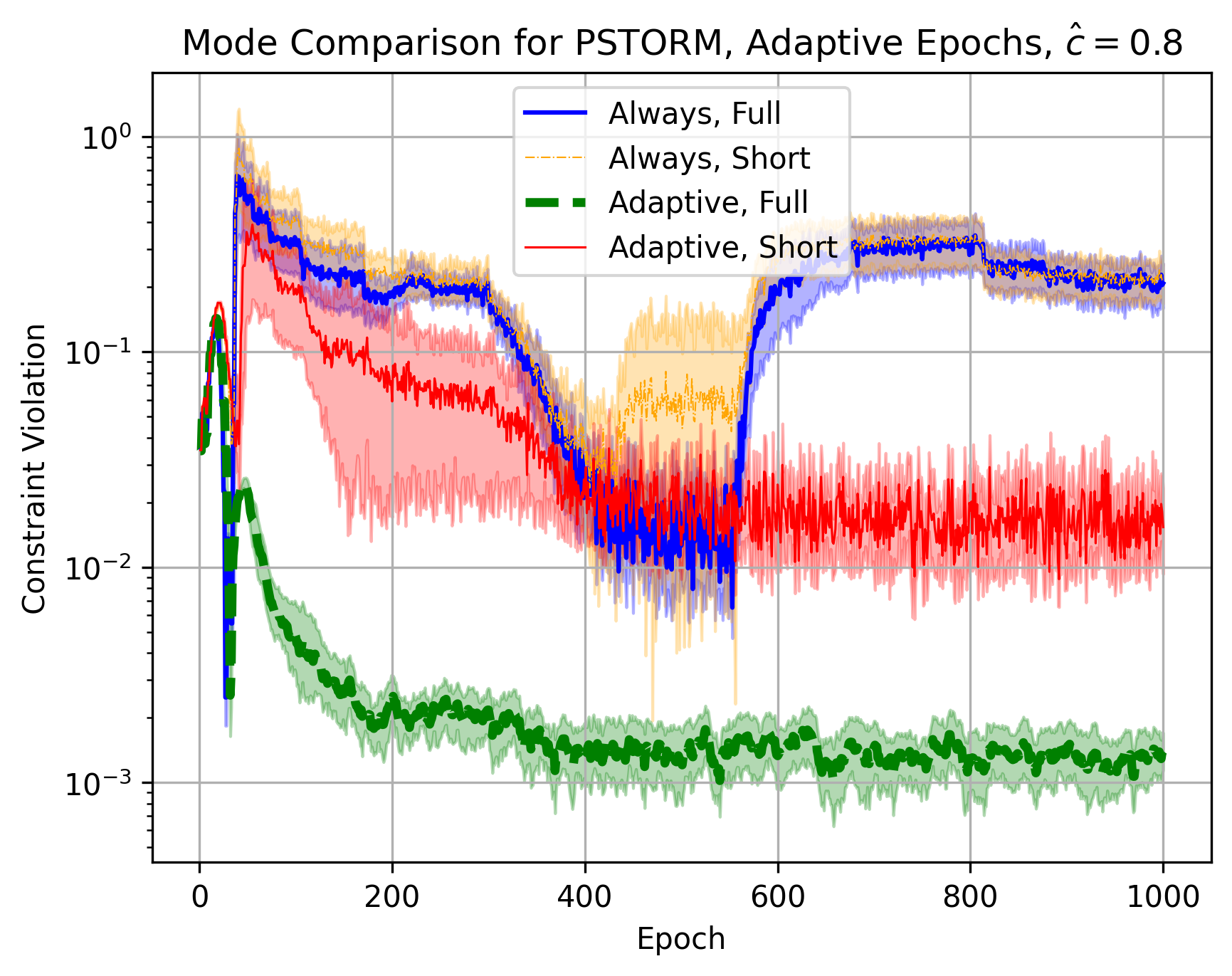}
        \caption{PStorm, Adaptive epochs, Constr. violation, $\hat{c} = 0.8$.}
        \label{fig:pstorm_adap_constr_0.8}
    \end{subfigure}

    \vspace{0.4cm} 

    \begin{subfigure}[b]{0.48\textwidth}
        \centering
        \includegraphics[width=\textwidth,height=4.14cm]{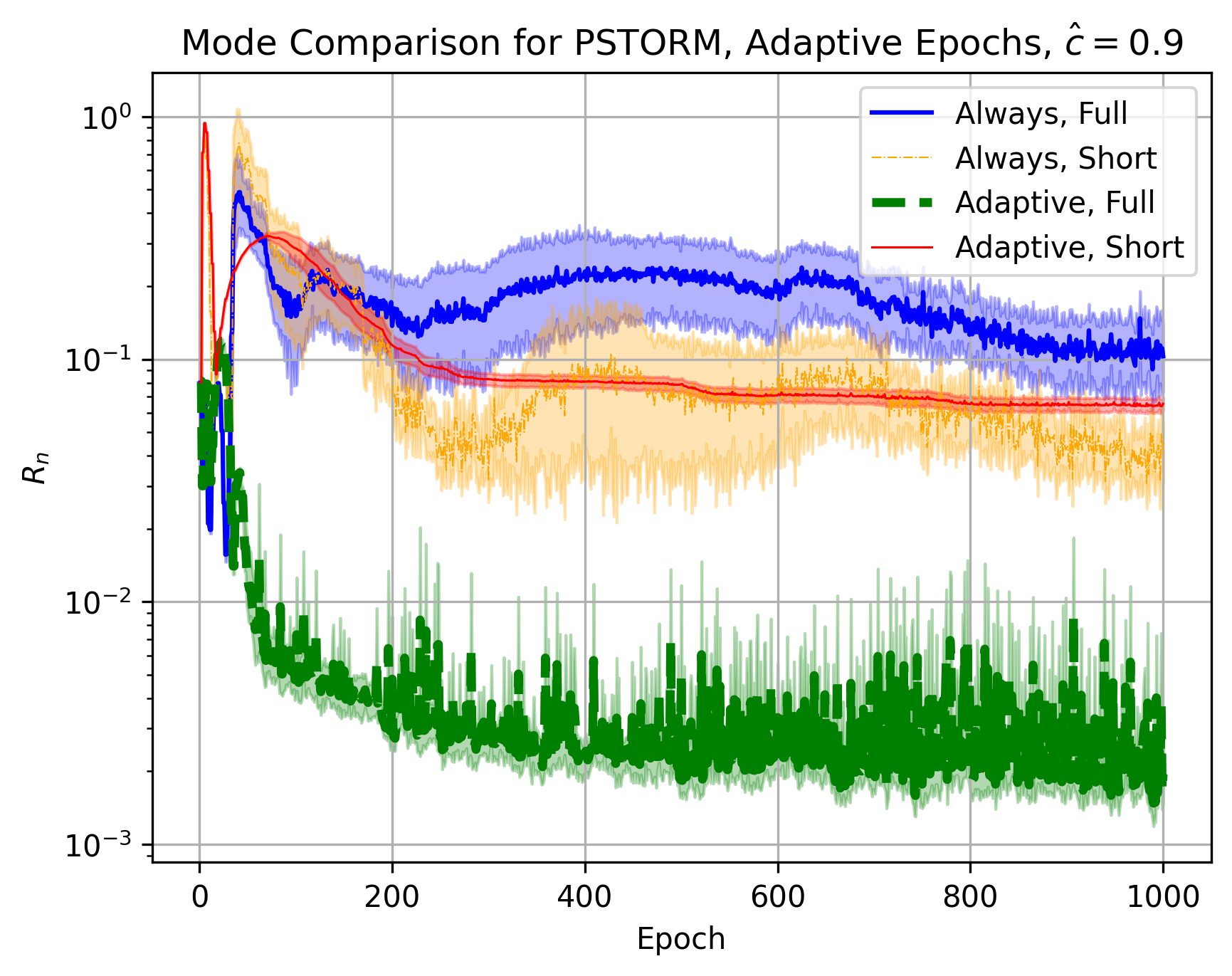}
        \caption{PStorm, Adaptive epochs, Accuracy, $\hat{c} = 0.9$.}
        \label{fig:pstorm_adap_acc_0.9}
    \end{subfigure}
    \hfill
    \begin{subfigure}[b]{0.48\textwidth}
        \centering
        \includegraphics[width=\textwidth,height=4.14cm]{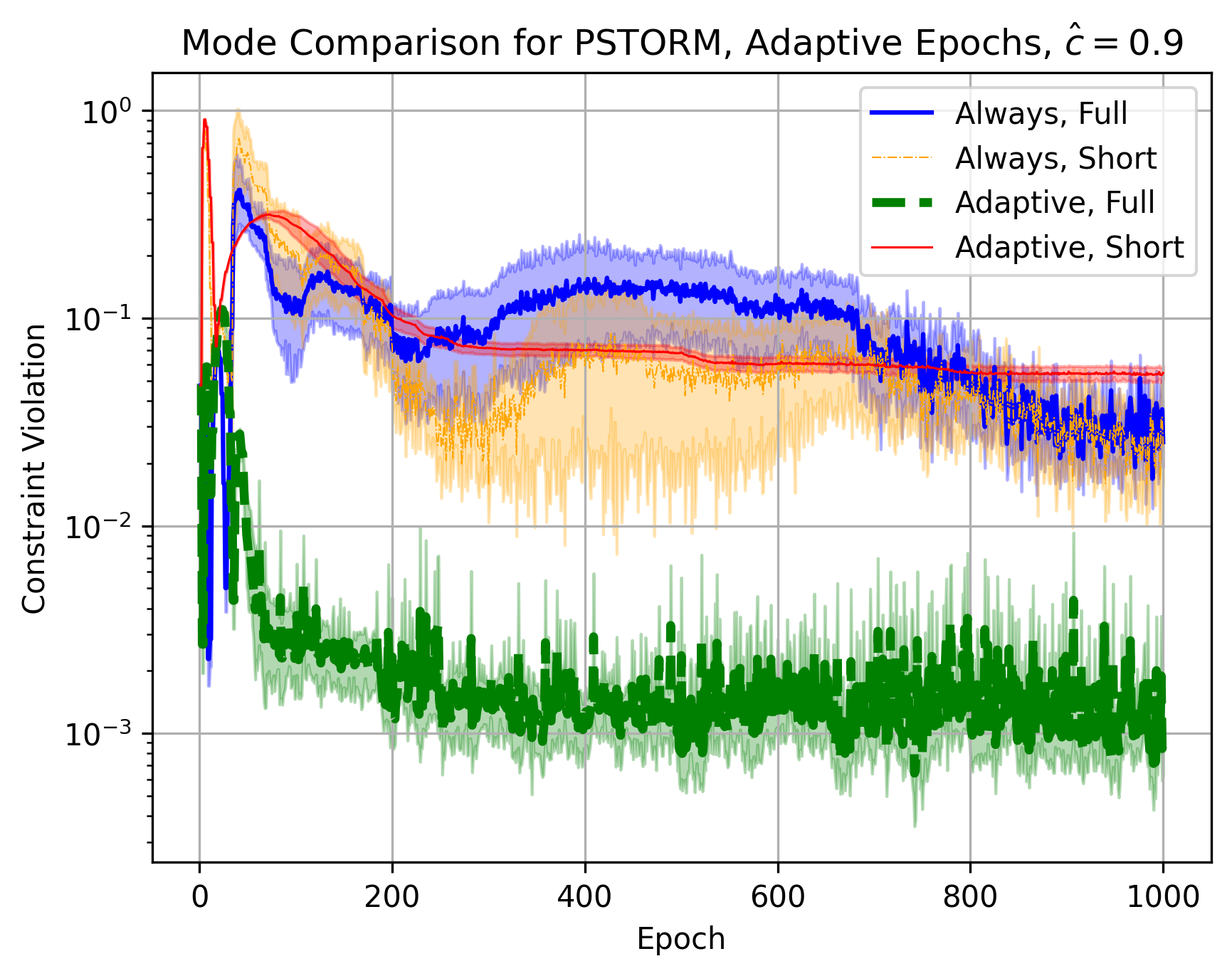}
        \caption{PStorm, Adaptive epochs, Constr. violation, $\hat{c} = 0.9$.}
        \label{fig:pstorm_adap_constr_0.9}
    \end{subfigure}
    
    \caption{Comparison of different outer loop modes for PStorm with adaptive inner epochs varying $\hat{c}$. Average results over 20 trials with 95\% confidence intervals.}
    \label{fig:modes_comparison_all_pstorm}
\end{figure}

\begin{figure}[h]
    \centering

    \begin{subfigure}[b]{0.48\textwidth}
        \centering
        \includegraphics[width=\textwidth,height=4.14cm]{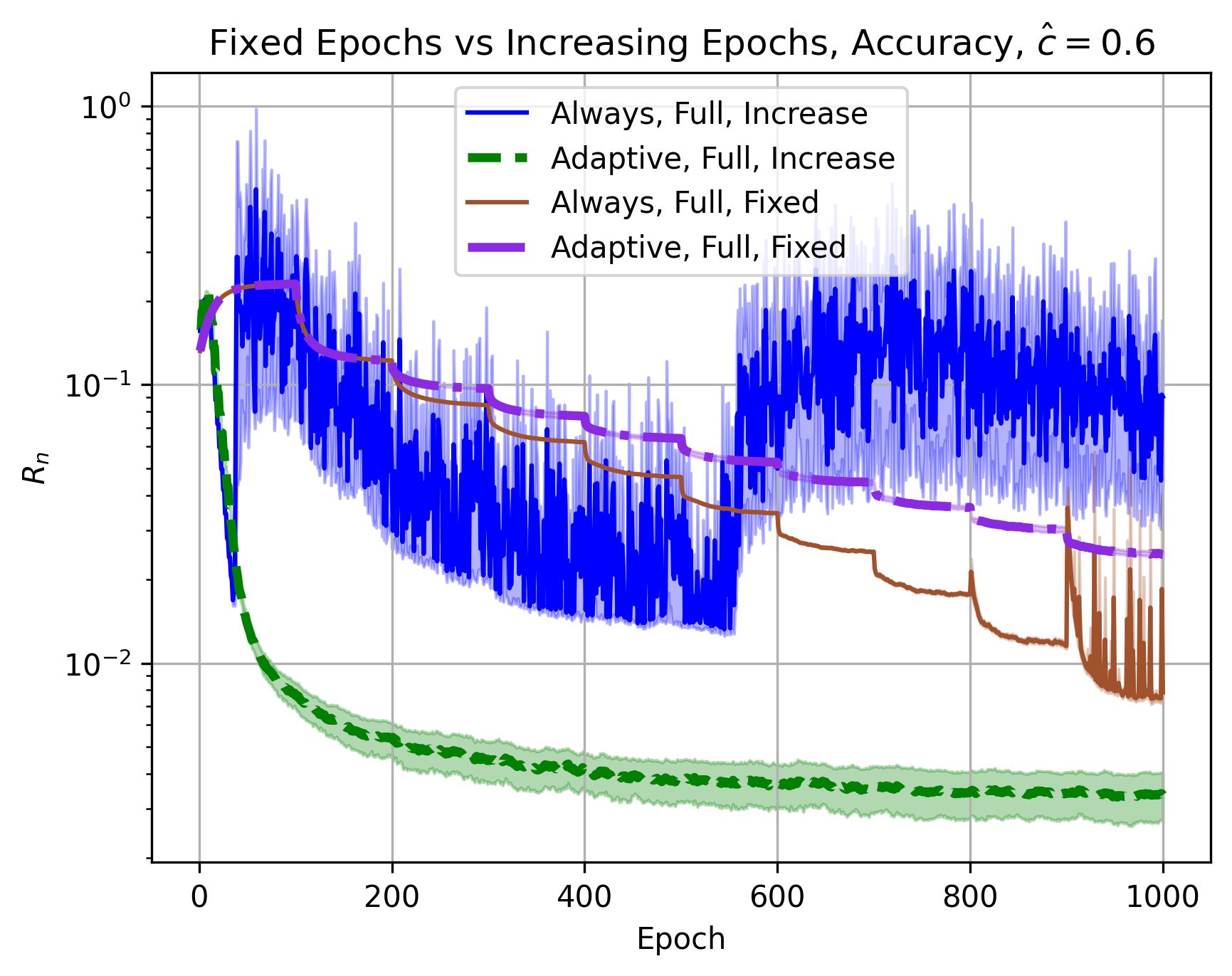}
        \caption{PStorm, Adapt vs. Fixed, Accuracy, $\hat{c} = 0.6$.}
        \label{fig:pstorm_comp_acc_0.6}
    \end{subfigure}
    \hfill
    \begin{subfigure}[b]{0.48\textwidth}
        \centering
        \includegraphics[width=\textwidth,height=4.14cm]{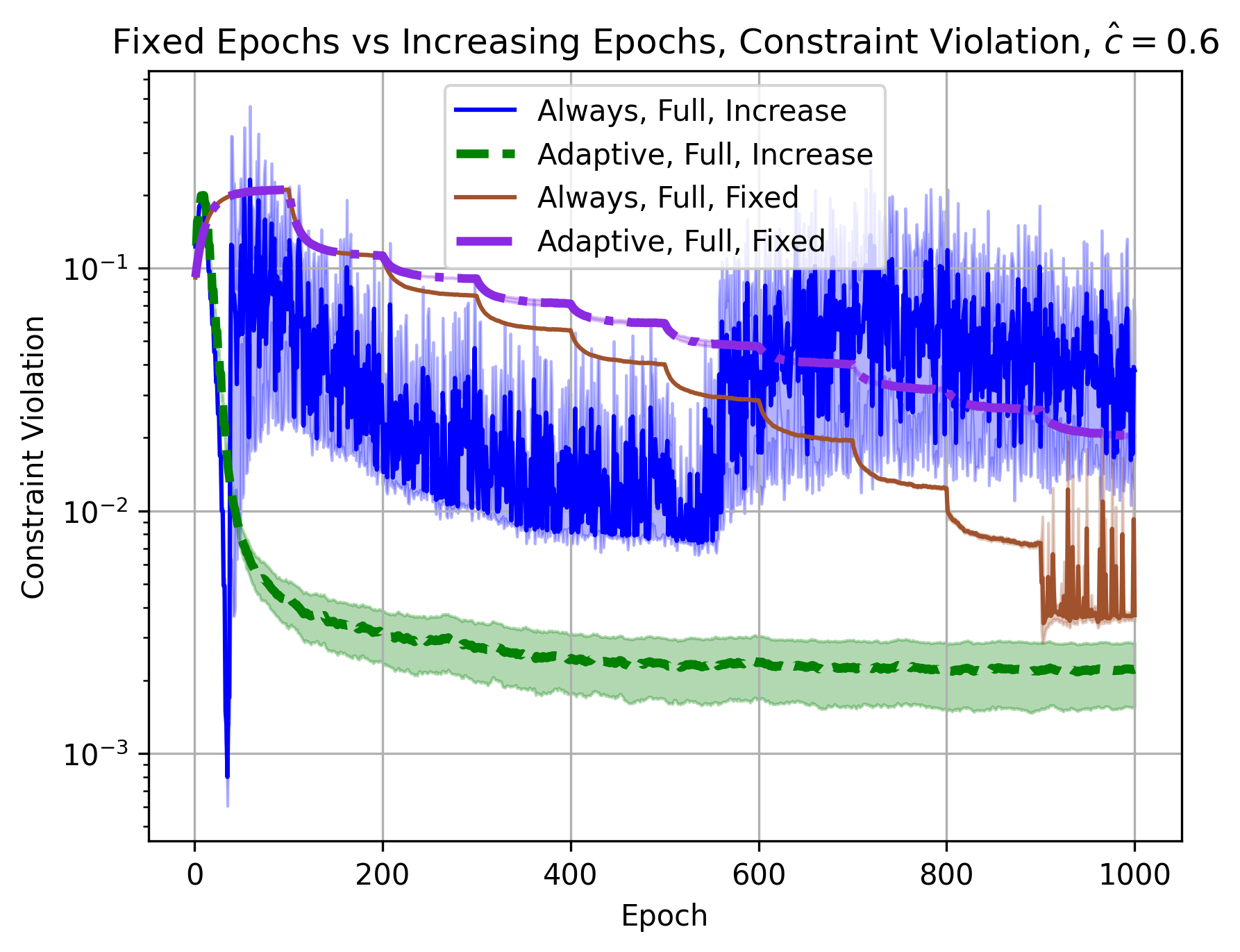}
        \caption{PStorm, Adapt vs. Fixed, Constr. Violation, $\hat{c} = 0.6$.}
        \label{fig:pstorm_comp_constr_0.6}
    \end{subfigure}

    \vspace{0.4cm} 

    \begin{subfigure}[b]{0.48\textwidth}
        \centering
        \includegraphics[width=\textwidth,height=4.14cm]{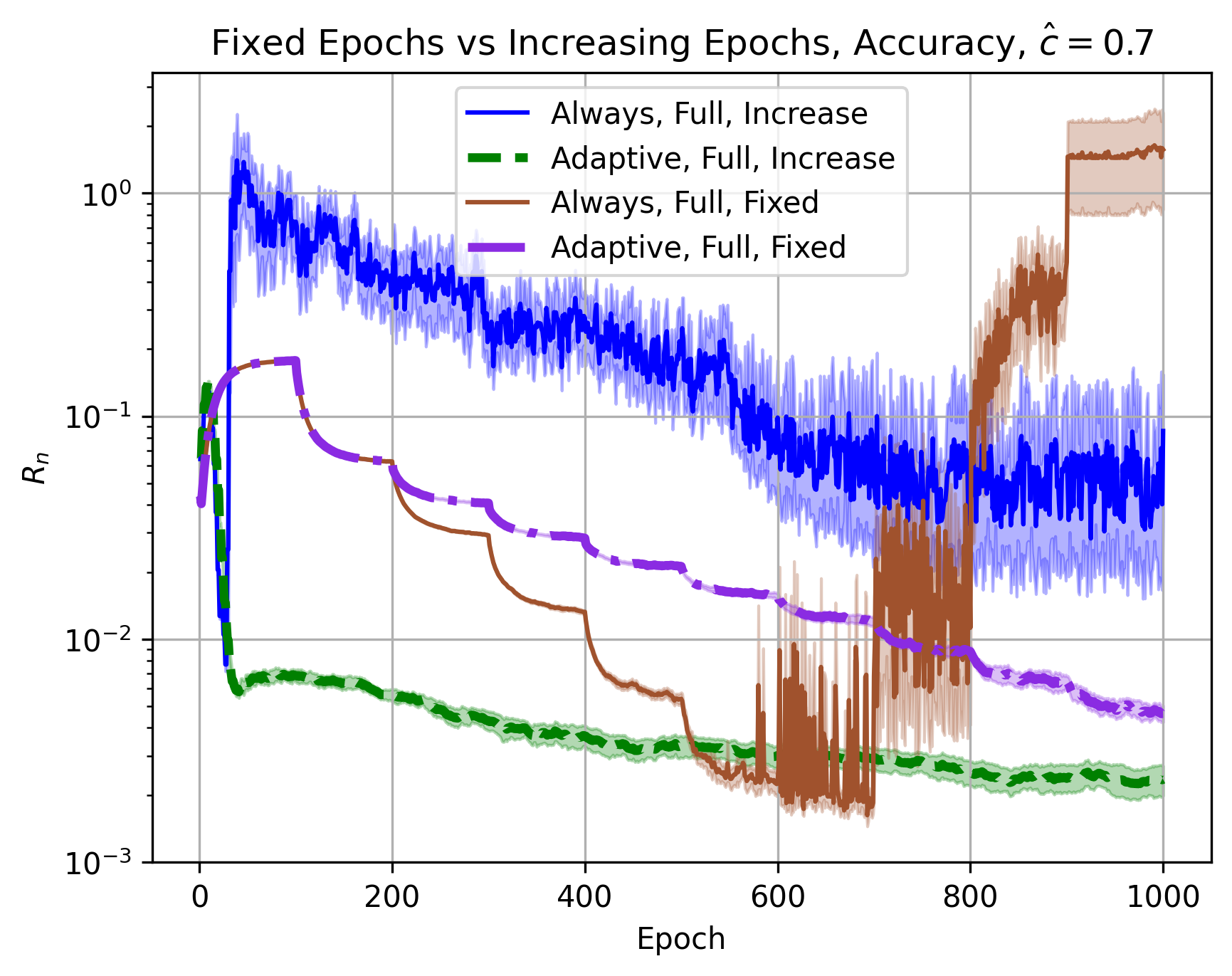}
        \caption{PStorm, Adapt vs. Fixed, Accuracy, $\hat{c} = 0.7$.}
        \label{fig:pstorm_comp_acc_0.7}
    \end{subfigure}
    \hfill
    \begin{subfigure}[b]{0.48\textwidth}
        \centering
        \includegraphics[width=\textwidth,height=4.14cm]{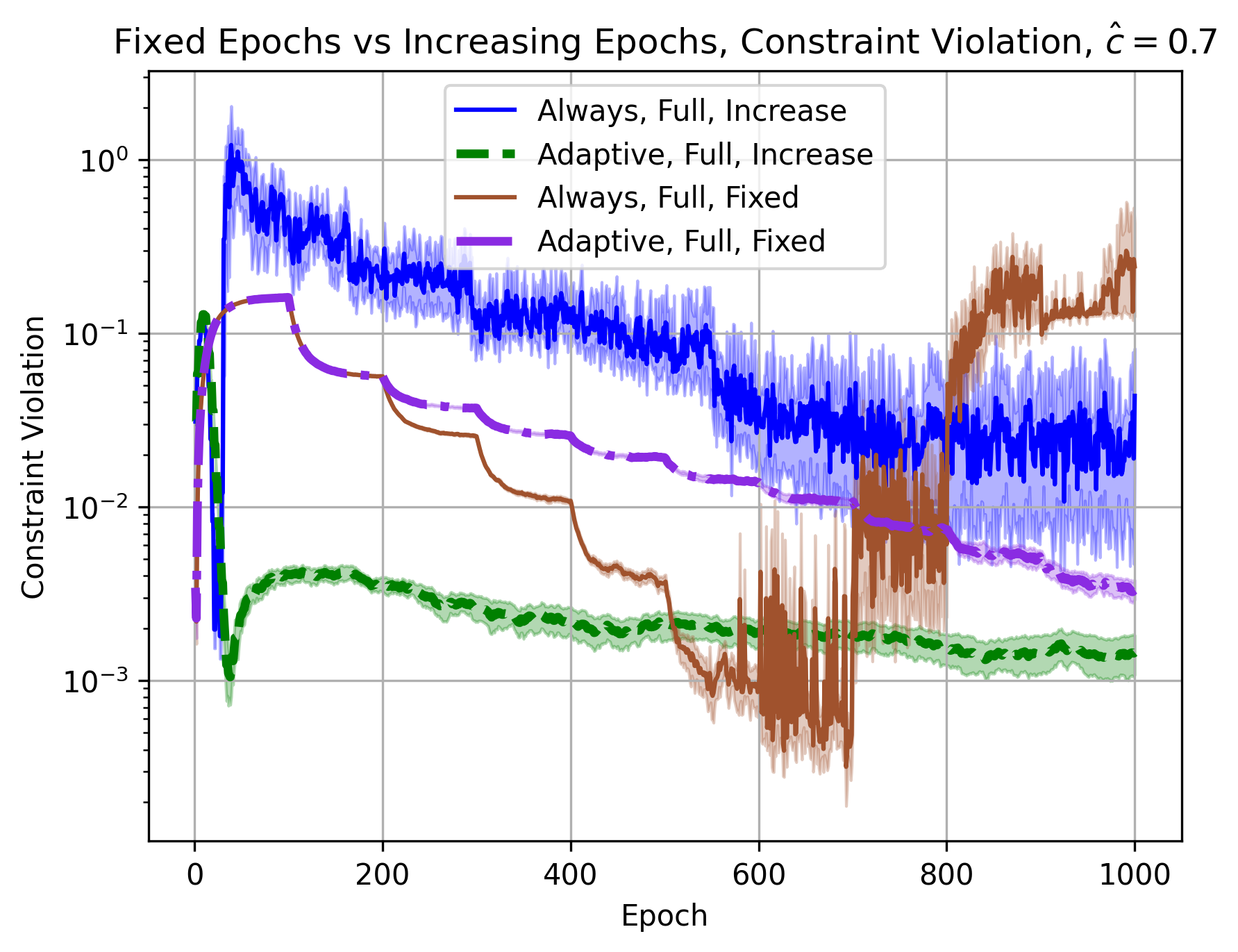}
        \caption{PStorm, Adapt vs. Fixed, Constr. Violation, $\hat{c} = 0.7$.}
        \label{fig:pstorm_comp_constr_0.7}
    \end{subfigure}

    \hfill

      \begin{subfigure}[b]{0.48\textwidth}
        \centering
        \includegraphics[width=\textwidth,height=4.14cm]{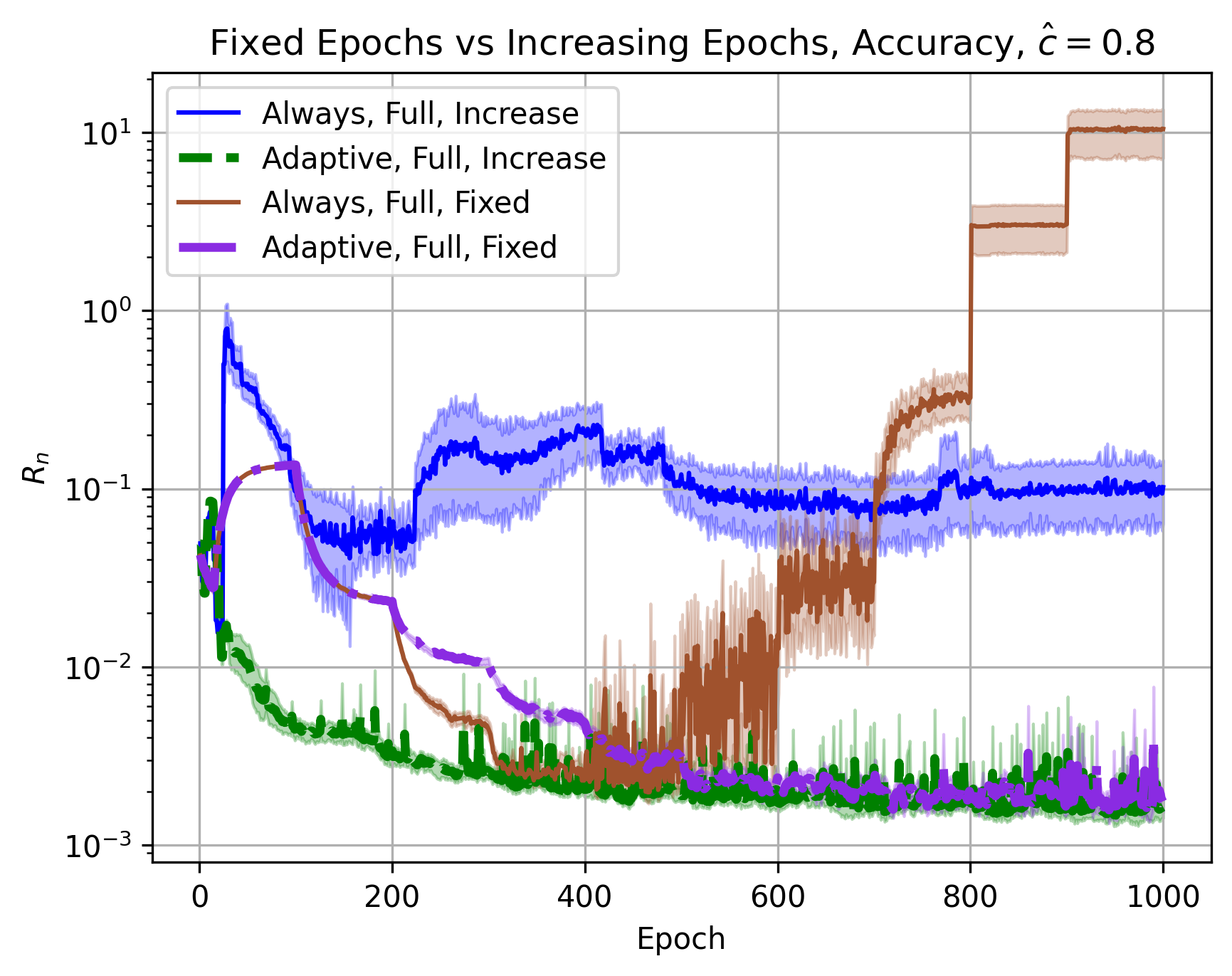}
        \caption{PStorm, Adapt vs. Fixed, Accuracy, $\hat{c} = 0.8$.}
        \label{fig:pstorm_comp_acc_0.8}
    \end{subfigure}
    \hfill
    \begin{subfigure}[b]{0.48\textwidth}
        \centering
        \includegraphics[width=\textwidth,height=4.14cm]{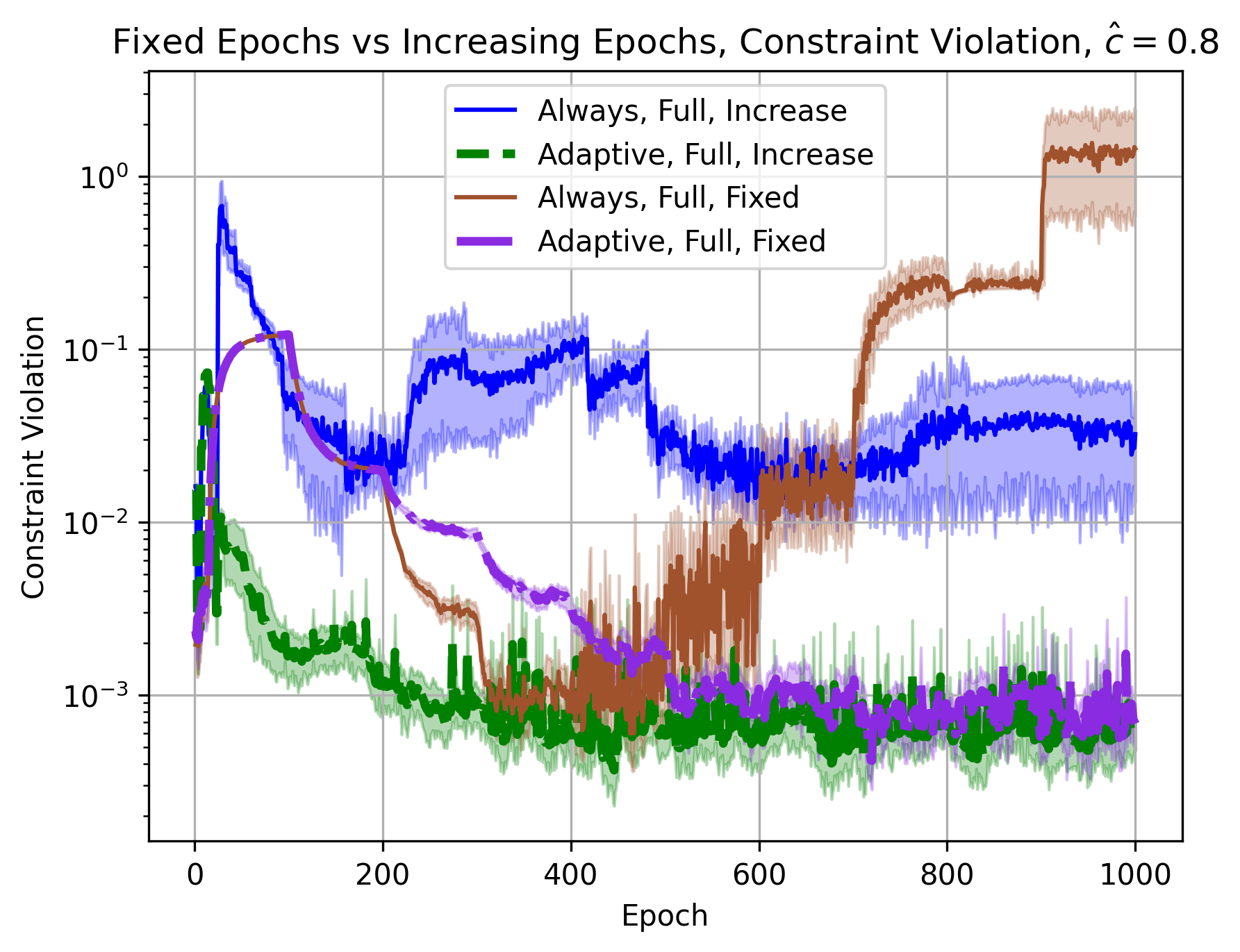}
        \caption{PStorm, Adapt vs. Fixed, Constr. Violation, $\hat{c} = 0.8$.}
        \label{fig:pstorm_comp_constr_0.8}
    \end{subfigure}

    \vspace{0.4cm} 

    \begin{subfigure}[b]{0.48\textwidth}
        \centering
        \includegraphics[width=\textwidth,height=4.14cm]{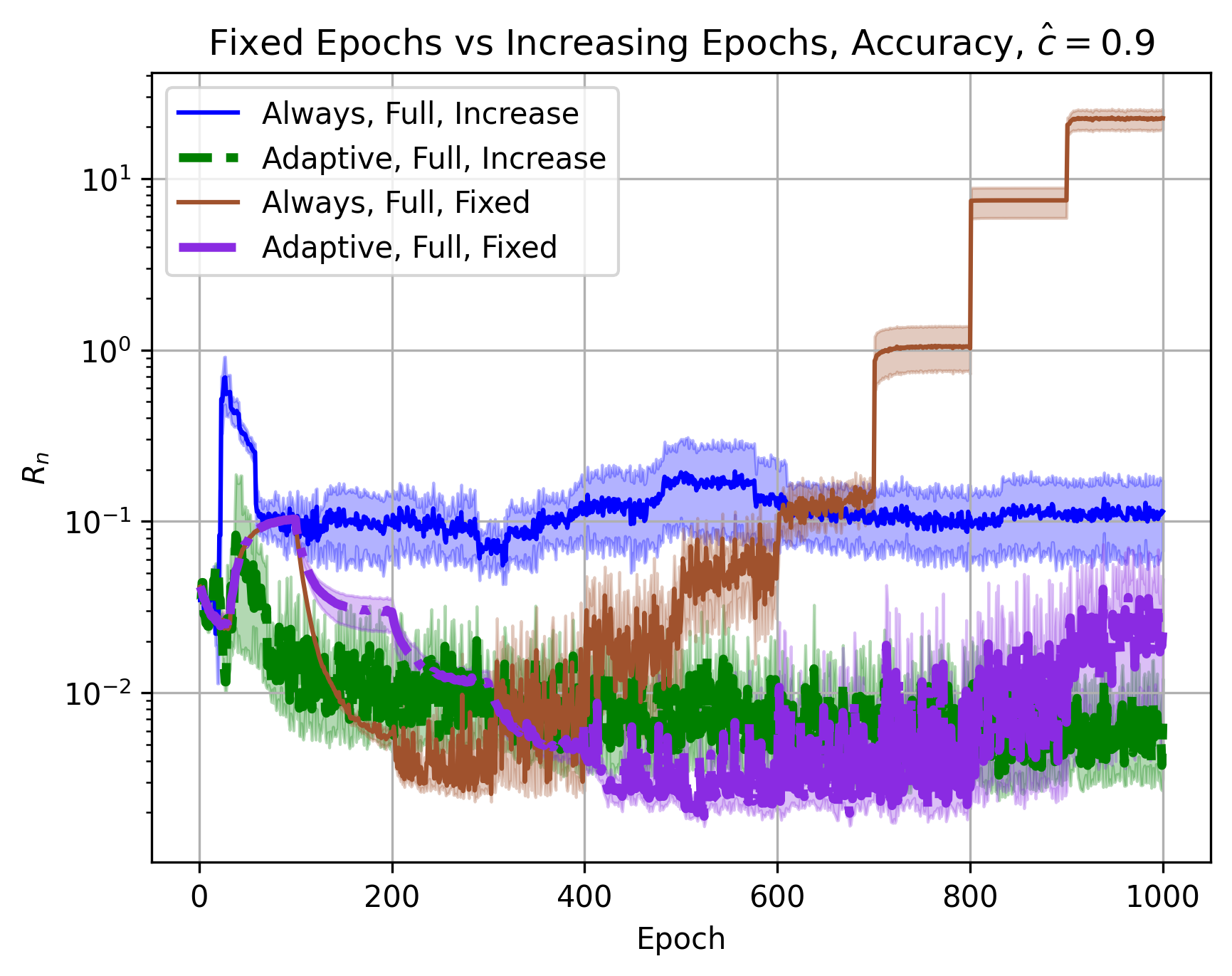}
        \caption{PStorm, Adapt vs. Fixed, Accuracy, $\hat{c} = 0.9$.}
        \label{fig:pstorm_comp_acc_0.9}
    \end{subfigure}
    \hfill
    \begin{subfigure}[b]{0.48\textwidth}
        \centering
        \includegraphics[width=\textwidth,height=4.14cm]{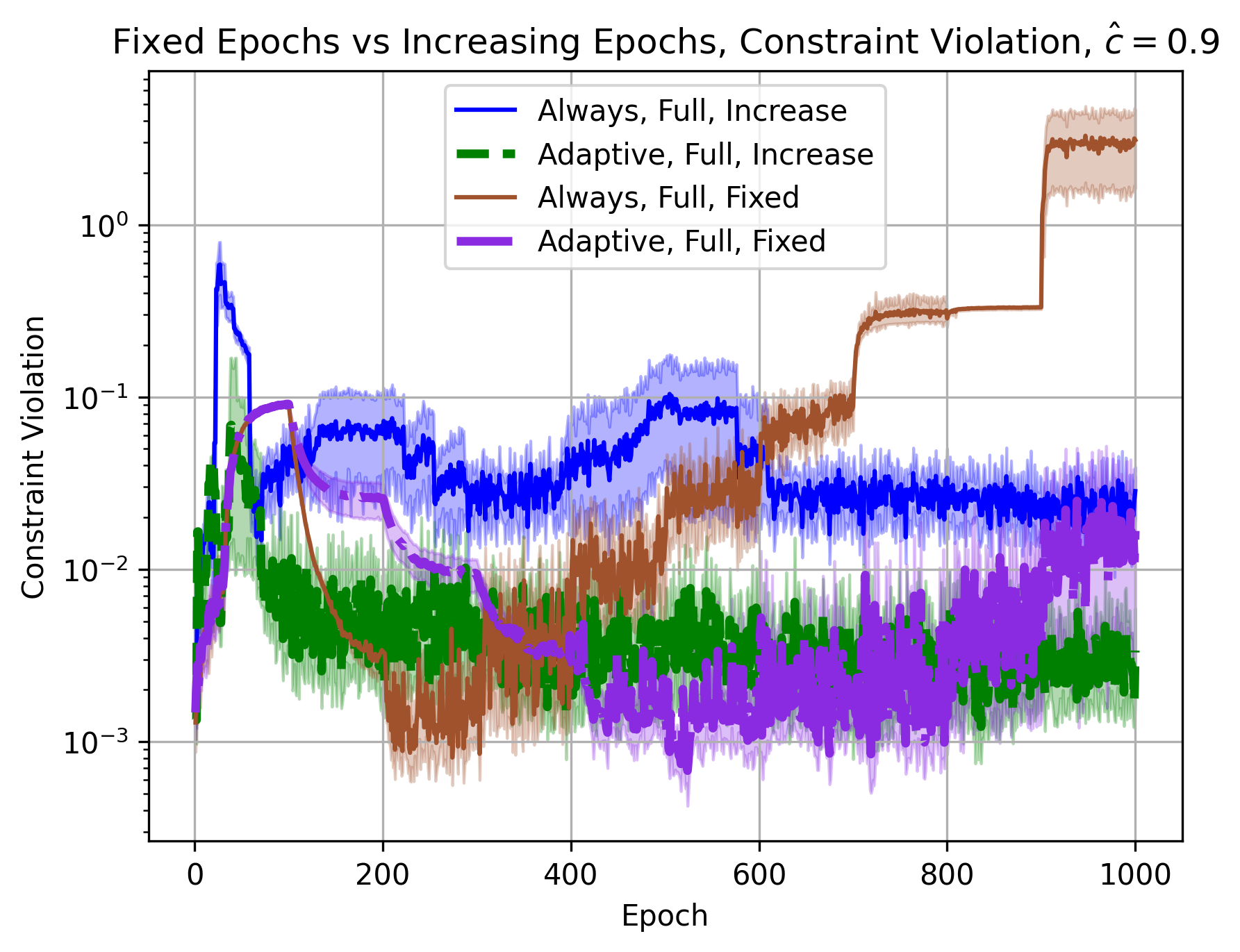}
        \caption{PStorm, Adapt vs. Fixed, Constr. Violation, $\hat{c} = 0.9$.}
        \label{fig:pstorm_comp_constr_0.9}
    \end{subfigure}
    
    \caption{Comparison of adaptive versus fixed inner epochs for PStorm varying $\hat{c}$. Average results over 20 trials with 95\% confidence intervals.}
    \label{fig:pstorm_comp}
\end{figure}

\end{document}